\documentclass[12pt]{article}
\usepackage{amsmath,latexsym,amssymb,amsthm,enumerate,amsmath,amscd}
\usepackage{amsthm,amsmath,amsfonts,amssymb,amscd,latexsym,stmaryrd}
\usepackage[vcentermath,enableskew]{youngtab}
\usepackage[mathscr]{eucal}
\usepackage{color}
\usepackage[usenames,dvipsnames]{xcolor}
\usepackage{colortbl}
\usepackage[all,cmtip]{xy}
\usepackage{graphicx}
\usepackage[utf8]{inputenc}
\usepackage{comment}
\usepackage{hyperref}
\usepackage{tikz}
\usetikzlibrary{calc,shadings,patterns}
\theoremstyle{plain}
\newtheorem{theorem}{Theorem}[section]
\newtheorem{proposition}[theorem]{Proposition}
\newtheorem{corollary}[theorem]{Corollary}
\newtheorem{lemma}[theorem]{Lemma}
\theoremstyle{definition}
\newtheorem{definition}[theorem]{Definition}
\newtheorem{conjecture}[theorem]{Conjecture}

\newtheorem{remark}[theorem]{Remark}
\newtheorem{example}[theorem]{Example}

\newtheorem{question}[theorem]{Question}
\newtheorem*{uclaim}{Claim}

\newtheorem{thm}{Theorem}

\newtheorem{rem}[thm]{Remark}

\newtheorem*{remark*}{Remark}

\newtheorem*{ack}{Acknowledgment}

\numberwithin{equation}{section}
\numberwithin{table}{section}
\newcommand{\Gor}{\rm{Gor}}

\definecolor{purple}{rgb}{0.4,0.2,0.4}

\DeclareMathOperator{\rk}{\mathrm{rk}}

\def\Gr{\mathrm{Gr}}

\def\cha{\mathrm{char}\ }

\def\Mat{\mathrm{Mat}}

\def\Hom{\mathrm{Hom}}

\def\im{\mathrm{im}}
\DeclareMathOperator{\Ann}{Ann}

\def\Hilb{\mathrm{Hilb}}

\def\Soc{\mathrm{Soc}}

\def\<{\left<}
\def\>{\right>}
\def\Gl{\mathrm{Gl}}
\def\P{\mathcal{P}}
\def\A{\mathcal{A}}
\def\B{\mathcal{B}}

\def\F{\mathsf{k}}
\def\G{\mathrm{B}}

\def\Spec{\mathrm{Spec}}

\def\ini{\mathrm{in}}
\def\gin{\mathrm{gin}}

\def\m{\mathfrak{m}}

\def\Soc{\mathrm{Soc}}

\def\ns{\footnotesize \it}

\def\Aut{\mathrm{Aut}}

\def\w{\mathsf{w}}

\def\max{\mathrm{max}}

\newcommand{\maxA}{\ensuremath{{\mathfrak{m}_{\mathcal{A}}}}}

\definecolor{light-gray}{gray}{0.9}
\definecolor{med-gray}{gray}{0.5}
\definecolor{gray1}{gray}{0.87}
\definecolor{gray2}{gray}{0.74}
\definecolor{gray3}{gray}{0.64}
\definecolor{gray4}{gray}{0.48}
\definecolor{verylight-yellow}{rgb}{1,1,0.7}
\definecolor{yellow}{rgb}{1,1,0.2}
\definecolor{vivid-blue}{rgb}{0.2,0,1}
\definecolor{light-pink}{rgb}{1,0.8,1}
\definecolor{med-pink}{rgb}{1,0.6,1}
\definecolor{aqua}{rgb}{0.0, 1.0, 1.0}
\definecolor{brightgreen}{rgb}{0.4, 1.0, 0.0}
\definecolor{light-gray}{rgb}{0.5, 0.9, 0.5}
\def\cha{\mathrm{char}\ }
\title{Artinian algebras and Jordan type
\footnote{\textbf{Keywords}: Artinian algebra, Hilbert function, Jordan type, Lefschetz property, tensor product.\textbf{2010 Mathematics Subject Classification}: Primary: 13E10; Secondary: 13D40, 13H10, 14B07, 14C05.\quad Email addresses: {\sf a.iarrobino@northeastern.edu, pmm@uevora.pt, cmcdanie@endicott.edu}}}
\author{Anthony Iarrobino\\[.05in]
{\ns Department of Mathematics, Northeastern University, Boston, MA 02115,
USA.
}\\[.2in] Pedro Macias Marques\\[.05in]
{\ns Departamento de Matem\'{a}tica, Escola de Ci\^{e}ncias e Tecnologia, Centro de Investiga\c{c}\~{a}o}\\[-.05in]
{\ns em Matem\'{a}tica e Aplica\c{c}\~{o}es, Instituto de Investiga\c{c}\~{a}o e Forma\c{c}\~{a}o Avan\c{c}ada,}\\[-.05in]
{\ns Universidade de \'{E}vora, Rua Rom\~{a}o Ramalho, 59, P--7000--671 \'{E}vora, Portugal.}
\\[.2in]
Chris McDaniel\\[0.05in]
{\ns Endicott College, 376 Hale St
Beverly, MA 01915, USA.}
}
\date{August 31, 2022}
\begin{document}
\maketitle
\begin{abstract}
There has been much work on strong and weak Lefschetz conditions for graded Artinian algebras $A$, especially those that are Artinian Gorenstein. A more general invariant of an Artinian algebra $A$ or finite $A$-module $M$ that we consider here is the set of Jordan types of elements of the maximal ideal ${\mathfrak{m}}$ of $A$, acting on $M$. Here, the Jordan type of $\ell\in {\mathfrak{m}_A}$ is the partition giving the Jordan blocks of the multiplication map $m_\ell:M\to M$. In particular, we consider the Jordan type of a generic linear element $\ell$ in $A_1$, or in the case of a local ring $\mathcal{A}$, that of a generic element $\ell\in {\mathfrak{m}}_{\mathcal{A}}$, the maximum ideal. \par
We often take $M=A$, the graded algebra, or $M=\A$ a local algebra. The strong Lefschetz property of an element, as well as the weak Lefschetz property can be expressed simply in terms of its Jordan type and the Hilbert function of $M$. However, there has not been until recently a systematic study of the set of possible Jordan types for a given Artinian algebra $A$ or $A$-module $M$, except, importantly, in modular invariant theory, or in the study of commuting Jordan types.\par
We first show some basic properties of the Jordan type. In a main result we show an inequality between the Jordan type of $\ell\in\mathfrak{m}_{\A}$ and a certain local Hilbert function. In our last sections we give an overview of topics such as the Jordan types for Nagata idealizations, for modular tensor products, and for free extensions, including examples and some new results. We as well propose open problems.
\end{abstract}
\tableofcontents
\section{Introduction.}\label{introsec}\par
The strong and weak Lefschetz properties of graded Artinian algebras $A$
have been extensively studied, especially in the last twenty years. These properties are determined by the ranks of the multiplication operators ${m_{\ell^i}: A\to A}$ by powers $\ell^i$ of generic linear elements of $\ell\in A$. Our first main goal is to generalize these properties in several ways. First, given any linear element $\ell$ in $A_1$, we will consider the set of all such ranks of $m_{\ell^i}$, which determine a \emph{Jordan type} $P_\ell$, which is a partition of $n=\dim_\F A$. 
The partition $P_\ell$ compared to the Hilbert function of $A$ determines whether the pair $(\ell,A)$ is strong Lefschetz, or weak Lefschetz; however, in general there are many other possible Jordan types, and we explore these possibilities. Second, we will consider not only graded (standard or non-standard) Artinian algebras $A$, but also extend the notions of Lefschetz properties and Jordan type to local Artinian algebras $\A$ with maximal ideal $\m_\A$: for these we will consider the Jordan type for either arbitrary or generic elements $\ell\in \m_\A$. \vskip 0.2cm
Our second goal is to give a systematic account of the basic facts about Jordan type for multiplication maps in Artinian local algebras, or modules over them. We state in broad terms what is known, prove new results, and give many examples. One of our main new results, Theorem~\ref{thm:JtIneq}, establishes inequalities between the Jordan type $P_{\ell,M}$ of any element $\ell\in \m_\A$, acting by the map $m_{\ell}: M\to M$ on an $\A$-module $M$ and its Hilbert function $H(M)$; if $\A$ admits a weight function $\w$, i.e.\ if it is graded, then the inequalities are between $P_{\ell,M}$ and its $\w$-Hilbert function $H_{\w}(M)$, where $\ell$ is $\w$-homogeneous. We explore subtleties of the definition of Jordan type and its connection with Lefschetz properties in Section \ref{lefsec}.  We also introduce a refinement, the Jordan degree type, which specifies not only the Jordan type, but the initial degrees of the strings -  the maximal cyclic $\mathsf{k}[\ell]$ submodules of $A$ under the action of $m_\ell$ into which we may decompose $A$, in Section \ref{Jdegsec}. For a graded Artinian Gorenstein algebra $A$ the Jordan degree type is symmetric: this greatly restricts the possible Jordan types for $A$. This invariant has been in effect studied by T. Harima and J. Watanabe \cite{HW1.5}, and by B. Costa and R. Gondim \cite{CGo}, the latter introducing a colorful notion of string diagrams.
\vskip 0.2cm
Our third main goal, in Section \ref{examplessec}, is to outline several other approaches to the study of Jordan type, adding our own comments and results. 
For example, there have been studies in representation theory of modules having a constant Jordan type, and of the generic Jordan type of a module, and, as well, of the connection of Jordan type loci to vector bundles \cite{CFP,FPS,Be,Pev1,P}.  In addition there has been progress in the study of the modular case of tensor products of Jordan types, generalizing the Clebsch-Gordan formula in characteristic zero (\S \ref{tensorprodsec}, Remark \ref{modularTPrem}).  These studies have been by several groups, some apparently unaware of the related work of others: we give an overview in Section \ref{tensorprodsec}. \par 
Free extensions of an Artinian algebra $A$ with fiber $B$ were introduced by T. Harima and J. Watanabe \cite{HW0}; in \cite[Theorem 2.1]{IMM} we showed that, geometrically, a free extension is a flat deformation of $A\otimes_\mathsf{k} B$. Our Theorem~\ref{thm:CI} gives criteria for finding free extensions that are complete intersections. \par There are subtle conditions, investigated by P.~Oblak, T.~Ko\v{s}ir and others, on which pairs of Jordan types $P_\ell, P_{\ell'}$, partitions of $n$, ``commute'' -- can simultaneously occur for a single Artinian graded algebra $A$ or local algebra $\mathcal{A}$ of length $n$ \cite{Kh1,Kh2,KOb,Ob1,Ob2,IKVZ,Pan} (\S \ref{commutingJTsec}).
We hope that our discussion and results in Section \ref{examplessec} will suggest new connections that will be useful to the reader. 

\subsubsection{Detailed Overview.}
In Section \ref{JTypesec} we state and prove the basic properties for Jordan types for elements of Artinian algebras. In \S \ref{JTsec} we present well known equivalent definitions of the Jordan type of the multiplication map $m_\ell$ for an element $ \ell\in {\mathfrak{m}}$, the maximal ideal, and we present some properties of Artinian algebras that we will need. In \S \ref{conjHsec} we prove a main result, Theorem \ref{thm:JtIneq}, bounding above the Jordan type of a non-unit element $\ell$ of a local algebra $\A$, acting via $m_\ell$ on an $\A$-module $M$, by the conjugate of the Hilbert function of $M$; and we also show the analogue for graded algebras. Here we use the dominance partial order on partitions (Definition \ref{dominancedef}).\par In \S \ref{lefsec} we connect Jordan type with Lefschetz properties.
We will also in \S \ref{SLJTsec} consider Jordan type for more general elements $\ell\in \mathfrak{m}_{\mathcal{A}}$ or for non-homogeneous elements of $\mathfrak{m}_A$ when $A$ is graded. We say that an element $\ell $ of the maximal ideal $\m_A$ has ``strong Lefschetz Jordan type'' (SLJT) if $P_\ell=H(A)^\vee$, the conjugate partition of the Hilbert function $H(A)$
(Definition~\ref{SLJTdef}) -- this gives a fine tuning of the concept of strong Lefschetz, and as well provides an additional invariant (namely, SLJT) for distinguishing Artinian non-standard graded algebras.
\par
In \S \ref{hilbsec} we define an invariant, the contiguous partition $P_c(H)$ associated to a Hilbert function $H$ (Definition \ref{def:relativeLef}) and show that for a graded $A$-module $M$, its Jordan type is bounded above in the dominance order by the contiguous partition $P_c(H)$ (Theorem~\ref{thm:contig}). \par
We introduce the Macaulay duality in \S \ref{AAMDsec}, which we use to define examples. In \S \ref{Jdegsec} we introduce a finer invariant, the Jordan degree type (JDT) of a graded module over a graded algebra (Definition \ref{degreeJT-def}) and relate it to the central simple modules of T. Harima and J.~Watanabe.  There is a natural partial order on JDT related to specialization (Lemma \ref{JDTspeclem}). We use this to show that a Jordan type locus in 
$G_T$ may have several irreducible components (Example \ref{2.30ex}). We show also that the Jordan degree type is bounded in the concatenation order by a degree-invariant associated to the Hilbert function (Proposition \ref{PJTgenlem}). Finally, we give in \S \ref{genericsec} several results and examples highlighting the deformation properties of Jordan type; in particular we discuss Jordan type and initial ideal, and we compare the generic Jordan types of a local algebra $\A$ and its associated graded algebra.\par
Section \ref{examplessec} outlines further work by many groups on several aspects of Jordan types. In \S \ref{idealizationsec} we discuss the generic Jordan types for algebras constructed by idealization or by partial idealization: these include some well-known non-WL examples of H. Sekiguchi (H.~Ikeda), idealization examples of J.~Watanabe and R. Stanley, and, as well, partial idealization examples of R. Gondim and collaborators. These Jordan types can be weak Lefschetz (number of parts equal to the Sperner number of $H$) but not strong Lefschetz. In \S \ref{tensorprodsec} we report briefly on the Jordan type of tensor products, in both non-modular ($\cha \mathsf{k} $ is zero or $\cha \mathsf{k}=p$ is large) and modular cases. Our contribution is to introduce here the Jordan degree type, and prove for it an analogue of the Clebsch-Gordan formula for the tensor product.\par T. Harima and J. Watanabe introduced free extensions $C$ of a base algebra $A$ with fibre $B$ \cite{HW0}, which we have shown elsewhere are deformations of the tensor product \cite{IMM}. We here in Section \S \ref{twoexsec} show a new result that is useful in deciding when $C$ is a free extension and also in determining complete intersection extensions (Theorem~\ref{thm:CI}). In \S \ref{commutingJTsec} we discuss which pairs of Jordan types are compatible for different elements of the same Artinian algebra $A$. In \S \ref{prodsec} we propose further problems and possible directions of study. \par 
Related notes are \cite{IMM}, which focuses on free extensions; and \cite{MCIM}, joint with S.~Chen, which studies Jordan type and rings of relative coinvariants. This paper is our introduction to Jordan type and Jordan degree type for Artinian algebras $A$, and to some of the subtleties that arise, particularly in comparing these invariants with the Hilbert function of $A$.
\subsection{Notation.}

Throughout the paper $\sf k$ will be an arbitrary field unless otherwise specified -- except that we will assume $\sf k$ is infinite when we discuss ``generic'' Jordan type or parametrization. 
\par
\subsubsection{Local Artinian algebras.}
We will on the one hand consider a local Artinian algebra $\A$ containing the field $\F= \A/\m_\A$, where $\m_\A$ is the maximal ideal. The Jordan type $P_\ell$ of an element $\ell\in \m_\A$ is the partition of $n=\dim_\F(\A)$ giving the sizes of the blocks in the Jordan block decomposition of the multiplication operator $m_\ell: \A\to \A$; since the homomorphism $m_\ell$ is nilpotent, it is determined up to similarity by the partition $P_\ell$. Recall that the \emph{socle} of $\A$ is the ideal $\Soc(\A)=0:\m_\A$. We denote by $j_\A$, the \emph{maximal socle degree} of $\A$, the highest integer $j$ such that $\m_\A^{\,j}\neq0$. There is a natural $\m_\A$-adic filtration
\begin{equation}\label{ordereq}
\A\supset \m_\A\supset \m_\A^{\,2}\supset \cdots \supset\m_\A^{\,j_\A}\supset 0
\end{equation}
of $\A$; the \emph{order} $\nu(a)$ of a nonzero element $a\in\A$ is the largest integer $i$ such that $a\in\m_\A^{\,i}$. The \emph{associated graded algebra} $\A^\ast $ of $\A$ is $\A^\ast=\Gr_{\m_\A}(\A)=\oplus_{i=0}^j \A_i$ where $\A_i=\m_\A^{\,i}/\m_\A^{\,i+1}$. Here $\A^\ast$ is standard graded, in the sense that $\A^\ast$ is generated over $\A_0=\F$ by $\A_1$. The \emph{Hilbert function} of the local algebra $\A$ is $H(\A)=(h_0,h_1,\ldots, h_j)$ where $h_i=\dim_\F \A_i$. \par
\subsubsection{Graded Artinian algebras.}
By an $\mathbb{N}$-\emph{graded Artinian algebra} over $\F$ we shall mean a graded ring $A=\bigoplus_{i\geq 0}A_i$ with $A_0=\F$ which has finite dimension as a $\F$-vector space. In this graded case we denote by $j_A$ the largest integer $j$ for which $A_j\neq 0$.  We can write $A=R/I$ where $R$ is the polynomial ring  $R=\mathsf{k}[x_1,\ldots,x_r]$, and $I$ an ideal. A local Artinian algebra $\mathcal{A}$ (with a single maximal ideal, as we will always assume) can be written $\mathcal{A}=\mathcal{R}/I$, where $\mathcal{R}$ is the regular local ring $\mathcal{R}=\mathsf{k}\{x_1,\ldots,x_r\}$.\par
The socle of $A$ is $\Soc (A)=0:\mathfrak{m}_A\subset A$, an ideal of $A$ that includes $A_{j_A}$. The Hilbert function of the graded algebra $A$ is the sequence of non-negative integers $H(A)=(h_0,\ldots,h_j)$ where $h_i=\dim_{{\F}}A_i$. We say that $A$ is \emph{standard graded} if $A=\F[A_1]$, the algebra generated by degree one (linear) forms over $A_0=\F$; otherwise it is non-standard graded. 

We observe that a graded Artinian algebra $A$ over $\F$ may be regarded as a local Artinian algebra
$\A$ over $\F$, with maximal ideal $\m_{\A}=\oplus_{i=1}^{j_A}A_i$, and a \emph{weight function} $\mathsf{w}$ specifying the grading. That is, 
there is an algebra homomorphism $\mathsf{w}\colon\mathcal{A}\rightarrow \mathcal{A}[t]$, and the
\emph{$i^{th}$ graded component} of $\A$ (with respect to $\mathsf{w}$) is given by $\mathcal{A}_{i(\mathsf{w})}= \mathsf{w}^{-1}(\mathcal{A}\cdot t^i)$. Then $\A(\mathsf{w})=\oplus_{i\geq 0}\A_{i(\mathsf{w})}$ is a graded Artinian algebra in the above sense. We will sometimes use the notation $\A(\mathsf{w})$ when we want to stress that the Artinian algebra $\A$ is endowed with the weight function $\w$. We define the $\w$-socle degree of $\A$ to be the largest integer $j=j(\w)$ for which $\A_{j(\w)}\neq 0$, and the $\w$-Hilbert function $H_{\w}(\A)=(h_0^{\w},\ldots,h_j^{\w})$ where $h_i^{\w}=\dim_{{\F}}\A_{i(\mathsf{w})}$. Recall that for a local ring $\A$ the Hilbert function $H(\A)$ is defined as that of the associated graded algebra $\Gr_{\m_\A}(\A)$: thus, $\A(\w)$ is standard-graded if and only if $H(\A)=H_{\w}(\A)$. \par For either $R$ or the regular local ring $\mathcal{R}=\F\{x_1,\ldots,x_r\}$ we may specify a weight function $\w$, hence a grading, on suitable quotients of either, using the shorthand notation ${\w}(x_1,\ldots,x_r)=(d_1,\ldots,d_r)$, meaning that ${\w}(x_i)=t^{d_i}$ for each $i$: a quotient $A=R/I$ is \emph{suitable} for $\w$ if $I$ is a homogeneous ideal in the weighting $\w$.

We fix a finite $\A$-module $M$, of $\dim_{{\F}}M=n$ and we consider multiplication maps $m_\ell\colon M\rightarrow M$ by an element $\ell\in\m_\A$. The \emph{Jordan type} $P_\ell=P_{\ell,M}$ (to make $M$ explicit) is the partition of $n$ specifying the block sizes in the Jordan canonical form of $m_\ell$. When $\F$ is infinite, the \emph{generic Jordan type} is $P_M=P_{\ell,M}$ for a sufficiently general element $\ell\in\m_\A$ (Definition \ref{genJTdef}).
We can make similar definitions for $\mathbb N$-graded algebras $A$ with $A_0=\F$ by localizing at the maximum ideal, so $\A=A_{\mathfrak{m}_A}$ where $ \mathfrak{m}_A=\oplus_{i\ge 1}A_i$. Evidently, the Jordan type of $a\in A$, the graded algebra, is the same as that of the corresponding element $a\in \A=A_{\mathfrak{m}_A}$. However, for a standard graded algebra (generated by $A_1$ over $A_0\cong\F$), a special role is played by linear elements $a\in A_1$. We explore this in \S \ref{conjHsec} and \S \ref{lefsec} below
\vskip 0.2cm\noindent
\begin{question}For a standard graded local algebra $\A$ are the two Jordan types $J_\A$ defined by a generic element $a\in \A_1$ and the local algebra Jordan type defined by a generic element $a\in \mathfrak{m}_\A$ the same? Of course, the latter is greater or equal $J_A$ in the dominance order (Definition~\ref{dominancedef}). We will show these Jordan types are the same when $\A$ is standard graded and the Hilbert function $H(\A)$ is unimodal (Proposition \ref{SLJTprop}). In the non-standard graded context, the Jordan type of $\A$ over $\ell\in\mathfrak m$ can be strictly greater than $J_A$ \cite[Proposition~3.9]{MCIM}.
\end{question}

\vskip 0.2cm\noindent
\section{Basic properties of Jordan type.}\label{JTypesec}
\subsection{Jordan type of a multiplication map.}\label{JTsec}
Let $\A=(\A,\mathfrak{m},\F)$ be an Artinian local algebra over $\F$ and let $M$ be a finite $\A$-module; in particular $M$ is a finite dimensional vector space over $\F$. For any element $\ell\in\m_\A$, let $m_\ell\colon M\rightarrow M$  denote the multiplication map $m_\ell(x)=\ell\cdot x$. Then $m_\ell$ is a nilpotent $\F$-linear transformation. We will sometimes denote $m_\ell$ by $\times \ell$.
\begin{definition}[Jordan type]\label{JTdef}(See also \cite[Section 3.5]{H-W})
For any element $\ell\in\mathfrak{m}$ its Jordan type is the partition of $\dim_{{\F}}M$, denoted $P_\ell=P_{\ell,M}=(p_1,\ldots,p_s)$, where $p_1\geq \cdots\geq p_s$, whose parts $p_i$ are the block sizes in the Jordan canonical form matrix of the multiplication map $m_\ell$.
\par
If $P_{\ell,M}=(p_1,\ldots,p_s)$ is the Jordan type of $\ell$, then there are elements $z_1,\ldots,z_s\in M$ (depending on $\ell$) such that $\left\{\ell^iz_k \mid 1\leq k\leq s,\, 0\leq i\leq p_k-1 \right\}$ is a $\F$-basis for $M$: we will term this set a \emph{pre-Jordan basis} for $M$. The Jordan blocks of the multiplication $m_\ell$ are determined by the \emph{strings} $S_k=\left\{z_k,\ell z_k,\ldots,\ell^{p_k-1}z_k\right\}$, and $M$ is the direct sum
\begin{equation}\label{stringeq}
M=\langle S_1 \rangle\oplus\cdots\oplus \langle S_s \rangle.
\end{equation}
We say that the same set is a \emph{Jordan basis} or \emph{$\ell$-basis} for $M$ if also, for each $k$, $\ell^{p_k}z_k=0$. In that case the $\langle S_k\rangle$ are cyclic $\F[\ell]$-submodules.\par

If $A$ is a graded Artinian algebra and $M$ is a graded $A$-module, we
say that the string $S_k$ is \emph{homogeneous} if each element
$\ell^iz_k$ is. 
\end{definition}
For simplicity, we state the following Lemma for $\ell\in A_1$, but it has a natural generalization to $\ell\in A_d$, $d\ge 1$.
\begin{lemma}
\label{JThomoglem} 
Let $A$ be a graded Artinian algebra, $M$ a finite-length graded $A$-module, and let $\ell\in A_1$. Then we may choose strings $S_1,\ldots, S_s$ as in Definition \ref{JTdef} defining the unique Jordan type $P_{\ell,M}$ such that
\begin{enumerate}[(i).]
\item Each $z_k$ is homogeneous of a degree $\nu_k$, and $\ell^iz_k$ has degree $\nu_k+i$ for $0\le i<p_k$, but $\ell^{p_k}z_k=0$;
\item  We have 
\[
\dim_\mathsf{k} M_d=\#\{k\mid \nu_k\le d<\nu_k+p_k\}.
\]
\item Any pre-Jordan basis for $M$ as in Definition \ref{JTdef} may be refined to a set satisfying (i) and (ii) by $\mathsf{k}[\ell]$-linear
operations.
\item Given $\ell\in A_1$, the set of pairs of integers $\mathcal{S}_{\ell,M}=\{(p_k,\nu_k), k=1,\ldots ,s\}$ is independent of the choice of the set of strings $\{S_k\}$ decomposing $M$, and satisfying (i). 
\end{enumerate}
\end{lemma}
\begin{proof}  
This is the result of applying the standard method of finding a good basis of a vector space $V=A$ in which a transformation $T=m_\ell=\times\ell$ will have Jordan normal form \cite[\S VII.7]{Ga}. We briefly sketch the proof. For (\textit{i}), choose an element $z_1$ such that ${\ell^{p_1-1}z_1\ne0}$. Since $\ell$ is homogeneous, we can assume that $z_1$ is also homogeneous and make $S_1=\{z_1,\ell z_1,\ldots,\ell^{p_1-1}z_1\}$. We can construct the remaining strings inductively, choosing in each step a homogeneous element ${z_{k+1}\in\ker(\times\ell^{p_{k+1}})}$ such that ${\ell^{p_{k+1}-1}z_{k+1}\notin\langle S_1 \rangle\oplus\cdots\oplus \langle S_k \rangle}$ -- this is always possible, given that, by construction, ${\im(\times\ell^{p_{k+1}})\subseteq\langle S_1 \rangle\oplus\cdots\oplus \langle S_k \rangle}$, but ${\im(\times\ell^{p_{k+1}-1})\not\subseteq\langle S_1 \rangle\oplus\cdots\oplus \langle S_k \rangle}$, so we can take $z_{k+1}'$ such that ${\ell^{p_{k+1}-1}z_{k+1}'\notin\langle S_1 \rangle\oplus\cdots\oplus \langle S_k \rangle}$,
and consider ${u\in\langle S_1 \rangle\oplus\cdots\oplus \langle S_k 
\rangle}$ such that ${\ell^{p_{k+1}}z_{k+1}'=\ell^{p_{k+1}}u}$.\par Then ${z_{k+1}'-u\in\ker(\times\ell^{p_{k+1}})}$, and we may take $z_{k+1}$ to be a homogeneous component of ${z_{k+1}'-u}$ satisfying ${\ell^{p_{k+1}-1}z_{k+1}\notin\langle S_1 \rangle\oplus\cdots\oplus \langle S_k \rangle}$.

Alternatively, we could start with a Jordan basis ${z_1',\ldots,\ell^{p_1-1}z_1',\ldots,z_s',\ldots,\ell^{p_s-1}z_s'}$, and replace each Jordan chain ${z_k',\ldots,\ell^{p_k-1}z_k'}$ by ${z_k,\ldots,\ell^{p_k-1}z_k}$, where $z_k$ is a homogeneous summand of $z_k'$ satisfying ${\ell^{p_k-1}z_k\ne0}$.
\par 

We easily see that (\textit{ii}) is a direct consequence of (\textit{i}).\par To check (\textit{iii}) we may also construct a new set of strings inductively. The key point here is that if ${S_1,\ldots,S_s}$ is a pre-Jordan basis of $M$ as in Definition \ref{JTdef} then for each $k$, the generator $z_k\notin\im(\times\ell)$, otherwise the lengths of the strings would not match the Jordan type. If ${S_1,\ldots,S_k}$ are already modified into strings $\{S_1',\ldots,S_k'\}$, with generators $z_1' ,\ldots,z_k' $, satisfying (\textit{i}), we can observe that 
\[
\ell^{p_{k+1}}z_{k+1}\in\bigl(\langle S_1'\rangle\oplus\cdots\oplus\langle S_k'\rangle\bigr)\cap\im(\times\ell^{p_{k+1}})=\langle l^iz_b' : i\ge p_{k+1}, 1\le b\le k\rangle.
\]
In particular there exists ${u\in\langle S_1'\rangle\oplus\cdots\oplus\langle S_k'\rangle}$ such that ${\ell^{p_{k+1}}z_{k+1}=\ell^{p_{k+1}}u}$. We now set\par\noindent ${z_{k+1}'=z_{k+1}-u}$ and ${S_{k+1}'=\{z_{k+1}',\ldots,\ell^{p_{k+1}-1}z_{k+1}'\}}$.  \par

For the proof of (iv), let $1\le n_1<\cdots<n_t=s$ be the integers 
defined by the conditions
\[
p_1=p_{n_1}>p_{n_1+1}=p_{n_2}>p_{n_2+1}=p_{n_3}>\cdots,
\]
capturing the places where the partition ${P_{\ell,M}=(p_1,\ldots,p_s)}$ 
drops. Let $S_1,\ldots,S_s$ be any set of strings of a Jordan basis 
satifying (i). Then the classes of $z_1,\ldots,z_{n_1}$ form a 
homogeneous basis of the graded module
\[
\frac{M}{\ker(\times\ell^{p_1-1})}=\frac{\ker(\times\ell^{p_1})}{\ker(\times\ell^{p_1-1})},
\]
so their degrees $\nu_1,\ldots,\nu_{n_1}$ are determined by the 
dimensions of its graded pieces, up to permutation, which shows that the 
pairs $(p_1,\nu_1),\ldots,(p_{n_1},\nu_{n_1})$ are uniquely determined. 
For induction, suppose that the pairs 
$(p_1,\nu_1),\ldots,(p_{n_i},\nu_{n_i})$ are uniquely determined. Then 
the classes of
\[
\ell^{p_1-p_{n_i+1}}z_1,\ldots,\ell^{p_1-p_{n_i+1}}z_{n_1},\ell^{p_{n_1+1}-p_{n_i+1}}z_{n_1+1},\ldots,\ell^{p_{n_1+1}-p_{n_i+1}}z_{n_2},\ldots,z_{n_i+1}\ldots,z_{n_{i+1}}
\]
form a homogeneous basis of the graded module 
$\ker(\times\ell^{p_{n_i+1}})/\ker(\times\ell^{p_{n_i+1}-1})$. Since the 
first pairs are determined, the degrees 
$\nu_{n_i+1},\ldots,\nu_{n_{i+1}}$ are also uniquely determined.
\end{proof}\par
The set of pairs $S_{\ell,M}=\{(p_k,\nu_k)\} $ in Lemma \ref{JThomoglem}(iv) we will later term the \emph{Jordan degree type} of $(M,\ell)$ (Definition \ref{degreeJT-def}).\vskip 0.2cm\par
For an increasing sequence $d_\ell=(d_0\le d_1\le \cdots \le d_j)$ we let $\delta_{d_\ell,i}=d_i-d_{i-1}$ for $0\le i$ with $d_{-1}=0$. For a decreasing sequence $r_\ell=(r_0\ge r_1\ge \cdots \ge r_j)$ we let $\delta_{r_{\ell},i}=r_i-r_{i+1}$.
\par
 \par The following result is well-known, see \cite[Lemma 3.60]{H-W}.
\begin{lemma}\label{JTlem}\vskip 0.2cm\begin{enumerate}[(i).]
\item Let $A$ be an Artinian graded or local algebra with maximum ideal ${\mathfrak{m}}_A$ and highest socle degree $j=j_A$ (so $A_j\not=0$ but $A_i=0$ for $i>j$); and assume $\ell\in {\mathfrak{m}}_A$. Let $M$ be a finite length $A$-module. The increasing dimension sequence
\begin{equation}\label{deq}
d_\ell\colon (0=d_0,d_1,\ldots, d_j,d_{j+1}), \text { where } d_i=\dim_\mathsf{k} M/\ell^{i} M,
\end{equation}
has first difference $\Delta(d_\ell)=(\delta_{d_\ell,1},\delta_{d_\ell,2},\ldots, \delta_{d_\ell, j+1}),$ which satisfies
\begin{equation}\label{Peq}
P_\ell={\Delta(d_\ell})^\vee.
\end{equation}
Here ${\Delta(d_\ell})^\vee$ is the conjugate (exchange rows and columns in the Ferrers diagram) of $\Delta(d_\ell)$.\vskip 0.2cm
\item The (decreasing) rank sequence
\begin{equation}\label{req}
r_\ell\colon (r_0,r_1,\ldots ,r_j,0), \text { where }r_i= \dim_\mathsf{k}(\ell^i\cdot M)=\text { rank }m_{\ell^i} \text { on }M,
\end{equation}
has first difference $\Delta( r_{\ell})=(\delta_{r_\ell,1},\delta_{r_\ell,2},\ldots , \delta_{r_\ell,j})$ which satisfies
\begin{equation}P_\ell^\vee= \Delta( r_{\ell})=\Delta( d_{\ell}).
\end{equation}
\end{enumerate}
\end{lemma}
\noindent
{\bf Note.} The Jordan type partition $P_\ell$ has sometimes been called the Segre characteristic of $\ell$ \cite{Sh}. The Weyr canonical form of a multiplication map is a block decomposition ``dual'' to the Jordan canonical form \cite[\S 2.4]{OCV}; the Weyr characteristic is the partition giving the sizes of the blocks in the Weyr form, and is just the conjugate $P_\ell^\vee$ of $P_\ell$. For further discussion of the Weyr form, which may have advantages for some problems, see \cite{Sh,{MaV},OCV, OW}.\par
It is readily seen that the Jordan type of $(A,\ell)$ may depend on the characteristic of $\mathsf{k}$. For example A. Wiebe notes that $\mathsf{k}[x,y,z]/(x^2,y^2,z^2)$ is strong Lefschetz when ${\cha \mathsf{k}\ne2}$, but is not even weak Lefschetz when ${\cha\mathsf{k}=2}$ \cite[Example 2.10]{Wi}. The same holds for $\mathsf{k}[x,y]/(x^p,y^p)$, which is strong Lefschetz for ${\cha \mathsf{k}=0}$ or ${\cha\mathsf{k}>p}$ but is not SL when ${\cha\mathsf{k}=p}$.  This dependence is studied in particular by D. Cook,~II for monomial ideals, or in codimension two \cite{Co1,Co2}. We discuss it for tensor products in Section \ref{tensorprodsec}.
\subsection{Jordan type and Hilbert function for a local algebra.}\label{conjHsec}
\subsubsection{Partitions and the dominance order.}
By a \emph{partition} we mean a weakly decreasing sequence of non-negative integers $P=(p_1,\ldots,p_s)$, $p_1\geq\cdots\geq p_s$. The $p_i$ are called the \emph{parts} of $P$, the \emph{length} of $P$ is the number of its parts $\ell(P)=s$, and the \emph{size} of $P$ is the sum of its parts $|P|=p_1+\cdots+p_s$. We can represent the partition $P$ as a \emph{Ferrers diagram} by which we mean a left justified array of boxes with $p_i$ boxes in the $i^{th}$ row. The \emph{conjugate partition} $P^\vee=(p_1^\vee,\ldots,p_t^\vee)$ is the partition with parts $p_i^\vee=\#\{j \mid p_j\geq i\}$; its Ferrers diagram can be obtained from that of $P$ by reflection about the main diagonal, that is, swapping the rows and columns. For example $P=(2,2,1,1)$ has Ferrers diagram {\tiny $\yng(2,2,1,1)$}; the conjugate $P^\vee=(4,2)$, has Ferrers diagram {\tiny $\yng(4,2)$}.  Note that $\left(P^\vee\right)^\vee=P$.

\begin{definition}[Dominance order]\label{dominancedef}
Given two partitions $P=(p_1,\ldots ,p_s)$ and $ P'=(p'_1,\ldots, p'_t)$ with $p_1\ge \cdots \ge p_s$ and $p'_1\ge\cdots \ge p'_t$, we write
\begin{equation}\label{dominanceeq}
P\le P' \text { if for all } i \text { we have} \sum_{k=1}^i p_k\le\sum_{k=1}^i p'_k.
\end{equation}
\end{definition}
Thus, $(2,2,1,1)<(3,2,1)$ but $(3,3,3)$ and $(4,2,2,1)$ are incomparable. \par
In the dominance partial order we have for two partitions $P,P'$ of $n$ (see \cite[Lemma~6.3.1]{CM})
\begin{equation}\label{conjPeq}
P\le P' \Leftrightarrow P^\vee\ge {P'}^\vee.
\end{equation}
Any sequence $H=(h_0,\ldots,h_j)$ of non-negative integers can be made into a partition by simply rearranging its parts so that they are in non-increasing order. In particular if $H=H(N)$, the Hilbert function of an Artinian module $N$ over an Artinian algebra $A$ and $\max\{h_k\}=r$ (a special case is $N=\A$, a local Artinian algebra) then we form the \emph{conjugate partition of the Hilbert function} 
\begin{equation}\label{conjHF}
H(N)^\vee=(h_0^\vee,\ldots,h_r^\vee) \text { where }h_i^\vee=\#\left\{k\mid h_k\geq i\right\}\text{ in }H(N).
\end{equation}
The following key result says that the Jordan type is always bounded above in the dominance partial order by the conjugate partition of the appropriate Hilbert function.  We will consider modules $M\subset \A^k$ that are subsets of free modules $\A^k$, in particular their Hilbert functions satisfy $H(M)=\bigl(h_0(M),h_1(M),\ldots\bigr)$ with entries only in non-negative degrees.
\begin{theorem}[Jordan type and Hilbert function]
\label{thm:JtIneq}
Let $\A=(\A,\mathfrak{m},\F)$ be a local Artinian algebra over $\F$ 
and let $M\subset \A^k$ be an $\mathcal{A}$-module, with Hilbert function $H(M)$. For any $\ell\in\mathfrak{m}_\A$, its Jordan type satisfies
\begin{equation}
\label{eq:JTA}
P_{\ell,M}\leq H(M)^\vee.
\end{equation}
If $\A$ has weight function $\w$, for which $\A_0=\F$, and if $M$ is a graded module over $\A(\w)$ with $\w$-Hilbert function $H_{\w}(M)$, then for any $\w$-homogeneous element $\ell\in\mathfrak{m}_\A$ its Jordan type also satisfies 
\begin{equation}
\label{eq:JTgA}
P_{\ell,M}\leq H_{\w}(M)^\vee.
\end{equation}
\end{theorem}

\begin{proof} 
To prove Equation \eqref{eq:JTA}, let $\A$ and $M$ be as above, fix any element $\ell$ of $\maxA$ and let $\mathsf{k}[x]$ act on the Artinian $\mathcal{A}$-module $M$ via $x=m_\ell $, multiplication by $\ell$. Recall that the integer $j(N)$ is the socle degree of an $\A$-module $N$. For $a\in \A$ the order $\nu(a)$ is the largest integer $\nu$ such that $a\in(\mathfrak{m}_{\A})^\nu$, and we extend the definition to elements of modules $M\subset \A^k$: the order of $m=(a_1,\ldots, a_k)$ is the minimum order $\min\{a_1,\ldots, a_k\}$. We will denote by $\{\{H(N)\}\}$ the multiset of integers in the sequence $H(N)$, with their multiplicities specified. We will show\vskip 0.2cm\noindent
{\bf Claim.} For any $T=\mathsf{k}[x]$ submodule $N$ of the $\mathcal{A}$-module $M$ where $\dim_\mathsf{k}M=m$, $\dim_\mathsf{k}N=n$, we have $P_{\ell}(N)\le H(N)^\vee$.\par
\noindent
\textit{Proof of Claim.} We proceed by induction on the pairs $(m,n)$,  where we let $(m,n)<(m',n') $ if $m<m'$ or if $m=m'$ and $n<n'$. The Claim is true for all pairs $(m,n)$ with $m=1$ or $n=1$. Fix $(m',n')$ and suppose the Claim is true for all pairs $(m,n)<(m',n')$, let $M$ be an $\mathcal{A}$-module of length $m'$ and $N$ a $T$ submodule of length $n'$. Let $S=(a,\ell a,\ell^2 a, \ldots )\subset N$ be a longest string ($\F$-basis of a cyclic $T$-submodule) in $N$: then $S$ has length $p_{1,\ell}$ no greater than $j(N)+1$, the largest part of $H(N)^\vee$, and $\langle S\rangle$ is a direct $T$-summand of $N$ (as it has maximum length). Consider a complementary  $T$ submodule $N'\subset N$ with $N'\cong N/\langle S\rangle$ and $ N'\oplus_T\langle S\rangle=N$ and choose $N'$ of maximum possible order.  Denote by $\{\{H(N)\}\}$ the multiset of integers from $H(N)$. Then $\{\{H(N')\}\}$ is obtained from $\{\{H(N)\}\}$ by decreasing $p=p_{1,\ell}$ entries of $\{\{H(N)\}\}$ by one. No entry $H(N)_i$ is decreased by 2 in $H(N')_i$ as the orders of $a,\ell a,\ell^2a,\ldots $ are strictly increasing. Evidently, $P_{\ell,N}=(p,P_{\ell,N'})$ - we simply adjoin a largest part $p$ to the Jordan partition $P_{\ell,N'}$. Since $P_{\ell,N'}\le H(N')^\vee$ by the induction assumption, we have by Equation \eqref{conjPeq}
\begin{equation}\label{keyHNeq}
P_{\ell,N'}^\vee \ge \bigl(H(N')^\vee\bigr)^\vee.
\end{equation} 
The partition $P_{\ell,N}^\vee $ is obtained from $P_{\ell,N'}^\vee$ by increasing the first $p$ entries by one. The multiset $\{\{H(N)\}\}$
 is obtained from $\{\{H(N')\}\}$ by increasing some subset of $p$ entries by one. Thus we have (here $\bigl(H(N)^\vee\bigr)^\vee$ is just the integers in the multiset $\{\{H(N)\}\}$ rearranged in non-increasing order to form a partition).
\begin{equation}\label{keyHN2eq}
P_{\ell, N'}^\vee \ge \bigl(H(N')^\vee\bigr)^\vee \Rightarrow P_{\ell, N}^\vee \ge
\bigl(H(N)^\vee\bigr)^\vee
\end{equation}
in the dominance partial order, since the sum of the first $k$ entries for $P_{\ell,N}^\vee$ remains greater than the analogous sum of the first $k$ entries of $H(N)$, for each $k=1,\ldots, j(N)+1$. We are using here that the difference $H(N)_i-H(N')_i\le 1$ for each $i$. By conjugating the partitions in Equation \eqref{keyHN2eq} and applying Equation \eqref{conjPeq} we have shown $P_{\ell,N}\le H(N)^\vee$. This completes the induction step.

For Equation \eqref{eq:JTgA}, i.e. the graded version, let $\w$ be any weight function on $\A$ as above, so that $A=\A(\w)$ is an $\mathbb N$-graded Artinian algebra over $\F$, and let $M$ be any finite $\A$-module $M\subset \A^k$ with highest socle degree $j_M$.  

Fix a $\w$-homogeneous element $\ell\in\mathfrak{m}_A$, and let $M=\oplus_i\langle S_i\rangle$ as in Equation \eqref{stringeq}, where each $S_i$ is a $\w$-homogeneous Jordan basis of the string $\langle S_i\rangle $ of $M$, $p_i=|S_i|$ and $p_1\geq\cdots\geq p_s$. The Jordan type of $\ell$ acting on $ M$ is $P_{\ell,M}=(p_1,\ldots,p_s)$.  For each $1\leq i\leq s$ and each integer $u$, define the new integer $m(i,u)$ 
to be the number of elements of degree $u$ in the disjoint union of strings
$S_1\sqcup\cdots\sqcup S_i$; clearly we have 
\[\sum_{u\geq 0}m(i,u)=|S_1\sqcup\cdots\sqcup S_i|=p_1+\cdots+p_i.\]

Recall that $H(M)^\vee=(q_1,\ldots,q_t)$ where 
\[q_i=\# \underbrace{\left\{u\mid \dim_{{\F}}M_u\geq i\right\}}_{T_i}=|T_i|.\] 
For each index $1\leq i\leq t$, and each integer $u$, define the new integer $n(i,u)$ to be the number of times the index $u$ appears in the multi-set 
$T_1\cup\cdots\cup T_i$; clearly we have 
\[\sum_{u\geq 0}n(i,u)=|T_1|+\cdots+|T_i|=q_1+\cdots+q_i.\]

Since no two elements of the same string have the same degree, we see that $0\leq m(i,u)\leq i$.  Since $\dim_\mathsf{k}(M_u)\geq m(i,u)$, the index $u$ must appear in $T_{m(i,u)}$, as well as in $T_{m(i,u)-1},\ldots,T_1$.  Thus we see that 
\begin{equation}
\label{mns}
m(i,u)\leq n(i,u).
\end{equation}
Summing Equation \eqref{mns} over all $u$, gives Equation \ref{eq:JTgA}, and completes the proof of the Theorem.
\end{proof}

In the graded case we will see that those linear forms $\ell\in A_1$ which achieve the bounds in Theorem \ref{thm:JtIneq} are exactly those with strong Lefschetz property (Proposition \ref{prop:LefHilb}).
\begin{example}\label{robbianoex}
We thank Lorenzo Robbiano for pointing out that we needed to make explicit our assumption that $A_0=\mathsf{k}$ in Theorem \ref{thm:JtIneq}. He provided the following example when $A_0\not=\mathsf{k}$. Let $\mathsf{k}=\mathbb Q$ be the rationals, set $P=\mathbb Q[x]$, let $M$ be the maximal ideal $M=(x^2+1)\subset P$ and denote by $K=P/M$ the quotient field. Consider the Artinian ring $A=P/M^2$ over $\mathbb Q$. Let $\mathfrak{m}_A$ be the maximal ideal of $A$, and consider the associated graded algebra $G = \Gr_{\mathfrak{m}_A}(A)$. 
It satisfies $\dim_K G=2$, with Hilbert function $H_K(G)=(1,1)$; the Jordan type of the mutiplication $m_x$ is $P_{x,G}=(2)=H_K(G)^\vee$. However, over $\mathbb Q$ we have $\dim_{\mathbb Q}(A) = 4$, $\dim_{\mathbb Q}\mathfrak{m}_A = 2$, so the Hilbert function $H_{\mathbb Q}(A)=(2,2)$. The multiplication $m_x$ on $A$ has the string $1\to x\to x^2\to x^3$, so the Jordan type $P_{x,A}= (4)>H_{\mathbb Q}(A)^\vee=(2,2)^\vee=(2,2).$
\end{example}
\begin{remark}\label{heighttworem} The Hilbert function $H=(1,2,\ldots )$ of a codimension two algebra is unimodal, and has no strict increases; it follows that $H$ is determined by the partition $P(H)$ (a reordering of $H$). This is no longer true in height three.
\end{remark}

\subsection{Lefschetz properties and Jordan type.}\label{lefsec}
We first recall in Definition \ref{def:WLSL} various traditional notions of Lefschetz properties for graded Artinian algebras; see for example \cite[Definition 3.8ff]{H-W}, or \cite[Definition 2.4]{MiNa}. It is well known that if $A$ is a standard graded Artinian algebra with Hilbert function $H=H(A)$ then $\ell\in A_1$ is a strong Lefschetz element for $A$ if and only if $P_\ell=H(A)^\vee$, the conjugate partition (exchange rows and columns in the Ferrers diagram) of $H(A)$ regarded as a partition (Proposition \ref{prop:LefHilb} refining \cite[Prop. 3.64]{H-W}). The element $\ell$ is weak Lefschetz for $A$ if the number of parts of $P_\ell$ is the Sperner number of $A$, the maximum value of the Hilbert function $H(A)$ (Lemma~\ref{WLJTlem}). We then define for a local algebra $\A$ the notion of its having an element of \emph{strong Lefschetz Jordan type} (SLJT) (Definition \ref{SLJTdef}); we give examples and show that if $H(A)$ is unimodal then $A$ has SLJT implies that the algebra is strong Lefschetz (Proposition~\ref{SLJTprop}).
\begin{definition}
\label{def:WLSL}
Let $A$ be a graded Artinian algebra of socle degree $j$ (not necessarily standard graded), and let $\ell\in A_1$ be a linear form. We say that $\ell$ is 
\begin{enumerate}[(i).]
\item (WL) weak Lefschetz if the multiplication maps $\times\ell\colon A_i\rightarrow A_{i+1}$ have maximal rank for each degree $0\leq i< j$.
\item (SL) strong Lefschetz if the multiplication maps $\times\ell^b\colon A_i\rightarrow A_{i+b}$ have maximal rank for each degree $0\leq i< j$ and each integer $b\geq 0$.
\end{enumerate}
$A$ is said to have the weak, resp.\ strong Lefschetz property if it has a weak, resp.\ strong Lefschetz element $\ell\in A_1$.
\end{definition}
For a survey (as of 2013) of the Lefschetz properties for graded Artinian algebras, see \cite{MiNa}. For more recent discussion see relevant portions of \cite{H-W}.
\begin{remark}
\label{rem:Narrow}
T. Harima and J. Watanabe refer to the strong Lefschetz property in the \emph{narrow sense} to mean that for every degree $0\leq i\leq \left\lfloor\frac{j}{2}\right\rfloor$ the multiplication maps $\times\ell^{j-2i}\colon A_i\rightarrow A_{j-i}$ are isomorphisms (see \cite[Definition 3.18]{H-W}). If the Hilbert function $H(A)$ is symmetric, i.e.\ $h_i=h_{j-i}$ for each $i$, then SL in the narrow sense is equivalent to SL in Definition \ref{def:WLSL}.
\end{remark}
Note that a necessary condition for $A$ to have SL is that its Hilbert function $H(A)$ is unimodal. The following result, which is a generalization of Proposition 3.64 in \cite{H-W}, relates the strong Lefschetz property and Jordan type.

\begin{proposition}
\label{prop:LefHilb}
Let $A$ be a (possibly non-standard) graded Artinian algebra and $\ell\in A_1$. Then the following statements are equivalent:
\begin{enumerate}[(i).]
\item For each integer $b$, the multiplication maps $\times\ell^b\colon A_i\rightarrow A_{i+b}$ have maximal rank in each degree $i$. (That is, $\ell$ is SL.)
\item The Jordan type of $\ell$ is equal to the conjugate partition of the Hilbert function, i.e.
\[
P_\ell=H(A)^\vee.
\]
\item There is a set of strings $S_1,\ldots,S_s$ as in Equation \eqref{stringeq}, composed of homogeneous elements, for the multiplication map $\times\ell^b\colon A\rightarrow A$ such that for each degree $u$ and each integer $i$ we have the equivalence 
\begin{equation}\label{emptyintersecteq}
\dim_\mathsf{k}A_u\geq i \ \ \Leftrightarrow \ \ A_u\cap S_a\neq\emptyset,\ \forall \ a\leq i.\end{equation}
\end{enumerate}
\end{proposition}
\begin{proof}
For each integer $i\in [1, j_A]$ define the set of indices $T_i=\left\{u \mid \dim A_u \geq i\right\},$ and let $n_i=\# T_i$. We denote by $t$ the maximum $i$ such that $T_i$ is non-empty. Let $S_1,\ldots,S_s$ as in Equation \eqref{stringeq} be strings for the action of $\ell$ on $A$, arranged so that their lengths $p_i=\# S_i$ are non-increasing, and with $S_i$ having generator $z_i$.  Note that since the map $\times\ell:A\to A$ respects the grading of $A$, its kernel is a homogeneous ideal, and therefore we can construct a Jordan basis composed of homogeneous elements, and we may assume that the elements in each string $S_i$ are homogeneous. To these strings we associate their degree sets $\deg(S_1),\ldots,\deg(S_s)$, where 
\[
\deg(S_i)=\left\{\deg z_i,\ldots,\deg \ell^{p_i-1}z_i\right\}.
\]
$(i)\Rightarrow (ii)$: Assume $\times\ell^b\colon A_i\rightarrow A_{i+b}$ has maximal rank for each $i,b$. We want to show that $t=s$ and $p_i=n_i$ for $1\leq i\leq s$.
\begin{uclaim}
There is an injective function $\sigma\colon \left\{1,\ldots,t\right\}\rightarrow \left\{1,\ldots, s\right\}$ such that for each index $i\in [1,t]$, we have 
\[
T_i\subseteq \deg(S_{\sigma(i)}).
\]
\end{uclaim}
Note that the claim implies that $n_i\leq p_{\sigma(i)}$ for $1\leq i\leq t\leq s$. Then we have 
\begin{equation*}
\dim_\mathsf{k}A= \sum_{i=1}^tn_i \leq \sum_{i=1}^tp_{\sigma(i)} \leq \sum_{i=1}^rp_i 
  = \dim_\mathsf{k} A
\end{equation*}
which implies that $t=s$ and $n_i=p_i$ for all $1\leq i\leq t$, as desired. 
\begin{proof}[Proof of Claim] 
We proceed by induction on $i$. We have $T_1=\left\{j \mid \dim_\mathsf{k}(A_j)\geq 1\right\}=\left\{0,\ldots,d\right\}$. Since $\times\ell^d\colon A_0\rightarrow A_d$ has full rank, we conclude that $T_1$ must belong to the degree sequence of some Jordan string, say $S_{\sigma(1)}$. Inductively, assume that we have defined an injective function $\sigma\colon\left\{1,\ldots,i-1\right\}\rightarrow \left\{1,\ldots,s\right\}$ for which 
\begin{equation}
\label{eq:TStrings}
\begin{cases} 
&T_1\subseteq  \deg(S_{\sigma(1)})\\
&\cdots  \\
&T_{i-1}\subseteq  \deg(S_{\sigma(i-1)})\\
\end{cases}.
\end{equation}
Write $T_i=\left\{u \mid \dim_\mathsf{k}A_u\geq i\right\}=\left\{u_1<\cdots < u_{n_i}\right\}$. By our assumption, the multiplication map 
\[
\times\ell^{u_{n_i}-u_1}\colon A_{u_1}\rightarrow A_{u_{n_i}}
\]
has rank at least $i$, hence there are at least $i$ distinct Jordan strings which meet both $A_{u_1}$ and $A_{u_{n_i}}$. Since there are only $(i-1)$ strings appearing in Equation \eqref{eq:TStrings}, there must be one not listed, call it $S_{\sigma(i)}$ for which $T_i\subseteq\deg(S_{\sigma(i)})$. This completes the induction step and proves the claim. 
\end{proof}\vskip 0.2cm
$(ii)\Rightarrow (iii)$: Assume that $P_\ell=H(A)^\vee$. Then $s=t$ and $p_i=n_i=\#\left\{u \mid \dim_\mathsf{k}A_u\geq i\right\}$. Clearly, if $A_u\cap S_a\neq \emptyset$ for all $a\leq i$ then $\dim_\mathsf{k}A_u\geq i$. We prove the other implication by downward induction on the integers $i\leq s$. For the base case, if $\dim_\mathsf{k}A_u\geq s$, then $A_u$ must contain exactly one element from each Jordan string, hence $A_u\cap S_a\neq \emptyset \ $ for all $a\leq s$. For the inductive step, assume the implication $\Rightarrow$ of (iii) for indices greater than $i$, and suppose that $\dim_\mathsf{k}A_u\geq i$. If $\dim_\mathsf{k}A_u\geq i+1$ then by the induction hypothesis $A_u\cap S_a\neq \emptyset $ for all $ a\leq i+1$. On the other hand if $\dim_\mathsf{k}A_u=i$, then $A_u\cap S_a=\emptyset$ for all $ a\geq i+1$. Indeed, for each index $m\geq i+1$, 
\[
\dim_\mathsf{k}A_u\geq m \ \Rightarrow \ A_u\cap S_m\neq\emptyset.
\]
By our assumption there are exactly $p_m$ such indices $u$, hence if $\dim_\mathsf{k}A_u=i<m$ then $A_u\cap S_m=\emptyset$. So if $\dim_\mathsf{k}A_u=i$ we must have $A_u\cap S_a\neq \emptyset $ for all $ a\leq i$. This completes the induction step.\vskip 0.2cm
$(iii)\Rightarrow (i)$: Assume that for each $1\leq i\leq t$ we have the equivalence 
\[
\dim_\mathsf{k}A_u\geq i \ \Leftrightarrow \ A_u\cap S_a\neq \emptyset \ \forall \ a\leq i.
\]
Fix an index $i\in[1,t]$ and an integer $b$, and consider the multiplication map $\times\ell^b\colon A_i\rightarrow A_{i+b}$. Let $m=\min\left\{\dim_\mathsf{k}A_i,\,\dim_\mathsf{k}A_{i+b}\right\}$. Then 
\[
\dim_\mathsf{k}A_{i},\,\dim_\mathsf{k}A_{i+b}\geq m
\]
implies that the $m$ Jordan strings $S_1,\ldots,S_m$ each intersect both $A_i$ and $A_{i+b}$, which in turn implies that $\times\ell^b\colon A_i\rightarrow A_{i+b}$ has rank $m$.
\end{proof}\par

Recall that for $A$ graded Artinian the \emph{Sperner number} $\mathrm{Sperner}(A)=\max\{H(A)_i\mid i\in [0,j]\}$ \cite[\S 2.3.4]{H-W}. For a local ring $\mathcal{A}$, the $\mathrm{Sperner}(\mathcal{A})= \max\{ \mu (\maxA^i) \mid i\in [0,j]\}$ where $\mu (I)=\#$ minimal generators of $I$.
\begin{lemma}\cite[Proposition 3.5]{H-W}\label{WLJTlem} 
When the Hilbert function $H(A)$ for a standard graded Artinian algebra $A$ is unimodal and symmetric then $\ell\in A_1$ is weak Lefschetz for $A$ if and only if $\dim_\mathsf{k}A/\ell A=\mathrm{Sperner}(A)$ or, equivalently, if $P_\ell$ has $\mathrm{Sperner}(A)$ parts.
\end{lemma}

\subsubsection{Lefschetz properties for local algebras.}\label{SLJTsec}
Here we use Jordan type to extend the strong and weak Lefschetz properties to local Artinian algebras.
\begin{definition}
\label{SLJTdef}
Let $\A=(\A,\mathfrak{m},\F)$ be a local Artinian algebra over $\F$ with Hilbert function $H(\A)$. We say that an element $\ell\in\mathfrak{m}$ has
\begin{enumerate}[(i).]
\item (SLJT) strong Lefschetz Jordan type if $P_\ell=H(\A)^\vee$.
\item (WLJT) weak Lefschetz Jordan type if $P_\ell$ has $\mathrm{Sperner}(\A)$ parts.\\
\end{enumerate}
Additionally if $\A$ is graded via a weight function $\w$ with $\w$-Hilbert function $H_{\w}(\A)$, then we say that a $\w$-homogeneous element $\ell\in\mathfrak{m}$ has
\begin{enumerate}[(i).]
\item[iii.] \emph{ $\w$-strong Lefschetz Jordan type} ($\w$-SLJT) if $P_\ell=H_{\w}(\A)^\vee$. 
\end{enumerate}
We say that $\A$ has SLJT, resp.\ WLJT, resp.\ $\w$-SLJT if it has an element $\ell\in\mathfrak{m}$ of that type.
\end{definition}
\begin{remark}
\label{rem:SLvsSLJT}
Note that if $\A$ is graded with weight function $\w$, and $H(\A)^\vee<H_{\w}(\A)^\vee$, then $\A$ cannot possibly be $\w$-SLJT and $\A(\w)$ cannot possibly be SL, while SLJT may or may not hold. On the other hand in the standard graded case we have $H(\A)=H_{\w}(\A)$, and hence $\w$-SLJT and SLJT are equivalent conditions on $\A$. Evidently, in the standard graded case the SL condition on $\A(\w)$ implies the SLJT condition on $\A$. The next result shows that the converse holds under the additional assumption that the Hilbert function $H(\A)$ is unimodal.
\end{remark}
\begin{proposition}\label{SLJTprop} 
Assume that $A$ is a standard graded Artinian algebra and $H(A)$ is unimodal. Then $A$ has an element of strong Lefschetz Jordan type if and only if $A$ has a strong Lefschetz element.
\end{proposition}
\begin{proof} 
Assume that $A$ has an element $\ell$ (possibly non-homogeneous) of strong Lefschetz Jordan type, so $P_\ell=H(A)^\vee=(p_1,p_2,\ldots, p_s)$ with $p_1\ge p_2\ge \cdots \ge p_s$. All that is needed for the forward direction is to show that there is a linear element $\ell'$ that is of strong Lefschetz Jordan type. Consider Jordan strings $S_1,\ldots ,S_s$ for $\ell$ as in Definition \ref{JTdef}, where $S_k=(z_{k},\ell z_{k},\ldots ,\ell^{p_k-1}z_{k})$. The orders of elements in a single string are distinct. Let $\ell'$ be the initial form of $\ell$, which, as we will see, must be linear. We will modify the strings, if needed, to a set of Jordan strings $S'_1,\ldots, S'_s$ for $\ell$ whose initial forms are Jordan strings for $\ell'$: this will show $P_{\ell'}=P_\ell=H(A)^\vee$, and prove that $A$ is strong Lefschetz.\par
Given that $H(A)$ is unimodal, we claim that we may choose the strings so that
\begin{enumerate}[(i).]
\item The first $t$ strings together contain $\min\{H(A)_i,\,t\}$ elements of order $i$ for each $i\in [0,j_A]$; and the initial forms of these elements are linearly independent;
\item The order $\nu(\ell^i z_{t})=\nu(z_{t})+i$ for each pair $(t,i)$ satisfying $1\le t\le s$ and $0\le i\le p_t-1$.
\end{enumerate}
We prove (i) and (ii) by complete induction on $t$. Considering $t=1$, the longest string $S_1=(z_{1},\ell z_{1},\ldots, \ell^{p_1-1}z_{1} )$ where $p_1-1=j_A$, the socle degree of $A$; and we may choose, after scaling by a non-zero constant, $z_{1}=1+\alpha$, $\alpha\in \m_A$. It follows (since $A$ has standard grading) that the initial term $\ell'$ of $\ell$ is linear, and that the elements in the string $S_1'=\bigl(1,\ell',\ldots ,(\ell')^{p_1-1}\bigr)$ satisfy $S_1'=\pi(S_1)$, where $\pi$ is the projection of the elements of $S_1$ onto their initial forms, and form a string of length $p_1$ for $\ell'$.\par
For the induction step we will need several facts. Denote by $m(t,H)$ the smallest integer $i$ such that $H(A)_i\ge t$, and $n(t,H)$ the largest integer $i$ such that $H(A)_i\ge t$.\vskip 0.2cm\noindent
\emph{Fact 1}. That $H(A)$ is unimodal is equivalent to the inequalities:
\begin{equation}\label{unimodal2eq}
m(1,H)\le m(2,H)\le \cdots \le m(s,H)\le n(s,H)\le n(s-1,H)\le \cdots \le n(1,H).
\end{equation}
Also, we have $p_u=1+n(u,H)-m(u,H)$.\vskip 0.2cm\noindent
\emph{Fact 2}.
Given $t\in [1,s]$ the condition (i) above implies\vskip 0.2cm
\begin{enumerate}[(i).]
\item[(iii).] Let $i<m(t,H)$ then the initial forms of all elements of $S_1\cup\cdots \cup S_t$ having order no greater than $i$ are a basis for $A/\mathfrak{m}_A^{\,i+1}\cong \oplus_{k=0}^i A_k$, and
\item[(iv).] Let $i>n(t,H)$. The union $\bigcup_{k=1}^t (\mathfrak{m}_A^{\,i}\cap S_k)$ is a basis for $\mathfrak{m}_A^i=\oplus_{k=i}^{j_A} A_k$.
\end{enumerate}
\emph{Induction step}: Fix $u\in [1,s-1]$ and assume that a set $S_1,\ldots, S_s$ of Jordan strings for $m_\ell$ has been chosen satisfying (i) and (ii) for all integers $t\le u$. We will keep the strings $S_1,S_2,\ldots,S_u$ fixed and will modify the chain $S_{u+1}$ to obtain a set of $s$ Jordan chains for $\ell$ so that the conditions (i), (ii) will be satisfied for all $t\in[1,u+1]$.\par
Consider the next string $S_{u+1}$, of length $p_{u+1}$; by assumption, its elements are linearly independent of the span of those from $S_1\cup \cdots \cup S_u$. Using that (i) and (ii), hence (iii), (iv) are satisfied for $t\le u$, we may adjust the generator $z_{u+1}$ for the string $S_{u+1}$ by linear combinations of elements from the previous strings to obtain a possibly new generator within the span of $S_1,\ldots, S_{u+1}$ having order $m(u+1,H)$, and whose initial form $z'_{u+1}=\pi(z_{u+1})$ is linearly independent of the degree $m(u+1,H)$ initial forms from elements of the strings $S_1,\ldots ,S_u$. Using (iv), we may adjust the generator $z_{u+1}$ further by suitable elements of order at least $m(u+1,H)$ from the previous $u$ strings so that $\ell^{p_{u+1}}\cdot z_{u+1}=0$. It follows that $\ell'^{p_{u+1}}\cdot z'_{u+1}=0$, and $z'_{u+1}$ is generator of an $\ell'$ string of length $p_{u+1}$, linearly independent from the $\ell'$ strings $S'_1, \dots ,S'_u$ determined by the initial elements from $S_1,\ldots,S_u$. It follows that (i) and (ii) are satisfied for $t\in [1, u+1]$. This completes the induction step. \par
We have shown (i) and (ii) for $S_1$ and the induction step. It follows that $P_{\ell'}=P_\ell=H(A)^\vee$, as claimed.\par
The converse, that $A$ has a strong Lefschetz element implies it has one of strong Lefschetz Jordan type is obvious from the definitions.
\end{proof}\vskip 0.2cm\noindent
The following result is well known (and has been reproved several times).
\begin{lemma}[Height two Artinian algebras are strong Lefschetz]
\label{heighttwolem} 
Let $A=\mathsf{k}[x,y]/I$ be Artinian standard graded of socle degree $j$, or $\mathcal{A}=\mathsf{k}\{x,y\}/I$ be local Artinian, and suppose $\cha \mathsf{k}=0$ or $\cha \mathsf{k}\ge j$. Let $\ell$ be a general element of $\mathfrak{m}_A$ in the first case, or of $\maxA$ in the second. Then $\ell$ has strong Lefschetz Jordan type and $A$ is strong Lefschetz, or $\mathcal{A}$ is of strong Lefschetz Jordan type, in the second.
\end{lemma}
\begin{proof} 
These statements follow readily from J. Brian\c{c}on's standard basis theorem for ideals in $\mathbb C[x,y]$ \cite{Bri}, that extends to the case $\cha \mathsf{k}=p\ge j$ (see \cite[Theorem 2.16]{BaI}).\footnote{Proofs in the case $A$ is graded occur also in \cite[Proposition 4.4]{HMNW} and \cite[Theorem 4.11]{Co2}, see \cite[Theorem 2.27]{MiNa}.}
\end{proof}\par
\begin{example} 
Let $\A=\F\{x,y\}/(xy-x^3, y^2)$ with weight function $\w(x,y)=(1,2)$. Then the Hilbert function and $\w$-Hilbert function are, respectively $H(\A)=(1,2,1,1,1)$ and $H_\w(\mathcal{A})=(1,1,2,1,1)$, with conjugate partitions $H(\A)^\vee=H_{\w}(\A)^\vee=(5,1)$.\par
A \emph{standard basis} $\sf B$ for the local algebra $\mathcal{A}$ is one such that the elements of $\mathsf{B}\cap \mathfrak{m}_{\mathcal{A}}^{\,i}$ are a basis for $\mathfrak{m}_{\mathcal{A}}^{\,i}$. Such a basis for $\mathcal{A}$ is $\{ 1, x; y, x^2; x^3, x^4\}$. The multiplication $\times x$ in this basis has Jordan strings 
\begin{equation}
1\to x\to x^2\to x^3\to x^4\to 0 \text { and }(y-x^2)\to 0.
\end{equation} 
Thus, $P_{x,\mathcal{A}}=(5,1).$\par
A basis for the graded ring $\A(\w)$ are the classes of $\{1, x, y, x^2, xy, x^2y\}$ and the only linear element is $x$: the strings of $m_x$ on this basis for $A$ are $(1\to x\to x^2\to xy=x^3\to x^2y=x^4\to 0)$ and $\bigl((y-x^2)\to 0\bigr)$ so here we have $P_{x,\mathcal{A}}=(5,1)=H_{\w}(\A)^\vee$. \par
\end{example}
However, the Example \ref{firstex} shows that a generic linear element $\ell$ of the graded algebra $A$, and a generic element $\ell'\in\mathfrak{m}_{\mathcal{A}}$ of the related local algebra $\mathcal{A}$, may satisfy $P_{\ell,A}=(7,1,1)<(7,2)=P_{\ell',\mathcal{A}}$.

\subsection{Jordan types consistent with a Hilbert function.}\label{hilbsec}
In \S \ref{conjHsec} we proved an inequality for Jordan types of elements of $\A$ and the Hilbert function (Theorem \ref{thm:JtIneq}). Here we show a similar inequality for graded algebras $A$, that is sharper when the Hilbert function is non-unimodal. It depends on a certain refinement $P_c(H)$ of the conjugate partition $H^\vee$, that we now define.

\subsubsection{The contiguous partition $P_c(H)$ of a Hilbert function.}
In general, given any finite sequence of non-negative integers $H=(h_0,\ldots,h_j)$ we consider its \emph{bar graph} as an array of dots (or boxes) arranged into columns with $h_i$ dots (or boxes) in the $i^{th}$ column. For example if $H=(2,3,1,4,0,2)$ then its bar graph is 
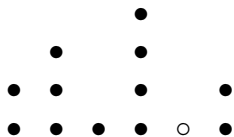
\[
\begin{array}{ccccccc}
&&& \bullet & & \\
& \bullet && \bullet & & \\
\bullet & \bullet &&\bullet & & \bullet \\
\bullet & \bullet & \bullet & \bullet & \circ & \bullet\\ 
\end{array}
\]
\begin{definition}[Contiguous partition, relative Lefschetz property]\label{def:relativeLef} 
i. The \emph{contiguous partition} $P_c(H)$ of the Hilbert function $H$ is the partition whose parts are the lengths of the maximal contiguous row segments of the bar graph of $H$. 
\par
ii. We say a linear form $\ell\in A_1$ in a graded Artinian algebra $A$ has the \emph{Lefschetz property relative to $H$} if its Jordan type is equal to the contiguous partition of $H$:
\[
P_\ell=P_c(H).
\]
\end{definition}
The following result is immediate.
\begin{lemma}[Hilbert function $H(M)$ of a finite-length module over $R$ and 
$P_\ell$] 
Let $M=M_0\oplus M_1\oplus \cdots\oplus M_j$ be an Artinian graded module over the polynomial ring $R=\mathsf{k}[x_1,\ldots ,x_r]$, satisfying $H(M)=(h_0,\ldots ,h_j)$, and let $\ell\in R_1$. Then for $1\le k\le j$
\begin{equation}\label{hilbineq}
\rk\, m_{\ell}^k\le \sum_{i=0}^{j-k}\min \{h_i,h_{i+1},\ldots h_{i+k}\}.
\end{equation}
Also, $P_\ell=P_c(H)$ if and only if there is equality for every $k$ in Equation \ref{hilbineq}.
\end{lemma}
\begin{proof} 
We give the proof for $M=A$. We observe that for any $\ell\in A_1$, the map $m_\ell^k\colon A_i\rightarrow A_{i+k}$ has rank at most $\min\left\{h_i,\ldots,h_{i+k}\right\}$. Summing over all $i$ we get the inequality of Equation~\ref{hilbineq}.\par
Recall that the conjugate Jordan type $P_\ell^\vee=(q_1,\ldots,q_j)$ is the first difference of the rank sequence of $m_\ell$, i.e.\ 
\[
q_k=\operatorname{rk}(m_\ell^k)-\operatorname{rk}(m_\ell^{k+1}).
\]
Hence if $\ell$ has the Lefschetz property relative to $H$, then it follows that Equation \eqref{hilbineq} is actually an equality. In particular, a Lefschetz element $\ell\in A_1$ relative to $H$ is one whose multiplication maps $m_\ell^k\colon A\rightarrow A$ have the maximal possible rank, given the Hilbert function, for each integer $k$. 
\end{proof}

The following result pertains to the conjugate of $P_c(H)$: when $H$ is unimodal then $P_c(H)^\vee=\{H\}$, that is, the Hilbert function viewed as a partition.
\begin{lemma}
\label{lem:contig}
Given any finite sequence of non-negative integers $H=(h_0,\ldots,h_j)$ if $P_c(H)^\vee =(p_1,\ldots,p_s)$, then 
the parts are given by 
\begin{equation}
\label{eq:conjcontig}
p_i=\sum_{k=0}^{j+1-i}\min\left\{h_k,\ldots,h_{k+i-1}\right\}-\sum_{k=0}^{j-i}\min\left\{h_k,\ldots,h_{k+i}\right\}.
\end{equation} 
\end{lemma}
\begin{proof}
The $i^{th}$ part of $P_c(H)^\vee$ is 
\[
p_i=\# \ \text{maximal contiguous row segments of length $\geq i$  in the bar graph of $H$.}
\]
Note that the sum 
\begin{equation}
\label{eq:sum}
\sum_{k=0}^{j+1-i}\min\left\{h_k,\ldots,h_{k+i-1}\right\}
\end{equation} 
counts the maximal contiguous row segments of length greater or equal to $i$ with a multiplicity equal to the number of length $i$-intervals it contains; in particular it counts a contiguous row segment of length $i$ exactly once. On the other hand, the sum
\begin{equation}
\label{eq:sum2}
\sum_{k=0}^{j-i}\min\left\{h_k,\ldots,h_{k+i}\right\}
\end{equation} 
counts maximal contiguous row segments of length $\geq i+1$ with multiplicity one less than they are counted in Equation \eqref{eq:sum}. Therefore the difference of the sums in Equations \eqref{eq:sum} and \eqref{eq:sum2} must count every maximal contiguous row segment of length $\geq i$ exactly once.
\end{proof}

\begin{theorem}
\label{thm:contig}
For a finite graded module $M$ over a graded Artinian algebra $A$ with Hilbert function $H(M)$, we have for any linear form $\ell\in A_1$
\[
P_{\ell,M}\leq P_c\bigl(H(M)\bigr).
\]
\end{theorem}
\begin{proof}
Given $\ell\in A_1$ we may choose a Jordan basis for $m_\ell\colon M\rightarrow M$ with strings $S_1,\ldots,S_s$ as in Equation \eqref{stringeq}, each of cardinality $|S_i|=p_i$ with $P_{\ell,M}=(p_1,\ldots,p_s)$ with $p_1\ge \cdots \ge p_s$. Since $\ell$ is linear, each string must belong to some maximal contiguous row segment of the Hilbert function $H(M)$, which implies the desired inequality. 
\end{proof}
\begin{example}\label{1.36ex}
Let $\A=\F\{y,z\}/(yz,z^3,y^7)$ with weights $\w(y,z)=(1,2)$, and $\w$-Hilbert function $H_{\w}(\A)=(1,1,2,1,2,1,1)$. Then $y\in \A_1$ is a generic linear form with Jordan type $P_y= P_c(H)=(7,1,1)$, the maximum possible by Theorem \ref{thm:contig}; in particular $y$ has the Lefschetz Property relative to $H_{\w}(\A)$ (Definition \ref{def:relativeLef}). On the other hand the conjugate partition of the $\w$-Hilbert function is $H_{\w}(\A)^\vee=(7,2)$. The Hilbert function for the related localization $\A$ at $\m_\A=\sum_{i\ge1} \A_i$ is $H(\A)=(1,2,2,1,1,1,1)$; the conjugate partition of this Hilbert function is also $H(\A)^\vee=(7,2)$. Thus $y$ does not have the strong Lefschetz Jordan type for $\A$, nor is it strong Lefschetz for $\A(\w)$ (or even WL). But the non-$\w$-homogeneous element $\ell=(y+z)\in \mathfrak{m}$ has Jordan type $P_\ell=(7,2)$, hence $\A$ has strong Lefschetz Jordan type (SLJT).
\end{example}
\noindent
\textbf{Note.} Recall that we have adopted in Definition \ref{def:WLSL}(ii) the convention of T. Harima and J.~Watanabe that a non-standard graded algebra $A$ is strong Lefschetz if and only if there is a \emph{linear} form $\ell\in A_1$ that has SLJT (see \cite{HW,H-W}). Thus, the rings $A$ of Examples~\ref{1.36ex}, \ref{firstex}, \ref{2.44ex} and \ref{dualgenex} are not strong Lefschetz, even though in each Example the corresponding local ring $\A$ has an element of strong Lefschetz Jordan type: as $y+z$ in Example \ref{1.36ex} is an element in $\A$ having SLJT --- so $\A$ has SLJT by our Definition \ref{SLJTdef}. 
\begin{remark}\label{sfPbargraphrem} 
Instead of using $(\dim_\mathsf{k}M,\, \dim_{\mathsf{k}}M/\ell M,\,\dim_\mathsf{k}M/\ell^2M,\ldots)$ as in Equation \eqref{deq} to define $P_{\ell,M}$ we may replace each $\ell^k$ by $(\ell_1\cdot \ell_2\cdots \ell_{k})$, a product of different -- generic -- linear forms, yielding a partition $Q(M)$. It can be shown similarly to the proof of Equation \eqref{eq:JTA} that
\begin{equation}
P_{\ell,M}\le Q(M)\le H(M)^\vee.
\end{equation}
We can ask similar questions for $Q(M)$ to those we ask about $P_{\ell,M}$. When is $P_{\ell,M}=Q(M)$? Also, the partition $Q(M)$ appears to be related to the concepts of ``$k$-Lefschetz'' \cite[\S 6.1]{H-W} and ``mixed Lefschetz'' \cite{Cat}. What is the relation of these concepts to Jordan type?
\end{remark}
We note that Jordan type and Hilbert function has been in particular studied for codimension two complete intersections in \cite{AIK}; also B. Costa and R. Gondim have used mixed Hessians to study other examples of Jordan type in higher codimension \cite{CGo}.

\subsection{Artinian Gorenstein algebras and Macaulay dual.}\label{AAMDsec}
The polynomial ring $R=\mathsf{k}[x_1,\ldots ,x_r]$,  acts on its dual $\mathfrak{D}=\mathsf{k}_{DP}[X_1,\ldots ,X_r]$ by contraction: $x_i^k\circ X_j^{[k']}= \delta_{i,j}X_j^{[k'-k]}$ for $k'\ge k$, extended multilinearly.\footnote{F.H.S. Macaulay used the notation $x_i^{-s}$ for the element we term $X_i^{[s]}$ in $\mathfrak{D}$.} We say that a graded Artinian algebra is \emph{standard graded} if $A$ is generated by $A_1$. We next define a homogeneous element of $\mathfrak{R}$, the Macaulay dual generator for a graded Artinian quotient $A=R/I$. \par 
Likewise, the regular local ring $\mathcal{R}=\mathsf{k}\{x_1,\ldots,x_r\}$ acts on
$\mathfrak{D}$, also by contraction and  we will define similarly a dual generator in $\mathfrak{R}$ for a local Artin algebra $A=\mathcal{R}/I$. \begin{definition}[Macaulay dual generator]\label{Macdualdef}
An Artinian Gorenstein (AG) algebra quotient  $A=R/I$ (respectively, $\mathcal{A}=\mathcal{R}/I$) satisfies $A=R/\Ann f$ (respectively, $\mathcal{A}=\mathcal{R}/\Ann {f}$), where $ f\in \mathfrak{D}=\mathsf{k}_{DP}[X_1,\ldots ,X_r]$ is called the \emph{dual generator} of $\mathcal{A}$. The module $\hat{A}=R\circ f$ in the graded case, or $\hat{\mathcal{A}}=\mathcal{R}\circ f$ in the local case is the Macaulay dual of $\mathcal{A}$, equivalent to the Macaulay \emph{inverse system} of the ideal $I$. The \emph{socle} of $A$ (respectively of $\mathcal{A}$) is $\Soc(A)=(0:\mathfrak{m}_A)\subset A$ (respectively $\Soc(\mathcal{A})=(0:{\mathfrak{m}_{\mathcal{A}}})\subset \mathcal{A}$), is the unique minimal non-zero ideal of $A$ or of $\mathcal{A}$, and $\dim_\mathsf{k}\Soc(\mathcal{A})=1$.\par
For a more general Artinian algebra $A=R/I$ (graded) or $\mathcal{A}=\mathcal{R}/I$ (local), a set of \emph{Macaulay dual generators} of $A$ are a minimal set of $A$ (or $\mathcal{A}$) module generators in $\mathfrak{D}$ of $I^{\perp}=\{h\in \mathfrak{D}\mid I\circ h=0\}$.
\end{definition}
\begin{example}[Artinian Gorenstein] 
(i). Let $R=\mathsf{k}[x,y]$, $A= R/I$, $I=\Ann f$ with $ f=XY\in \mathfrak{D}=\mathsf{k}_{DP}[X,Y]$. Then $I=(x^2,y^2)$ and $A=R/(x^2,y^2)$ of Hilbert function $H(A)=(1,2,1)$. Here $x^2\circ XY=0$ is the contraction analogue of $\partial^2(XY)/(\partial X)^2=0$ and the dualizing module $A^\vee=R\circ f=\langle 1,X,Y,XY\rangle$.\par
(ii). Let $\mathcal{R}=\mathsf{k}\{x,y\}$, the regular local ring, and take $f=X^{[4]}+X^{[2]}Y$. Then $\mathcal{A}=\mathcal{R}/I$, $I=\Ann f=(xy-x^3, y^2)$, and the Hilbert function $H(\mathcal{A})=(1,2,1,1,1)$. The dualizing module $A^\vee=\mathcal{R}\circ f=\langle 1,\,X,\,Y,\,X^{[2]},\,X^{[3]}+XY,\,f\rangle$.
\end{example}
Letting $\mathfrak{m}_{\mathcal{A}}$ be the maximal ideal of the Artinian Gorenstein local algebra $\mathcal{A}$, we have $\Soc (\mathcal{A})=\mathfrak{m}_{\mathcal{A}}^{\,j}$ where $\mathfrak{m}_{\mathcal{A}}^{\,j}\not=0$ and $\mathfrak{m}_{\mathcal{A}}^{\,j+1}=0$. Then we have the following result (\cite[\S 60-63]{Mac}, \cite[Lemma 1.1]{I1}, or, in the graded case, \cite[Lemma 1.1.1]{MS}):
\begin{lemma}\label{dualgenlem}[Dual generator for AG algebra]
\begin{enumerate}[(i).]
\item Assume that $\mathcal{A}=\mathcal{R}/I$ is Artinian Gorenstein of socle degree $j$. Then there is a degree-$j$ element ${f}\in \mathfrak{D}$ such that $I=I_f=\Ann f$. Furthermore $f$ is uniquely determined up to action of a differential unit $u\in \mathcal{R}$: that is 
\begin{equation}
\Ann f=\Ann (u\circ f). \text { Also, } \Ann f=\Ann g\Leftrightarrow g=u\circ f \text { for some unit } u\in \mathcal{R}.
\end{equation} 
The $\mathcal{R}$-module $(\Ann f)^\perp=\{h\in \mathfrak{D}\mid (\Ann f)\circ h=0\}$ satisfies $(\Ann f)^\perp =R\circ f$.
When ${f}$ is homogeneous, it is uniquely determined by $\Ann f$ up to nonzero constant multiple.\vskip 0.2cm
\item Denote by $\phi:\Soc(\mathcal{A})\to \mathsf{k}$ a fixed non-trivial isomorphism, and define the pairing $\langle \cdot,\cdot\rangle_\phi$ on $\mathcal{A}\times \mathcal{A}$ by $\langle (a,b)\rangle_\phi=\phi(ab)$. Then the pairing $\langle (\cdot,\cdot)\rangle_\phi$ is an exact pairing on $ \mathcal{A}$, for which $(\m_\A^{\,i})^\perp=(0:\mathfrak{m}_{\mathcal{A}}^{\,i})$. We also have $0:\m_{\A}^{\,i}=\Ann(\m_{\A}^{\,i}\circ f)$ and $\Ann (\ell^i\circ f)=I_f:\ell^i$.
\vskip 0.2cm
\item When $A=\oplus_0^j A_i$ is graded (not necessarily standard-graded) of socle degree $j$ (largest integer for which $A_j\not=0$), analogously to (ii), we choose an isomorphism $\phi:A_j\to\mathsf{k}$, and then define the inner product $\langle \cdot,\cdot\rangle_\phi$. 

We note that in the above pairing $(A_{\ge i})^\perp=(0:A_{\ge i})=A_{\ge j+1-i}$. Passing to quotients $A_i=A_{\ge i}/A_{\ge i+1}$ we conclude that each $A_i\cdot A_{j-i}\to A_j$ is an exact pairing, and the Hilbert function $H(A)$ is symmetric about $j/2$. When $A$ is standard graded, we also have, taking $\mathfrak{m}_A=\oplus_{k=1}^j A_k$, that $(\m_A^{\,i})=A_{\ge i}$. 
\end{enumerate}
\end{lemma}
When the AG algebra $\mathcal{A}=\mathcal{R}/\mathcal{I}$ is a local ring then in general the dual generator $f$ is not homogeneous, and the Hilbert function $H(\A)$ is not in general symmetric: however the associated graded algebra $\A^\ast$ has a filtration whose successive quotients are reflexive $\A^\ast$ modules, and $H(\A)$ has a corresponding ``symmetric decomposition'' \cite{I1,IM}. When the AG algebra $A$ is (perhaps non-standard) graded, then the dual generator $f\in \mathfrak{D}$ may be taken homogeneous in a suitable grading of $\mathfrak{D}$, and (iii) implies that the Hilbert function $H(A)$ is symmetric about $j/2$, where $j$ is the socle degree of $A$.
\begin{example} 
We let $R=\mathsf{k}[x,y]$, with weights $\mathsf{w}(x,y)=(3,1)$, and consider the complete intersection algebra $A=R/(x^2-y^6, xy)$, then $I^\perp=R\circ f$ where $ f=X^2+Y^6$, which is homogeneous in the analogous grading of $\mathfrak{D}$. We have $H(A)=(1,1,1,2,1,1,1)$.
\end{example}
\begin{lemma}\label{JT2lem} 
Let $\mathcal{A}=\mathcal{R}/I$ be local Artinian Gorenstein of socle degree $j$ with Macaulay dual generator $F\in \mathfrak{D}$ and let $\ell\in {\mathfrak{m}}_\mathcal{A}$. The conjugate $(P_\ell)^\vee$ to the Jordan type $P_\ell$ satisfies
\begin{equation}\label{dualeq}
(P_\ell)^\vee =\Delta\bigl(\dim_\mathsf{k} \mathcal{A},\dim_\mathsf{k}\mathcal{A}(1),\dots,\dim_\mathsf{k}\mathcal{A}(i),\ldots ,\dim_\mathsf{k}\mathcal{A}(j)\bigr)
\end{equation}
where $\mathcal{A}(i)=\mathcal{R}/(I:\ell^i)=\mathcal{A}/(0:\ell^i)=\mathcal{R}/\Ann(\ell^i\circ F)$.
\end{lemma}
\begin{proof} 
This result is standard and follows from Lemma \ref{JTlem}. See also \cite[Lemma 3.60]{H-W}.
\end{proof}\par
For a generalization of Macaulay dual over a field $\sf k$ (as here) to Macaulay dual over any base, see S. Kleiman and J. O. Kleppe \cite{KlKl}. Several authors have studied the Artinian Gorenstein algebras arising from polynomials attached to combinatorial objects such as a matroid
\cite{MN,NSW}.
\subsection{Jordan degree type and Hilbert function.}\label{Jdegsec}
We next introduce a finer invariant than Jordan type of an $A$-module $M$, the Jordan degree type. We will define the Jordan degree type for graded modules $M$ over a graded ring $A$; although there is an analogue for local Artinian algebras $\A$, the results are less compelling, so we omit the local version.\vskip 0.2cm\par
\begin{definition}\label{degreeJT-def}Jordan degree type, contiguous degree type of $H$, order on JDT.
\vskip -0.2cm
\begin{enumerate}[(i).]
\item {\emph {Jordan degree-type of $M$.}} We give several equivalent notations for Jordan degree type.
Let $A$ be a graded Artinian algebra, let $M$ be a finite graded $A$-module, and let $\ell\in \mathfrak{m}$ be any homogeneous element.\par
(a). Suppose $P_{\ell,M}=(p_1,\ldots ,p_s)$, and write $M$ as a direct sum $M=\langle S_1\rangle\oplus \cdots \oplus \langle S_s\rangle$ of cyclic $\mathsf{k}[\ell]$-modules generated by $\ell$-strings of the form $S_k=\{z_k,\, \ell z_k,\ldots, \ell^{p_k-1}z_k\}$ satisfying $\ell^{p_k}z_k=0$, as in Definition \ref{JTdef}, and let $\nu_k$ be the order of $z_k$. For any ${k,k'\in\{1,\ldots,s\}}$ if ${k<k'}$ and ${p_k=p_{k'}}$, we assume ${\nu_k\le\nu_{k'}}$. By Lemma \ref{JThomoglem}(iv) the sequence of pairs of integers 
\begin{equation}\label{JDT1eq}
\mathcal{S}_{\ell,M}=\bigl((p_1,\nu_1),\ldots, (p_s,\nu_s)\bigr) 
\end{equation}
is an invariant of $(M,\ell)$, that we term the \emph{Jordan degree type} of $M$ with respect to $\ell$. \par\noindent
{\it{Notation}}. With $\ell$ understood, we will denote the pair $(n,\nu)$ by 
\begin{equation}\label{string2eq}
\mathsf{n}_\nu =\text{ a string -- a cyclic $\mathsf{k}[\ell]$-module -- of length $n$ beginning in degree $\nu$. }
\end{equation}
Thus, $\mathcal{S}=(\mathsf{5}_0,\mathsf{3}_1,\mathsf{3}_1,\mathsf{1}_2)$ denotes a Jordan degree type consistent with the Hilbert function $H=(1,3,4,3,1)$.\vskip 0.2cm\par
(b). Denote by $P_{\ell,i}$ (or $P_{\ell,i,M}$) the partition giving the lengths of those strings of $m_\ell$ acting on $M$ that begin in degree $i$: that is $P_{\ell,i}=(p_k\mid \nu_k=i)$. We denote by $\mathcal{P}=\mathcal{P}_{\deg,\ell}$ or by $\mathcal{P}_{\ell,M}=\mathcal{P}_{\deg,\ell,M}$ (to specify the module $M$) the sequence
\begin{equation}\label{degreeJTeq}
\mathcal{P}_{\deg,\ell}=(P_{\ell,0},\ldots,P_{\ell,j-1}),
\end{equation}
which we also term the \emph{Jordan degree type} (JDT) of $\ell$. For example the JDT $\mathcal{S}=(\mathsf{5}_0,\mathsf{3}_1,\mathsf{3}_1,\mathsf{1}_2)$ can be written $\mathcal{P}=\bigl(P_0=(5),\, P_1=(3,3),\, P_2=(1)\bigr)$. Given such a JDT sequence $\mathcal{S}_{\ell,M}$ or $\mathcal{P}_{\deg,\ell,M}$ as in Equation \eqref{JDT1eq} or \eqref{degreeJTeq} we denote by $H(\mathcal{S})$ or $H(\mathcal{P})$ the sequence $H=(h_0,\ldots, h_j)$ where $h_i$ counts the number of beads (basis elements) of the strings of $\mathcal{S}$ having degree $i$: it is the Hilbert function of any module having JDT $\mathcal{S}$ or $\mathcal{P}$.\par
Given an $A$-module $M$, we will denote by $Q_{\ell,i}$ (or $Q_{\ell,i,M}$) the partition giving the lengths of those strings of $m_\ell$ acting on $M$ that \emph{end} in degree $i$. We analogously define $\mathcal{Q}_{\ell,M}$ the \emph{end Jordan degree type} to how we defined $\mathcal{P}_{\ell,M}$.\vskip 0.2cm\par
(c). We give an alternate notation for Jordan degree type, closer to T. Harima and J.~Watanabe's central simple modules (Definition \ref{CSMdef} below, see \cite{HW1.5} and also B.~Costa and R.~Gondim's \cite[Definition 4.1]{CGo}). Recall from Lemma \ref{JThomoglem}(i) that, given an element $\ell\in \mathfrak{m}_A=\oplus_{i=1}^j A_i$, we may write $M=\oplus_{k=1}^s \langle S_u\rangle$ where each $S_k$ is an  \emph{$\ell$-string} of $M$, a cyclic $\mathsf{k}[t]/(t^{p_k})$-submodule with generator $z_k$, where $t$ acts as $m_\ell$; the Jordan type $P_{\ell,M}=(p_1,\ldots, p_s)$ where $p_k=|S_k|$. We define $\mathcal{E}_{\ell,M}=\{e_i^n(M,\ell), i\in [0,j], n\in [p_s,p_1]\}$ where
\begin{align}
e_i^n(\ell)&=e_i^n(M,\ell)=\# \{ {\text {length-$n$ strings $S_k$, $\langle S_k\rangle\cong\mathsf{k}[\ell]/({\ell^n})$, in $M$ beginning in degree $i$.\}}}\notag\\
&=\#\{(p_k,\nu_k)\in \mathcal{S}_{\ell,M}\mid  p_k=n, \nu_k=i\}.\label{multiplicityeqn}
\end{align}
By Lemma \ref{JThomoglem}(iv) the integers $e_i^n(\ell)$ are an invariant of the pair $(M,\ell)$ and do not depend on the particular decomposition.\par
\item \emph{Contiguous degree type of $H$.}
Given a Hilbert function $H$ of an Artinian algebra, we define the \emph{contiguous Jordan degree type} $\P_{c,\deg}(H)$ to be the degree-type obtained from the bar graph of $H$ (similar construction to the continguous Jordan type $P_c(H)$ of Definition~\ref{def:relativeLef}). More precisely, let $H=(h_0,h_1,\ldots ,h_j)$ be a sequence of non-negative integers (the Hilbert function). We denote by $P_{\deg,i}(H)$ the partition having $[h_i-h_{i-1}]^+$ parts, each of which is the length of a contiguous string of the bar graph of $H$ beginning in degree $i$. The \emph{degree-type} of the sequence $H$ is the sequence $\P_{\deg}(H)$ of partitions
\begin{equation}\label{degree-typeHeq}
\P_{c,\deg}(H)=\bigl(P_{\deg,0}(H),P_{\deg,1}(H)\ldots ,P_{\deg,j}(H)\bigr).
\end{equation}
It is the stratification of the contiguous partition $P_c(H)$ by the initial degree of the bars. We may also write $\mathcal{S}_{c,\deg}(H)$ as the JDT associated to $H$ in the sense of Equation \eqref{string2eq}.
\item  We will say for Jordan degree types $\mathcal{P}, \mathcal{P}'$ with the same Hilbert function $H(\mathcal{P})=H(\mathcal{P}')$ that 
\begin{equation}\label{concatenateeq}
\mathcal{P}\le_c \mathcal{P}'
\end{equation} if the strings of $\mathcal{P}$ can be concatenated -- that is, combined -- so as to form $\mathcal{P}' $. For example, $\mathcal{S}=(\mathsf{3}_0, \mathsf{2}_3)\le_c \mathcal{S}' =(\mathsf{5}_0)$ (notation of Equations \eqref{JDT1eq},\eqref{string2eq}).
\item We say that $\ell$ has the \emph{relative degree-Lefschetz property} with respect to $H$ if $\mathcal{P}_{\deg,\ell}=P_{c,\deg}(H)$.
\item A \emph{truncation} $\mathcal S_{\ell,A,\le k}$ of the Jordan degree type $\mathcal S_{\ell,A}$ of a graded algebra $A$ to degree less or equal $k$ is its projection to 
$A/\mathfrak m_A ^{k+1}$.  That is, each pair $(p_i,\nu_i)\in\mathcal S_{\ell,A}$ is replaced by
\begin{equation}\label{truncationeqn}
(\min\{p_i,k+1-\nu_i\},\nu_i) \in\mathcal S_{\ell,A,\le k}, \text { (or is omitted if $\nu_i\ge k+1$)}.
\end{equation}
\item
Given two standard-graded algebras $A,B\in \G_T, A=R/I, B=R/J$ of the same Jordan type $P_{\ell,A}=P_{\ell,B}$ with respect to a fixed element $\ell\in R_1$ we say that the Jordan degree type $\mathcal S_{\ell,A}\ge \mathcal S_{\ell,B}$ if for each $k$, the partition associated to $\mathcal S_{\ell,A,\le k}$ is greater or equal to that associated to $ \mathcal S_{\ell,B,\le k}$ in the dominance partial order (Definition \ref{dominancedef}). 
\end{enumerate}\vskip 0.2cm
 For an example of (v), the Jordan degree type $\mathcal S=\left((3,0),(3,1),(3,2),(3,3)\right)$ has truncation $\mathcal S_{\ell,\le 4}=\left((3,0),(3,1),(3,2),(2,3)\right)$. The JDT $\mathcal S^\prime=\left((3,0),(3,1),(3,1),(3,2)\right)$ with $S^\prime_{\le 4}=\mathcal S^\prime$ satisfies $\mathcal S^\prime > \mathcal S$.  See Example \ref{2.30ex}.

\end{definition}\par\vskip 0.2cm\noindent

\begin{lemma}[Specialization of Jordan degree type]\label{JDTspeclem}\par\noindent
Fix $\ell\in R_1$ and let $A(w),w\in W\backslash w_0$ be a family of graded Artinian algebras in $G_T$ of constant Jordan type $P_\ell$, and constant Jordan degree type
$\mathcal S_{\ell,A(w)}=S_\ell$ for $w\not=w_0$.  Assume that the limit algebra $A(w_0)$ has the same Jordan type $P_\ell$. Then $\mathcal S_\ell\ge \mathcal S_{\ell,A(w_0)}$.\par
\end{lemma}
\begin{proof} The projection from $G_T$ to $G_{T\le k}$ forgetting the portion of the algebra in degrees $k+1$ and higher, is an algebraic morphism. If there is a specialization
of Jordan types, it needs to extend to a specialization of the Jordan types projected to $A(w)_{\le k}$. Now the condition of (i), is obtained by reading the Jordan types of the projections from the Jordan degree types of $A(w)$, then applying Corollary \ref{semicontcor} below about the semicontinuity of Jordan types in the dominance order. \par\noindent
\end{proof}

The following result is a consequence of J. Brian\c{c}on's ``vertical strata'' analysis of ideals in ${\sf k}[x,y]$ \cite{Bri}. See also \cite{Got}, \cite[\S 2]{Yam}, \cite[p.\ 6-7]{AIKY}.
\begin{lemma}\label{cod2JDTlem}
When $A$ is a standard graded algebra of codimension two, and  has Jordan type $P_{\ell,A}=(p_1,p_2,\ldots, p_s)$ with respect to an element $\ell\in A_1$ then the Jordan degree type satisfies
\begin{equation}\label{JDTht2eq}
\mathcal S_{\ell,A}=\bigl((p_1,0), (p_2,1), \ldots, (p_i,i-1),\ldots ,(p_s,s-1)\bigr).
\end{equation}
\end{lemma}
\begin{proof} Let $A=R/I$, $I$ graded. Supposing $P_{\ell,A}=P$, then replacing $\ell$ by $x$ (change of basis), and using degree-lex order $1<x<y<\cdots <x^i<x^{i-1}y<\cdots <y^i<\cdots$ we may project $I$ to its initial monomial ideal $E_P$, which satisfies 
\begin{equation}\label{EPeqn}
E_P=(x^{p_1},yx^{p_2-1},\ldots,y^{i-1}x^{p_i-1},\ldots ,y^{s-1}x^{p_s-1},y^s).
\end{equation}
The Jordan degree type of $I$ and $E_P$ are the same, as the projection to initial form 
fixes degree: and the JDT of $E_P$ is $\mathcal S_{\ell,A}$ of Equation \eqref{JDTht2eq}.
\end{proof}

The beginning idea of the next example is that when $I\subset R$ is a monomial ideal defining $A=R/I$, and $z\in R_1$ there is a kind of ``order ideal'' of $z$-strings: that is, if $\mu $ is a monomial generator of a length-$t$ $z$-string of $A$ and if $\nu$ divides $\mu$ then $A$ has a $z$-string with generator $\nu$ whose length is at least $t$.  For the first algebra $A$ we began with a length-$3$ string $xyW$ where $W=\{ \langle 1,z,z^2\rangle\}$. For the second algebra $B$ we began with a length-$3$ string $x^3W$. Our subsequent discussion of structure/components involves more general AG algebras -- where $I$ is non-monomial -- and  shows that for $P=(3^4,1^4)$ the locus $G_{T,P}\subset G_T\times \mathbb P(R_1), T=(1,3,5,4,2,1)$ of pairs $(A,\ell), A\in G_T,\ell\in \mathbb P(R_1)$ (linear forms up to constant multiple) for which the Jordan type $P_{\ell,A}=P$, has several irreducible components.

\begin{example}[JDT not equivalent to JT in codimension three]\label{2.30ex}
Let $R={\sf k}[x,y,z]$ and set $T=(1,3,5,4,2,1)$. We will define $A=R/I_A$, $B=R/I_B$, each with Jordan type $P_{z,A}=P_{z,B}=(3^4,1^4)$ for the linear form $z$, but where 
\[
\mathcal S_A=\mathcal{S}_{z,A}=\bigl((3,0),(3,1)^2,(3,2);(1,2),(1,3),(1,4),(1,5)\bigr),
\] 
but 
\[
\mathcal S_B=\mathcal{S}_{z,B}=\bigl((3,0),(3,1),(3,2),(3,3);(1,1),(1,2)^2,(1,3)\bigr);
\]
after Equation \eqref{string2eq}, $\mathcal S_A=({\sf 3}_0,{\sf 3}_1,{\sf 3}_1,{\sf 3}_2,{\sf 1}_2,{\sf 1}_3,{\sf 1}_4,{\sf 1}_5)$ and $\mathcal S_B=({\sf 3}_0,{\sf 3}_1,{\sf 3}_2,{\sf 3}_3,{\sf 1}_1,{\sf 1}_2,{\sf 1}_2,{\sf 1}_3)$.
Let $A=\langle W,xW,yW,xyW, \{x^i, 2\le i\le 5\}\rangle $ where $W=\langle 1,z,z^2\rangle$; it is defined by the ideal $I_A=(y^2,x^2z,x^2y,z^3,x^6)$. Let $B=\langle W, xW,x^2W,x^3W,y,y^2,xy,y^3\rangle$, defined by the ideal $I_B=(yz,x^2y,xy^2, z^3,x^4,y^4)$. Note that both $A$ and $B$ are strong Lefschetz, as the Jordan type of ${x+y+z}$ is ${(6,4,3,2,1)=T^\vee}$.\vskip 0.2cm\par\noindent
{\bf Specialization of JDT, Structure of $G_{T,P}\subset G_T$}.
From Lemma \ref{JDTspeclem} one concludes that a family of AG algebras having JDT $\mathcal S_B$ cannot specialize to an algebra having JDT $\mathcal S_A$. We now show the converse. Let $A^\prime$ have JDT $\mathcal S_A$ and $B^\prime$ have JDT $ \mathcal S_B$ with respect to $z$. Evidently, there is an element $yz$ (for a suitable choice of $y$) in $I_2(B^\prime)$. We will now show that either\par
i. $I_2(A^\prime)$ is a perfect square. Then a family $I_2(A^\prime (w))=x_w^2$ cannot specialize to an $I_2$ which is composite,  {\it OR}\par
ii. $I_2(A^\prime)=\langle xy\rangle$ and $A^{'\vee}_5=(aY-bX)^{[5]}$, a pure power. But we will show in (iii) that $B^\vee_5$ is composite, and again this subfamily with JDT $\mathcal S_A$ cannot specialize to an algebra with JDT $\mathcal S_B$.\par
\begin{proof}[Proof that $A'$ satisfies (i) or (ii)]: The elements of $A'_2$ must include $zR_1$:  we may assume then that $I_2(A')=xy$ (up to change of basis $x,y$), or, case (i), that $I_2(A')= \langle u^2\rangle$ for some $ u\in \langle x,y\rangle$.\par\noindent
(ii)  Let us assume that $I_2(A')=xy$. Then $A'_2\supset x^2,y^2$, and the string beginning in degree two has generator $\alpha=x^2+cy^2$, and last element $z^2(x^2+cy^2)$. Then we claim $I_5(A')\supset (z,xy)\cap R_5$: this is so, since $(zx^k, zy^k)\in \langle I,\alpha z, \alpha z^2\rangle)$ for all $k\ge 2$ because of the JT $(3^4,1^4)$ and JDT $\mathcal S_A$. For example $zx^4\in I$ since there is no string $\{x^4, zx^4\}$ as ${\sf 2}_4$ does not occur. So we may assume that (after possible base change) $I_5=\langle (z,xy)_5,(ax+by)^5\rangle$. Then the dual $A{'^\vee}_5$ can be written as a pure $5$-th power, $(aY-bX)^{[5]}$. This shows (ii). \par\noindent
{\it Other ingredients:}\par
iii. $B{'^\vee}_5$ is composite. It needs to have a mixed $Z^{[2]}\beta$ term where $\beta\in {\sf k}_{DP}[X,Y]$. This cannot be part of a perfect power as $Z^{[2]}$ is the highest power of $Z$ that can occur.\par
This completes the proof that families of AG algebras with the Jordan degree type $\mathcal S_A$ and Hilbert function $T$ cannot specialize to an algebra in $G_T$ having Jordan degree type $\mathcal S_B$.  We have also shown that those algebras in $G_T$ of JDT $\mathcal S_A$ have two irreducible components, corresponding to whether $I_2$ is a perfect square or is composite.
\par iv. There are no further JDT associated to $P=(3^4,1^4)$ for $T$. Let $(A,z)\in G_{T,P}$.
First there must be ${\sf 3}_0$ and ${\sf 3}_1$ (else $I_2\supset xz,yz$, but $\dim_{\sf k}I_2=1$).
We must rule out $({\sf 3}_0,{\sf 3}_1,{\sf 3}_1,{\sf 3}_3)$ and $({\sf 3}_0,{\sf 3}_1,{\sf 3}_2,{\sf 3}_3)$ as the $3^4$ part of JDT. The former requires $I_3\supset \langle zx^2,zxy,zy^2\rangle$ but then $z^2A_3\subset I$ and there is no room for a string ${\sf 3}_3$. The latter requires $\langle yz\rangle=I_2$ (for suitable $y$); and $A_2=\langle y^2,x^2,xy,z^2, zx\rangle$. It follows that one of the two strings ${\sf 3}_2,{\sf 3}_2$ must begin with $\alpha=ay^2+byx$, but then $z\alpha\in (zy)\in I$, a contradiction.\par
We have shown that there are exactly two JDT $\mathcal S_A, \mathcal S_B$ associated to the pair $(P,T), P=(3^4,1^4), T=(1,3,5,4,2,1)$, and that a family of algebras having one of the JDT cannot specialize to an algebra having the other JDT.  We have also shown that the JDT locus $\mathcal S_A$ has two components. This implies that the locus of pairs $(A,\ell\in A_1)\subset G_{T,P}$ has three irreducible components.
\end{proof} 

\end{example}\par\noindent
{\bf Comments.}
When $H(A)$ is not unimodal, the relative degree-Lefschetz property is the closest one can get to strong Lefschetz (Proposition \ref{PJTgenlem}).\par
In the Jordan type $P_\ell$ the part $n$ occurs with multiplicity $\sum_i e^n_i$ where $e^n_i$ is from Equation~\eqref{multiplicityeqn} so 
\begin{align}
P_{\ell,i}&=(\ldots ,n^{e^n_i},\ldots) \text { and }\notag\\
P_\ell&=(\ldots, n^{\sum_ie^n_i},\ldots).\label{P-degreetypeeq}
\end{align}
Evidently, the degree-type $P_{c,\deg}(H)$ of the Hilbert function determines $H$, so it is equivalent to $H$ -- in contrast to $P(H)$ or even $P_c(H)$ which, when $H$ is non-unimodal, may not determine $H$ (see Example \ref{HPex} below).  Here are two more examples of concatenation: first, using the  $\mathcal{S}_{c,\deg}(H)$ notation, $({\color{blue}{\mathsf{2}_2}},{\color{red}{\mathsf{2}_4}})\le_c (\mathsf{4}_2)$ of Hilbert function $H=(0,0,{\color{blue}{1,1}},{\color{red}{1,1}})$; and, second, $(\mathsf{2}_2,{\color{blue}{\mathsf{2}_2}},{\color{red}{\mathsf{3}_4}})\le_c (\mathsf{5}_2,\mathsf{2}_2)$ of Hilbert function $H'=(0,0,{\color{blue}{2,2}},{\color{red}{1,1,1}})$.\par
Note that $P_{\deg,\ell}\le_c P_{c,\deg}(H)$ implies that $P_\ell\le P_c(H)$, but not vice versa.
Note also that the contiguous Jordan degree type $P_{c,\deg}(H)$ determines $H$, so is equivalent in information content to giving $H$. Recall that 
Theorem~\ref{thm:contig} bounded the Jordan type by the contiguous Hilbert function. We prove a refinement to JDT in the special case $A$ is standard-graded.
\begin{proposition}[Jordan degree type bound]\label{PJTgenlem} Let $M$ be a finite-length graded module over a standard graded Artinian algebra $A$. Let $\ell\in A_1$ be any linear form and let $\mathcal{P}_{\deg,\ell}=\mathcal{P}_{\deg,\ell,M}$ be its Jordan degree type as in Equation \eqref{degreeJTeq}. Then in the concatenation partial order,
\begin{equation}\label{PdegtypeHPeq}
\mathcal{P}_{\deg,\ell,M}\le_c \P_{c,\deg}\bigl(H(M)\bigr).
\end{equation}
Let $M$ be a fixed finite-length graded $A$-module, then there is a generic linear Jordan degree type 
$\mathcal{P}_{\deg}(M)=\mathcal{P}_{\deg, \ell}(M)$ for $\ell\in U$, a dense open set of 
$A_1$.
\end{proposition}
\begin{proof}[Proof] 
The first statement is evident. For the second, begin with the generic Jordan type $P(M)$, consider the set of highest length parts of $P(M)$, and their initial degrees: that these initial degrees are minimal is an open condition on $\ell$. Now fix this open set $U_1$ and go to the set of next highest-length parts for $\ell\in U_1$, forming an open $U_2$. In a finite number of steps one shows that $P_{deg}(M)$ is achieved for an open dense set $U$ of $\ell\in A_1$. 
\end{proof}

For a graded Artinian algebra $A$, knowing the Jordan degree type $\mathcal{P}_{\deg,\ell}$ is equivalent to knowing the Hilbert functions with respect to $\mathfrak{m}_A$ of the central simple modules (CSM) defined by T. Harima and J. Watanabe in \cite{HW}. We now explain this.
\begin{definition}[Central simple module]\label{CSMdef} 
Let $A$ be a graded Artinian algebra. Suppose that $\ell\in A$ satisfies $\ell^c\not=0$, $\ell^{c+1}=0$. The central simple modules defined by T. Harima and J.~Watanabe in \cite{HW} are the nonzero factors in the series,
\begin{equation}\label{csmeq}
A=(0:\ell^c)+(\ell)\supset (0:\ell^{c-1})+(\ell)\supset \cdots \supset (0:\ell)+(\ell).
\end{equation}
Let $s_\ell$ be the number of distinct parts of $P_\ell$. We denote by $V_{i,\ell}$ for $1\le i\le s_\ell$ the $i$-th central simple module: the vector space span of the
initial elements of length-$f_i$ strings of the multiplication $m_\ell$ on $M$:
\begin{equation}\label{Vieqn}
V_{i,\ell}\cong \langle (0:\ell^{f_i})+(\ell)\rangle\mod \langle (0:\ell^{f_i-1})+(\ell)\rangle. 
\end{equation}
Let $W_i=\oplus_{k=0}^{f_i-1}\ell^kV_{i\ell}$, a direct sum of those cyclic submodules $\langle S_u\rangle$ from Equation \eqref{stringeq} corresponding to length-$f_i$ strings.  Then, evidently $M=\oplus_{i}W_i.$
\end{definition}
\noindent
We have for the dimension of the degree-$u$ component of $V_{i,\ell}$,
\begin{equation}\label{Veeq}
\dim_\mathsf{k} (V_{i,\ell})_u=e^{f_i}_u(\ell),\text{ and }\dim_\mathsf{k} V_{i,\ell}=\sum_u e_u^{f_i}(\ell).
\end{equation}
This definition of CSM is perfectly general, and does not require $A$ to be graded nor $\ell\in \mathfrak{m}_A$ to be special. See \cite[\S 3.1]{H-W}, \cite[\S 5.1]{HW}; the latter treats non-standard grading. \par
The following Lemma connects the Jordan degree type and the central simple modules. The proof is straightforward. Besides \cite{HW} and \cite[\S 4.1]{H-W} see also \cite[Corollary 2.7]{BoI} for an approach using the symmetric decomposition with respect to the principal ideal $(\ell)$.
\begin{lemma}\label{HilbVillem} 
The set $\mathcal{H}_\ell=\bigl(H(V_{1,\ell}),\ldots ,H(V_{s_\ell,\ell})\bigr)$ of Hilbert functions of the central simple modules $V_{i,\ell}$, is equivalent to the Jordan degree type of $\ell$, rearranged according to the lengths $f_1>f_2>\cdots >f_{s_\ell}$.  In particular $H(V_{i,\ell})_u=( \ldots ,e^{f_i}_u(\ell), \ldots ) $.
\end{lemma}
\begin{example}[Degree types of Hilbert functions]
\label{HPex} We illustrate that the contiguous Hilbert function partition $P_c(H)$ can distinguish two Hilbert functions of the same partition $P(H)$; also, the Jordan degree-type $P_{c,\deg}(H)$ can distinguish two Hilbert functions of the same contiguous partition $P_c(H)$. For $H=(1,3,2,3,3,1)$ we have $P(H)=(6,4,3)$, $P_c(H)=(6,4,2,1) $ and $P_{c,\deg}(H)=\bigl((6_0), (4_1,1_1),(2_3)\bigr)$. For $H'=(1,3,1,3,3,2)$ we have $P(H')=(6,4,3)$, $P_c(H')=(6,3,2,1,1)$, and ${P}_{c,\deg}(H')=\bigl((6_0), (1_1,1_1),(3_3,2_3)\bigr)$. For $H'' =(1,3,3,2,1,3)$ (not pictured) we have $P(H'' )=(6,4,3)$, $P_c(H'' )=P_c(H' )=(6,3,2,1,1)$, but $P_{c,\deg}(H'')=\bigl((6_0),(3_1,2_1), (1_4,1_4)\bigr)$. This last illustrates that the Hilbert function is not determined by the contiguous partition $P_c(H)$, but is determined by the contiguous degree-type $P_{c,\deg}(H)$. We chose these examples, even thought they are not standard-graded, because their diagrams in Figure \ref{PartHilbfig} are particularly transparent. Note that by F.H.S. Macaulay's inequalities for Hilbert functions, none of the Hilbert functions $H,H',H'' $ above can occur for a standard graded algebra. We next give some similar comparisons whose Hilbert functions do occur for a standard graded algebra.
\vskip 0.2cm\noindent
\textbf{Standard graded Hilbert functions}: we compare contiguous partitions $P_c(H)$ and contiguous degree-types $P_{c,\deg }(H)$  (Equation \ref{degree-typeHeq}). Here $\{H_i, i\in [1,8]\}$ denotes Hilbert functions.\par
Same partition $P(H)=H^\vee$ but different $P_c(H)$: take $H_1=(1,3,6,4,5,6,2)$, $H_2=(1,3,4,5,6,6,2)$ then $P_c(H_1)=(7,6,5,4,2,1^3)$,  $P_c(H_2)=(7,6,5,4,3,2)$.\par
Same $P_c(H)$ but different Hilbert function (so different $P_{c,\deg}(H)$): then compare
$H_3=(1,3,5,7,6)$ with $H_4=(1,3,5,6,7)$. The Hilbert functions $H_5=(1,3,6,10,9,11,12,10)$ and $H_6=(1,3,6,10,9,10,11,12)$ have the same $P_c(H)$, but their degree types $\mathcal{S}(H)$ differ in having the subsequence $(1_3,2_5,1_6)$ for $H_5$ but $(1_3,2_6,1_7)$ for $H_6$. All of $H_1,H_2,\ldots ,H_6$ satisfy the Macaulay growth conditions. A simpler example compares  $H_7=(1,3,3,4,5)$ with $H_8=(1,3,4,5,3)$: since they are unimodal, and $P(H_7)=P(H_8)$ we have also $P_c(H_7)=P_c(H_8)$, but, of course, $H_7\not=H_8$ so $\mathcal{S}_{c,\deg}( H_7)=(\mathsf{5}_0,\mathsf{4}_1^2,\mathsf{2}_3,\mathsf{1}_4)$ is not $\mathcal{S}_{c,\deg}( H_8)=(\mathsf{5}_0,\mathsf{4}^2_1,\mathsf{2}_2,\mathsf{1}_3)$.\par
\end{example}
\begin{question} For which Hilbert functions $H$ can we find graded Artinian algebras $A$
with $H(A)=H$ and such that for a generic $\ell \in A_1$ we have, in increasing level of refinement,
\begin{equation}
P_{\ell,A}=P(H), \text { or } P_{c,\ell}=P_c(H), \text { or } P_{c,\deg,\ell}=P_{c,\deg}(H)?
\end{equation}
Note that a graded algebra $A=\mathsf{k}[x,y,z]/I$ of Hilbert function $H(A)=(1,3,3,4)$ cannot be even weak Lefschetz as the minimal growth from degree 2 to degree 3 implies that $I_2=a_1(x,y,z)$ for some $a_1\in A_1$, so multiplication by an $\ell\in A_1$ cannot be injective from $A_1$ to $A_2$.
\end{question}
There has been some study of a different question, namely, which Hilbert functions $H$ force $A$ to have one of the Lefschetz properties \cite{MiZa1,ZaZy}. See also \cite{MiNa}.
\begin{example} We first construct the idealization of $B=\mathsf{k}[x,y,z]/\mathfrak{m}^3$ of Hilbert function $H(A)=(1,3,6)$ with its dual giving an algebra $A$ of Hilbert function $H(A)=(1,3,6,0)+(0,6,3,1)=(1,9,9,1)$ (see also \S \ref{idealizationsec} below). We may take a Macaulay dual generator $F=\sum_{i=1}^6 U_i\mu_i$ where $\mu_i$ runs through the six monomials of degree 2 in $X,Y,Z$, in lexicographic order, and $U_1,\ldots, U_6$ are variables, as
\[
F=U_1X^{[2]}+U_2XY+U_3XZ+U_4Y^{[2]}+U_5YZ+U_6Z^{[2]}.
\]
Take $R=\mathsf{k}[x,y,z,u_1,\ldots, u_6]$ acting by contraction on $S=\mathsf{k}_{DP}[X,Y,Z,U_1,\ldots, U_6]$. Then $A_2^\vee=R_1\circ F$ satisfies
\begin{equation}
A_2^\vee=\langle X^{[2]},XY,XZ,Y^{[2]},YZ,Z^{[2]}, U_1X+U_2Y+U_3Z, U_2X+U_4Y+U_5Z, U_3X+U_5Y+U_6Z\rangle
\end{equation}
while $A_1^\vee=S_1=\mathsf{k}_{DP}[X,Y,Z,U_1,\ldots ,U_6]$. We may take (after scaling) as a generic linear form $\ell=x+y+z+\sum_{i=1}^6 u_i$. Then the rank of $m_\ell: A_1\to A_2$ is by duality the same as that of $m_\ell: A_2^\vee\to A_1^\vee$. But $m_\ell: A_2^\vee\to A_1^\vee$ takes a $6$ dimensional space $\langle X^{[2]},XY,XZ,Y^{[2]},YZ,Z^{[2]}\rangle $ to the $3$-dimensonal space $\langle X,Y,Z\rangle$ so has a three dimensional kernel. Thus, using symmetry of Proposition \ref{symprop} we have that $P_\ell=(4,2^5,1^6)$ and the Jordan degree type is $\mathcal{P}_\ell=(4_0,2^5_2,3_1,3_2)$.\par
Evidently, we may make similar examples using, say, a general-enough four-dimensional subspace $\mu_1,\ldots,\mu_4$ of $\mathsf{k}[x,y,z]_2$, and $F=\sum \mu_i U_i$, and finding there is a one-dimensional kernel, this gives an algebra $A$ of Hilbert function $(1,7,7,1)$ where $\ell=x+y+z+u_1+\cdots +u_4$ has Jordan type $(4,2^5,1,1)$ and Jordan degree type $\mathcal{S}_\ell=(\mathsf{4}_0,\mathsf{2}^5_1,\mathsf{1}_1,\mathsf{1}_2)$. Likewise we can determine $A$ of Hilbert function $H=(1,8,8,1)$ with $\ell$ a generic enough linear form having Jordan type $P_\ell=(4,2^5,1^4)$ and Jordan degree-type $\mathcal{S}_\ell=(\mathsf{4}_0,\mathsf{2}^5_1,\mathsf{1}^2_1,\mathsf{1}^2_2)$.
\end{example}
{\tiny
\begin{figure}
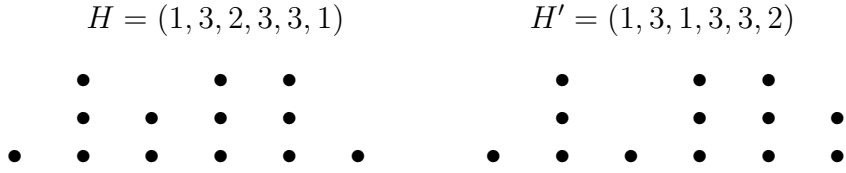

$\qquad\qquad\quad \qquad H=(1,3,2,3,3,1)\qquad\qquad\qquad H'=(1,3,1,3,3,2)$\vskip 0.3cm
$\qquad\qquad
\begin{array}{ccccccccccc}
&&\bullet&&&&\bullet&&\bullet&&\\
&&\bullet&&\bullet&&\bullet&&\bullet&\\
\bullet&&\bullet&&\bullet&&\bullet&&\bullet&&\bullet\\
\end{array}\qquad\quad
\begin{array}{ccccccccccc}
&&\bullet&&&&\bullet&&\bullet&&\\
&&\bullet&&&&\bullet&&\bullet&&\bullet\\
\bullet&&\bullet&&\bullet&&\bullet&&\bullet&&\bullet\\
\end{array}$
\vskip 0.4cm
\caption{$P_c(H)=(6,4,2,1); \,\text { and }\,\, P_c(H')=(6,3,2,1,1)$.}\label{PartHilbfig}\quad\qquad\qquad\quad
\end{figure}
}
\vskip 0.2cm
B. Costa and R. Gondim show the following Symmetry Proposition using the numerics of central simple modules \cite[Lemma 4.6]{CGo}: thus, this result along with the string diagrams they introduce \cite[Remark 4.9]{CGo} is essentially a felicitous and visual interpretation of the work of T. Harima and J.~Watanabe in \cite{HW,HW3}.\footnote{It was in a work group at the conference
Lefschetz Properties and Jordan Type at Levico, Italy, 25-29 June 2018 that Rodrigo Gondim had presented the symmetric string diagrams of \cite{CGo}; we realized a few days later in a discussion with Alessandra Bernardi and Daniele Taufer at University of Trento that the Jordan degree type we introduce here is a natural context for understanding this symmetry, leading to the Proposition.}
\begin{proposition}[Symmetry of Jordan degree type for graded AG algebras]\label{symprop}
Let $A$ be a standard graded Artinian Gorenstein algebra of socle degree $j$, and let $\ell\in A_1$. Then the Jordan degree type is symmetric:  
The integers
$e^n_\nu$ from Equation \eqref{multiplicityeqn} satisfy
\begin{equation}\label{symmetryeeq}
e^n_\nu=e^n_{j+1-n-\nu}.
\end{equation}
In other words, the set of strings of $\mathcal{S}=\mathcal{S}_{\ell,A}$ of Equation \eqref{JDT1eq} and notation
\eqref{string2eq} satisfy
$\mathsf{n}_\nu\in \mathcal{S}\Leftrightarrow \mathsf{n}_{j+1-n-\nu}\in \mathcal{S},$
with the same multiplicity $e^n_\nu=e^n_{j+1-n-\nu}$.
\end{proposition}
\begin{proof}The homomorphism: $m_{\ell^{f_i-1}}: V_{i,\ell}\to  \ell^{f_i-1}V_{i,\ell}$ is an isomorphism of $A$-modules, so we have $H(V_{i,\ell})_u=H(\ell^{f_i-1}V)_{u+f_i-1}$, but from the exact pairing $A\times A\to \mathsf{k}: (a,b)\to \phi(ab)$ of Lemma \ref{dualgenlem}, we have that $H(\ell^{f_i-1}V_{i,\ell})_{j-u}=H(V_{i,\ell})_u$. We conclude $H(V_{i,\ell})_u=H(V_{i,\ell})_{j+1-f_i-u}$; taking $n=f_i$ and using Equation \eqref{Veeq} we obtain the result.
\end{proof}\par
The Proposition is also a consequence of symmetric decomposition of the associated graded algebra $\Gr_\ell(A)$ with respect to $\ell$. The reflexive $\Gr_\ell(A)$ module $Q_\ell(j+1-f_i)$ in the symmetric decomposition of $\Gr_\ell(A)$ has first graded component $V_{i,\ell}$ and last component $\ell^{j+1-f_i}V_{i,\ell}$ (see \cite[Corollary 2.7]{BoI}).  This symmetry, which essentially states that there is a 1-1 map between the strings $\mathsf{n}_u(\ell)$ and the strings $\mathsf{n}_{j+1-n-u}(\ell)$ of the Jordan degree type of an AG algebra, greatly restricts the possible Jordan types: see \cite{AIK,CGo} for examples.
\begin{example}\cite[Example 4.6, Figure 14]{AIK}. We let $H=(1,2,3,2,1)$ and consider the Jordan degree types for elements $\ell\in A_1$ for Artinian complete intersections $A=R/(f,g)$ of Hilbert function $H$.  These are
$\begin{array}{cccc}
(\mathsf{5}_0,\mathsf{3}_1,\mathsf{1}_2),&(\mathsf{5}_0,\mathsf{2}_1,\mathsf{2}_2),&(\mathsf{4}_0,\mathsf{4}_1,\mathsf{1}_2),&(\mathsf{3}_0,\mathsf{3}_1,\mathsf{3}_2),
\end{array}$ in the notation of Equation \eqref{string2eq}.
The partition $P=(3,3,1,1,1)$, which ostensibly could correspond to a symmetric Jordan degree type
$(\mathsf{3}_0,\mathsf{3}_2,\mathsf{1}_1,\mathsf{1}_2,\mathsf{1}_3)$, does not occur for a CI (Gorenstein) quotient of $R=\mathsf{k}[x,y]$. Thus, the symmetry condition of Equation \eqref{symmetryeeq}, although quite restrictive, is not sufficient to determine the Jordan types for Artinian Gorenstein algebras of given Hilbert function.
\end{example}
\begin{question} How does the degree type $\mathcal S_{\ell,M}$ or $\mathcal P_{\deg,\ell,M}$ behave under
\begin{enumerate}[(i).]
\item deformation of $\ell\in A_1$?
\item deformation of $M$ within the family of $A$-modules of fixed Hilbert function $H$?\par
\item tensor product  (see Corollary \ref{JDtensorcor})?
\item projection to the quotient $R/\mathrm{in} I$? Are there JDT that cannot occur for an Artin graded algebra of Hilbert function $T$ defined by a monomial ideal?
\end{enumerate}
Also, does the inequality of Equation \eqref{PdegtypeHPeq} of Proposition \ref{PJTgenlem} extend to finite length modules over a local Artinian $\mathcal{A}$, taking $\ell\in \maxA$?  
\end{question}
\begin{question} Can we extend the notion of Jordan-degree type of $(A,\ell)$ for graded algebras to a ``Jordan-order type'' for local algebras $\mathcal{A}$, with properties analogous to those of Lemma~\ref{JThomoglem}, replacing degree by order, and omitting homogeneity? In particular, concatenating the orders of all elements in a good set of strings for $(A,\ell)$ should give the Hilbert function of $A$.
\end{question} That there are issues in defining Jordan order type for a local algebra is illustrated in the following example. Recall from Equation \ref{ordereq}ff that the order $\nu(a)$ of $a\in A$ is the largest power of the maximum ideal $\m$ of $A$ such that $a\in \m^{\nu}$.
\begin{example} Let $A=\mathsf{k}\{x,y\}/I$, $I=(x^2-xy^2,x^4,y^4)=(x^2-xy^2,y^4,x^3y^2)$, variables each of weight $1$, with $\mathsf{k}$-basis the classes of $\{1,x,y,xy,y^2,xy^2,y^3,xy^3\}$ and Hilbert function 
$H(A)=(1,2,2,2,1)$. Consider $\ell=x$ for which $P_x=(3,3,1,1)$ and the Jordan strings:  $S_1=(1,x,x^2=xy^2)$ of orders $(0,1,3)$, $S_2=(y,xy,x^2y=xy^3)$ of orders $(1,2,4)$, $S_3=(x-y^2)$ and $S_4=(xy-y^3)$ of orders $(1)$ and $(2)$; the choice of the strings has been made so that $x^{p_i}m_i=0$ where $m_i$ is the cyclic generator  of $S_i$ (hence the choice of $m_3=(x-y^2)$, so $xm_3=0$).
There are three basis elements (of the form $x^km_i, 0\le k\le p_i-1$) in the strings having order one; also, there is  no $\mathsf{k}[x]$ linear combination of the strings with the property that the orders of the basis elements match the Hilbert function (compare with Lemma \ref{JThomoglem}(i,iv) for the graded case). \par
One ``solution'' might be to adjust the sense of order in the presence of previous strings: for example $S_3=\langle x-y^2\rangle$ would be considered to have \emph{adjusted order} two, as
all of $R_1$, in particular $x$, is already in $\langle S_1, S_2\rangle$. 
\end{example}
\subsection{Deformations and generic Jordan type.}\label{genericsec}
We assume that $\mathsf{k}$ is an infinite field when discussing either generic Jordan type or deformation. A \emph{deformation} of a local Artinian algebra $\A$ over $\F$ is a flat family $\mathcal{A}_t$, $t\in T$ of Artinian algebras with special fiber $\mathcal{A}_{t_0}=\mathcal{A}$; then ${\mathcal{A}}_{t_0}$ is a \emph{specialization} of the family ${\mathcal{A}}_t$. We note that an algebraic family ${\mathcal{A}}_t,t\in T$ of Artinian algebras over a parameter space is flat if the fibers ${\mathcal{A}}_t$, $t\in T$ have constant length.\footnote{ A flat deformation means that relations among generators at the special point extend to relations among generators at the general point (see \S I.3 ``Meaning of flatness in terms of relations'' in \cite{Artin}). The flatness of a reduced family of Artinian algebras over an algebraically closed field that has constant fiber length is well known, and noted by J. Brian\c{c}on in \cite{Bri}. See \cite[Proposition 8 p. 44]{Mu} and \cite[Exercise 5.8c p.125]{Hart}.}  If also, for $t\not=t_0$ the algebras ${\mathcal{A}}_t$ have constant isomorphism type, we will say ${\mathcal{A}}_t$, $t\not=t_0$ is a jump deformation of ${\mathcal{A}}_{t_0}$. We use the dominance partial order on partitions of $n$ (Definition \ref{dominancedef}). The following result is well known in other contexts.
\begin{lemma}\label{dominancelem} 
Let $V$ be a $\sf k$-vector space of dimension $n$. Let $T$ be a parameter variety (as $T=\mathbb A_1$). Let $ \ell(t)\in \Mat_n(V)$, $t\in T$ be a family of nilpotent linear maps, and let $P$ be a partition of $n$. Then the condition on Jordan types, $P_{\ell(t)}\le P$ is a closed condition on $T$.
\end{lemma}
\begin{proof}
This is straightforward to show from Lemma \ref{JTlem} and the semicontinuity of the rank of $\ell(t)^i$ (see \cite[Thm. 6.2.5]{CM}, or \cite[Lemma 3.1]{IKh})
\end{proof}\par
This result is not unrelated to V.I. Arnol'd \cite[\S 4.4. Theorem]{Arn}, which studies the versal deformation of a matrix $M$ having a single eigenvalue (for us, the eigenvalue is $0$), and gives its dimension as $p_1+3p_2+5p_3+\cdots $. His article discusses singularities related to the different Jordan loci in the deformation space of $M$, a generalization of ``bifurcations'', a topic we do not develop here. The centralizer of $M$ is given in \cite[\S VIII.2]{Ga}, and is basic to the study of the nilpotent commutator of $M$ -- see the discussion after Example \ref{2posetex}, and references there.
\begin{corollary}[Semicontinuity of Jordan type]\label{semicontcor} {\phantom{make space}}\begin{enumerate}[(i).]
\item Let $M_t$ for $t\in T$ be a family of constant length modules over a parameter space $T$. Then for a neighborhood $U_0\subset T$ of $t=0$, we have that the generic Jordan types satisfy $t\in U_0\Rightarrow P_{M_t}\ge P_{M_0}$.
\item Let ${\mathcal{A}}_t$, $t\in T$ be a constant length family of local or graded Artinian algebras. Then for a neighborhood $U_0\subset T$ of $t=t_0$, we have $t\in U_0\Rightarrow P_{{\mathcal{A}}_{t}}\ge P_{{\mathcal{A}}_{t_0}}$.
\item Let $\ell_t\in\mathcal{M}_n(\mathsf{k})$ for $t\in T$ be a family of $n\times n$ nilpotent matrices, and let $P_t$ be their Jordan type. Then there is a neighborhood $U\subset T$ of $t_0$ such that $P_t\geq P_{t_0}$ for all $t\in U$.
\end{enumerate}
\end{corollary}
Applying this result to the deformation from the associated graded algebra $\mathcal{A}^\ast$ to the local Artinian algebra $\mathcal{A}$ we have
\begin{corollary}\label{comparetypecor} 
Suppose that $\mathcal{A}$ is a local Artinian algebra with maximum ideal ${\mathfrak{m}}$ and $\ell\in {\mathfrak{m}}$. Then $P_{\ell}(\mathcal{A})\ge P_{\ell}(\mathcal{A}^\ast)$. 
\end{corollary}
\begin{proof} 
Consider the natural flat deformation\footnote{See \cite[Theorem 15.17]{Ei}. This was shown by M. Gerstenhaber \cite{Ger} but \cite[Chapter 5]{Fu} gives a history showing prior use by D. Rim in 1956 and D. Mumford in 1959. It is easy to show in the Artinian case using the constant-length-fiber criterion for flatness.} from $\mathcal{A}^\ast$ to $\mathcal{A}$. For $t\not=0$, $\mathcal{A}_t$ has constant isomorphism type: this is a jump deformation, so the open neighborhood $U$ of Corollary \ref{semicontcor} includes elements where $P_{\ell,\mathcal{A}_t}=P_{\ell,\mathcal{A}}$. 
\end{proof}\par
Corollary \ref{comparetypecor} gives the following sufficient condition for checking SLJT:
\begin{corollary}
\label{cor:SLJTAG}
If an element $\ell\in\mathfrak{m}$ is SL for the associated graded Artinian algebra $\A^*$, then $\ell$ is SLJT for $\A$. 
\end{corollary}
\subsubsection{Jordan type and initial ideal.}
We thank the referee for calling to our attention the work of A. Wiebe, connecting Lefschetz properties of ideals with those of the initial ideals.  We assume that $<$ is a term order (monomial order) on the monomials of $R$: that is, if $\mu<\mu' $ then $\nu\mu<\nu\mu' $ for each monomial $\nu$. The \emph{initial ideal}  in $I$ is the monomial ideal generated by the initial monomials of its elements, and $\gin_< I\subset R$ is the initial monomial ideal of $\alpha\circ I$ where $\alpha$ is a general enough element of $\Gl_r(\mathsf{k})=\Aut (R_1)$ acting on $I$.  A. Wiebe has shown the first two parts of the next lemma. Part (iii) is newly stated here, but it follows from the discussion after \cite[Proposition 2.8]{Wi}, based on A. Conca's \cite[Lemma 1.2]{Con}: we include the proof. See also \cite[\S 6.1.2]{H-W} for a further discussion in the context of $k-$Lefschetz properties, which we do not treat here. The reverse-lex order is $<_\mathit{revlex}$ and satisfies  $x_1>x_2>\cdots >x_r>x_1^2>x_1x_2>x_2^2>x_1x_3>\cdots $ (see \cite[Definition 6.12]{H-W}).
\begin{proposition}\label{Wiebelem}\cite{Wi} Let $I\subset R$ be an $\mathfrak{m}$-primary graded ideal.\begin{enumerate}[(i).]
\item \cite[Prop. 2.8]{Wi} Then $A = R/I$ has SLP (respectively, WLP) if and only if $ R/\gin_<(I)$ has SLP (respectively WLP)\par
\item \cite[Prop. 2.9]{Wi}  Let $J$ be the initial ideal of $I$ with respect to a term order. If $S/J$ has the weak (resp. strong) Lefschetz property, then the same holds for $R/I$.
\item Let $K=\gin_< I$. Then the generic Jordan type $P_A$ of $A=R/K$ is the same as that of $R/I$. For $K=\gin_{<_\mathit{revlex}} I$ in the reverse lex order, this is the Jordan type of $R/K$ with respect to $x_r$.
\end{enumerate}
\end{proposition}
\begin{proof}[Proof of (iii)] (See A. Wiebe \cite[Prop. 2.8ff]{Wi}).\footnote{There are some misprints in the statements and proofs of these results in \cite{Wi}, that are correct in the arXiv version we have referenced.} He writes that by generalizing Conca's proof \cite[Lemma 1.2]{Con}, using
$\ini_\mathit{revlex}( g I + ( r^k )) =\ini_\mathit{revlex} ( g I ) + ( x_r^{\,k} )$ for all $k\ge 1$ one obtains that the Hilbert function of $S/(J,x_r^{\,k})$ is equal to the Hilbert function of $S/(I,\ell^k)$ for a general linear form $\ell \in S$ and
all $k\ge 1$. By the theorem of A. Galligo ($\cha \mathsf{k}=0)$ and D. Bayer and M. Stillman (arbitrary characteristic), the ideal $J=\gin_< I$ is Borel-fixed (\cite[Theorem 2.4]{Wi}, proved in \cite[Theorem 15.20]{Ei}).  It follows that for the revlex order, the generic Jordan type of $R/J$ is $P_{x_r}(R/J)$.
\end{proof}\par
A. Wiebe's discussion brings in the upper semicontinuity of $\dim \bigl(R/(I,\ell^k)\bigr)_i$: attention to this semicontinuity in a more general setting shows Lemma \ref{dominancelem}. 
A. Wiebe used his results to give criteria for componentwise linear ideals to have the WLP or SLP in terms of the Betti numbers of a resolution; these criteria were extended to all $\mathfrak{m} $-full ideals by T. Harima and J.~Watanabe \cite{HW5}, and further studied by J. Ahn, Young Hyun Cho, and J. P. Park \cite{ACP}.\par
However, the Jordan type may be different for $A=R/I$ and for $R/\ini(I)$ when the latter is not strong Lefschetz.
\begin{example}\label{JTinex} Take $R=\mathsf{k}[x,y,z]$, consider the graded lex monomial order 
$1<x<y<z<x^2<xy<xz<y^2<yz<z^2< \cdots$ and consider the complete intersection $A=R/I$, $I=(x^2, xy+z^2, xz+y^2)$ of Hilbert function $H(A)=(1,3,3,1)$. The generic Jordan type of $A$ is $(4,2)=H(A)^\vee$, so $A$ is strong Lefschetz.  But $J=\ini I=(x^2,xy,xz,y^3,z^4)$, $B=R/J$ has Jordan type $(4,1,1)$ since for any linear form  $\ell$, $m_\ell:A_1\to A_2$ has kernel $x$. Thus, $B$ is not weak Lefschetz.
\end{example}
\subsubsection{Examples of deformation, generic Jordan type.}
Recall that a local Artinian algebra $\mathcal{A}$ is \emph{curvilinear} if $H(\mathcal{A})=(1,1,\ldots ,1_j,0)$, in which case $\mathcal{A}$ is isomorphic to $\mathsf{k}\{x\}/(x^{j+1})$. In the following example we illustrate that a local algebra $\A$ determined by $I=(x,y)^2$ in $R=\mathsf{k}[x,y]$ can be deformed to a curvilinear algebra.  This is a special case of J. Brian\c{c}on's celebrated result that the fiber of the Hilbert scheme $\Hilb^n(\mathbb A^2)$ over $(0,0)$ is irreducible, and is the closure of the curvilinear locus -- which for embedding dimension two is the locus where the ideal defining $I$ has an element of order one \cite{Bri}.\footnote{J. Brian\c{c}on proved his result over ${\mathsf{k}}=\mathbb C$; his proof extends to algebraically closed fields satisfying $\cha \mathsf{k}>n$: this was improved to $\cha\mathsf{k}>n/2$ by R. Basili \cite{Ba} and to 
all characteristics by A. Premet \cite{Pr2}; see also V.~Baranovsky's \cite{Bar}.}\par
\begin{example}\label{firstdefex}
Let $\mathcal{A}(t)=\mathsf{k}\{x,y\}/I_t$ where for $t\in \mathsf{k}$, $t\not=0$, we let $I_t=(tx-y^2,y^3)$ and where $\mathcal{A}(0)=\lim_{t\to 0}\mathcal{A}(t)=\mathsf{k}[x,y]/(x^2,xy,y^2)$ (since for $t\not=0$ the ideal $I_t\supset \{tx-y^2,xy,x^2,y^3\}$). For $t\not=0$ the local algebra $\mathcal{A}(t)$ is a complete intersection with Hilbert function $(1,1,1)$. This family specializes to $\mathcal{A}(0)$ which is non-Gorenstein, with Hilbert function $(1,2)$. Here for $t\not=0$, taking $\ell=x+y$, we have that the generic linear Jordan type $ P_{\ell,\mathcal{A}(t)}=(3)>(2,1)=P_{\ell,\mathcal{A}(0)}$.
\end{example}
In the next examples we use the Macaulay dual generator notation from Definition \ref{Macdualdef}. The examples use standard grading.
\begin{example}\label{seconddefex}
Let $B_t=\mathsf{k}[x,y]/\Ann F_t$, where ${F_t=tX^{[5]}+X^{[2]}Y}$.
Then for $t\in \mathsf{k}$, $t\ne0$, the algebra $B_t$ is a curvilinear complete intersection, as in the previous example, with Hilbert function $(1,1,1,1,1,1)$. The family specializes to $B_0=\mathsf{k}[x,y]/(y^2,x^3)$, also a complete intersection (CI), with Hilbert function $(1,2,2,1)$. Here $\ell=x+y$ determines the generic (also generic linear) Jordan type $ P_{\ell.B_t}=(6)>(4,2)=P_{\ell,B_0}$.
\end{example}
\begin{example}\label{thirddefex}
Let $B_t=\mathsf{k}\{x,y,z\}/\Ann F_t$, where $F_t=t^2X^{[3]}Y^{[2]}+tX^{[2]}YZ+XZ^{[2]}$. Then for $t\not=0$ the algebra $B_t=\mathsf{k}\{x,y,z\}/(tz-xy,y^3,x^4)$ is an Artinian complete intersection with Hilbert function $(1,2,3,3,2,1)$. The family specializes to the complete intersection $B_0=\mathsf{k}[x,y,z]/(y^2,z^2,x^3)$, which has Hilbert function $(1,3,4,3,1)$. Here for $t\not=0$ ${P_{B_t}=(6,4,2)>(5,3,3,1)=P_{B_0}}$.
\end{example} 
\begin{definition}[The poset $\mathfrak{P}_M$]\label{posetPMdef} 
Let $M$ be a finite-length $\mathcal{A}$ module for an Artinian algebra $\mathcal{A}$. We denote by $\mathfrak{P}_M$ the poset $\{P_\ell \mid \ell\in {\mathfrak{m}}_{\mathcal{A}}\}$, with the dominance partial order. We denote by $\mathfrak{P}_{i,M}$ the sub-poset $\mathfrak{P}_{i,M}=\{P_\ell \mid \ell\in ({\mathfrak{m}}_{\mathcal{A}})^i\}$. 
\end{definition}
There is a related poset $\mathfrak{Z}_M$ of loci within $\mathfrak{m}_A$ or $\mathfrak{m}_\A$ determined by the set of partitions $P_\ell$: this is a subposet of $\mathfrak{P}_M$ but may be a proper subset: see Example \ref{2posetex}.
\begin{example}\label{firstposetex} 
For ${\mathcal{A}}_t=\mathsf{k}\{x,y\}/(tx-y^2,y^3)$, $t\not=0$, from Example \ref{firstdefex}, we have $\mathfrak{P}_{{\mathcal{A}}_t}=\{P_y=(3),\, P_x=(2,1),\, P_0=(1,1,1)\}$. For ${\mathcal{A}}_0=\mathsf{k}[x,y]/(x^2,xy,y^2)$, we have $\mathfrak{P}_{{\mathcal{A}}_0}=\{(2,1),(1,1,1)\}$.
\end{example}
Recall that the elements of a local Artinian algebra $\A$ are parametrized by the affine space $\mathbb{A}^n$, $n=\dim_\mathsf{k}(\A)$ and with our assumption $\mathfrak{m}_{\A}$ is parametrized by the affine space $\mathbb A^{n-1}$. So $\mathfrak{m}_A$ is an irreducible variety. Since the rank of each power of $m_\ell$ acting on $M$ is semicontinuous, and since by Lemma~\ref{JTlem} these ranks determine the Jordan type of $m_\ell$ we have
\begin{lemma}[Generic Jordan type of $M$]\label{genericJTlem}
Given an $A$ or $\mathcal{A}$ module $M$, there is an open dense subset $U_M\subset\mathfrak{m}=\mathbb A^{n-1}$ for which $\ell\in U_M $ implies that the partition $P_\ell$ satisfies $P_\ell\ge P_{\ell'}$ for any other element $\ell'\in \mathfrak{m}$.\par
Likewise, if $\A$ admits a weight function $\w$, then for each weight $i$, there is a dense open set $U_{i,M}\subset\A_i(\w)$ for which $\ell\in U_{i,M}$ implies that $P_\ell\geq P_{\ell'}$ for any other $\ell'\in \A_i(\w)$.
\end{lemma}
\begin{definition}\label{genJTdef}
For $M$ a finite module over a local Artinian algebra $\A$ over $\F$, we define the \emph{generic Jordan type} $P_M$ by $P_M=P_\ell$ where $\ell$ is a generic element of the maximal ideal of $\A$. In the graded case, we may also define the \emph{generic degree $i$ Jordan type} for $M$ is $P_{i,M}=P_\ell$ for $\ell$ a generic element of $A_i$ (it is not defined when $A_1=0$); for $i=1$, we call $P_{1,M}$ the \emph{generic linear Jordan type}. 
\end{definition}
\par Evidently we have 
\[
P_M\geq P_{1,M}\geq\cdots\geq P_{j,M}.
\]
As we next see in Example~\ref{firstex}, when $A$ is a non-standard graded algebra the generic Jordan type $P_A$ may not equal $P_{lin,A}$ even when $A_1\not=0$.
\begin{question} 
Under what conditions on a graded module $M$ over a graded Artinian algebra $A$ does its generic Jordan type satisfy $P_M=P_{1,M},$ the generic linear Jordan type? In particular, let $A$ be a standard-graded Artinian algebra $A$ with unimodal Hilbert function. Is the generic linear Jordan type of $A$ always the same as the generic Jordan type of $A$? Proposition \ref{SLJTprop} shows this when the generic Jordan type is maximal.
\end{question}
Let $A$ be a non-standard-graded Artinian algebra. The Jordan type $P_\ell$ of a non-homogeneous element $\ell\in {\mathfrak{m}}_A$ may be the same as that would be expected for a strong Lefschetz element, even though $A$ may have no linear strong Lefschetz elements, so $A$ is not SL. This we first noticed on the following example of relative covariants proposed by the third author (see \cite[Example 3.7]{MCIM}). 
\begin{example}\label{firstex}
We let $R=\mathsf{k}[y,z]$, with weights $\mathsf{w}(y,z)=(1,2)$, and let $A=R/I$, $I=(yz,z^3,y^7)$, having $\mathsf{k}$-basis $A=\langle 1,y,y^2,z,y^3,y^4,z^2,y^5,y^6\rangle$ and having Hilbert function $H(A)=(1,1,2,1,2,1,1)$, with Macaulay dual $R\circ \langle Z^2,Y^6\rangle$. The only linear element of $A$, up to non-zero constant multiple is $y$, and the partition given by the multiplication $m_y$ is $P_y=(7,1,1)$, so $A$ is not strong-Lefschetz. However the non-homogeneous element $\ell=y+z$, has strings $\{1,y+z,y^2+z^2,y^3,y^4,y^5,y^6\}$ and $\{z,z^2\}$ so $P_\ell=(7,2)$, which is the maximum possible given $H(A)$, so $\ell$ has strong Lefschetz Jordan type (Definition \ref{SLJTdef}). The Macaulay dual generator for $A$ is $F=Z^{[2]}+Y^{[6]}$\par A related local algebra is $\mathcal{A}=\mathcal{R}/(yz,z^3+y^6)$, with weights $\mathsf{w}'(y,z)=(1,1)$, of Hilbert function $H(\mathcal{A})=(1,2,2,1,1,1,1)$; here the element $\ell'=(y+z)\in \mathfrak{m}_{\mathcal{A}}$ has Jordan type $(7,2)$, so $\mathcal{A}$ is strong Lefschetz. The associated graded algebra $\mathcal{A}^\ast$ with respect to $\mathfrak{m}_{\mathcal{A}}$ is $\mathsf{k}[y,z]/(yz,z^3,y^7)$, with standard grading; and $P_{\ell', \mathcal{A}^*}=P_{\ell',A}=(7,2).$ See also Theorem~\ref{thm:JtIneq} and Example \ref{1.36ex}.
\end{example}

\section{Constructions, examples, and commuting Jordan types.}\label{examplessec}
\subsection{Idealization and Macaulay dual generator.}\label{idealizationsec}
The principle of idealization, introduced by M. Nagata to study modules, has been used to ``glue'' an Artinian algebra to its dual, and so to construct graded Artinian Gorenstein (AG) algebras either having non-unimodal Hilbert functions; or with unimodal Hilbert functions but not having a strong or weak Lefschetz property \cite{St,BI,BoLa,MiZa,Bo,Bo2}. Examples of M. Boij \cite{Bo} show that the Hilbert functions of AG algebras may have an arbitrarily high number of valleys, with local maxima at assigned values, at the cost of increasing the embedding dimension. The Jordan types and Jordan degree types -- which are symmetric by Proposition \ref{symprop} -- of these examples have in general not been studied, and could be of interest. H.~Ikeda (H. Sekiguchi) and  H. Ikeda with J. Watanabe, and also M. Boij gave examples of Artinian Gorenstein (AG) algebra having unimodal Hilbert functions, but not satisfying even weak Lefschetz \cite{Se,IkW,Bo2}.  Similar examples involving a partial idealization, also not strong Lefschetz or not weak Lefschetz have been constructed more recently by R.~Gondim and G. Zappal\'{a} \cite{Go,GoZ} and by A. Cerminara, R. Gondim, G.~Ilardi, F. Maddaloni \cite{CGIM}. An example where $m_L$ has Jordan type strictly between weak and strong Lefschetz was already given in \cite[Section 5.4]{H-W}, referencing \cite{IkW}. \par
We will give here several idealization or partial idealization examples where we calculate the generic Jordan type. These idealizations that are Artinian Gorenstein can arise from a particular structure for a homogeneous Macaulay dual generator for the AG algebra (\S \ref{AAMDsec}). There is a far-reaching generalization by S.L. Kleiman and J.O. Kleppe to Macaulay duality over an arbitrary Noetherian base ring \cite{KlKl}, and it would be natural to ask how the notions of Jordan type and idealization studied below might generalize from a base field, as here, to a more general base. \vskip 0.2cm
M. Nagata \cite{Na} introduced the following definition of idealization, see also \cite[Definition~2.75]{H-W}; for some further developments and history see D. Anderson and M. Winters \cite{AnWi}. For characterization by dual generator see \cite{BI}, \cite[Theorem 2.77]{H-W} and for another approach to idealization see \cite[Lemma 3.3]{MSS}.
\begin{definition} 
Let $M$ be an $A$-module. The idealization of $M$ is the algebra $A\oplus M$ whose multiplication is given by $(a,m)\cdot (b,n)=(ab, an+bm)$. 
\end{definition}
The idealization makes $A\oplus M$ into a ring, in which the $R$-submodules of $M$ correspond to the subideals of $M$. A particular example is formed when $M=\Hom (R,\mathsf{k})$ the dual of an Artinian level ring $A$ (i.e., the socle of $A$ is in a single degree): then the idealization is an Artinian Gorenstein algebra. R.~Stanley and subsequently others used this construction to give examples of AG algebras having non-unimodal Hilbert function (Example~\ref{stanex}).
\subsubsection{Examples of Idealization and Jordan type.}
\vskip 0.2cm Our first is the Jordan type for the example of R. Stanley in codimension $13$.
See also \cite[Example 2.28]{BoI} where this was also calculated. We will find that the actual generic linear Jordan type is very far from the bound $P_c(H)$ given in Theorem \ref{thm:contig}.
We will use the notation $m^k$ to represent the partition $(m,\ldots ,m)$ with $k$ parts.
\begin{example}\label{stanex}\cite[Example 4.3]{St}
We let $R=\mathsf{k}[x,y,z]$ and $S=\mathsf{k}[x,y,z, u_1,\ldots u_{10}]$, $\mathfrak{D}=\mathsf{k}_{DP}[X,Y,Z]$ and $\mathfrak{F}=\mathsf{k}_{DP}[X,Y,Z,U_1,\ldots ,U_{10}]$. Stanley's example results from idealization of $R/{\mathfrak{m}_R}^4$ and its dual. We let $I\subset S$, $I=\Ann F$, $F=\sum U_i\Xi_i\in \mathfrak{F}_4$ where $\Xi_1,\ldots \Xi_{10}$ is a basis for $\mathfrak{D}_3\subset\mathfrak{D}$. Then $A=S/I$ has dual module $S\circ F$ satisfying
\begin{align}
S_1\circ F&=\langle R_1\circ F, \Xi_1,\ldots ,\Xi_{10}\rangle\notag\\
S_2\circ F&=\langle \mathfrak{D}_2, R_2\circ F\rangle\notag\\
S_3\circ F&=\langle X,Y,Z, U_1,\ldots ,U_{10}\rangle.
\end{align}
Consequently, $H(A)=(1,13,12,13,1)$ of length 40. Taking a general element $\ell\in S_1$, - up to action of $(\mathsf{k}^\ast)^{\times 13}, $ take $ \ell=x+y+z+u_1+\cdots +u_{10}$, it is straightforward to calculate
\begin{align*}
H\bigl(S/\Ann (\ell\circ F)\bigr)&=(1,9,9,1), &H(S/\Ann \bigl(\ell^2\circ F)\bigr)&=(1,6,1),\\
H\bigl(S/\Ann (\ell^3\circ F)\bigr)&=(1,1), &H\bigl(S/\Ann (\ell^4\circ F)\bigr)&=(1).
\end{align*}
By Lemma \ref{JT2lem} Equation \eqref{dualeq}, the conjugate
$(P_\ell)^\vee$ is the first differences $\Delta$ of the lengths of the modules $S/\Ann (\ell^k\circ F)$ for $0\le k\le 4$, so here
\begin{equation}
(P_\ell)^\vee=\Delta (40,20,8,2,1)=(20,12,6,1,1).
\end{equation} 
Thus, $P_\ell=(5,3^5,2^6,1^8)$ with $20$ parts in contrast to $P_c(H)=(5,3^{11},1^2)$. This is related to the following:\vskip 0.15cm
$\quad m_{\ell^2}\colon A_1\to A_3$ has rank 6 (from the five parts 3, and one part 5), but \par
$\quad m_\ell\colon A_1\to A_2$ has rank 9, and kernel rank 4. \vskip 0.15cm\noindent
By symmetry $m_\ell\colon A_2\to A_3$ also has rank $9$ and cokernel rank 4.\par\noindent
Note that the contraction $R_1\circ \mathfrak{D}_3=\mathfrak{D}_2$ takes a 10-dimensional space to a 6 dimensional space: thus any multiplication map $m_\ell: \mathfrak{D}_3\to\mathfrak{D}_2$ has kernel rank at least 4. The Jordan degree type of this $\ell$ is 
\begin{equation}\label{dJTexeq}
\mathcal{P}_\ell=\big( P_{\ell,0}=5, P_{\ell,1}=(3^5,2^3,1^4), P_{\ell,2}=(2^3), P_{\ell,3}=(1^4)\big).
\end{equation}
According to Theorem \ref{thm:contig} the maximum Jordan type of a multiplication map $m_\ell$ for $\ell\in \mathfrak{m}_A$ (non-homogeneous) consistent with the Hilbert function $H=H(A)$ would be $P_c(H)=(5,3^{11},1^2)$ with 14 parts for a linear form. For an element $\ell\in{\mathfrak{m}}$ the upper bound would be $P(H)= (5,3^{11},2)$ with 13 parts expected -- if a quadratic term in $\ell$ takes the kernel of the linear part $m_{\ell_1}$ on $A_1$ to an element of $A_3$ not in $\ell_1\cdot A_2 $. Here the actual generic linear Jordan type $P_\ell=(5,3^5,2^6,1^8)$ with 20 parts for $\ell\in A_1$ (see above) or even for $\ell\in \mathfrak{m}_A$ (verified for several random $\ell\in\mathfrak{m}_A$, by calculation in \textsc{Macaulay2}) is very far from these bounds; therefore, $A$ does not have the Lefschetz property relative to $H$ of Definition \ref{def:relativeLef}.
\end{example}
\vskip 0.2cm
R. Gondim applying work of T. Maeno and J. Watanabe \cite{MW} relating higher Hessians and Lefschetz properties, exhibited Gorenstein algebras $A$ with bihomogeneous dual generators of the form $F=\sum \mu_i\nu_i$, in $\mathfrak{F}=\mathsf{k}_{DP}[X,U]$, such that $A$ does not satisfy weak Lefschetz, or, sometimes, has generic Jordan type strictly between WL and SL \cite{Go}. Here are two examples from R. Gondim.\footnote{The first example is from R. Gondim's talk at the workshop ``Lefschetz Properties and Artinian algebras'' at BIRS on March 15 2016, at ``https://www.birs.ca/workshops/2016/16w5114/files/Gondim.pdf''. The second is a private communication from R. Gondim, following a discussion there with the first author.}
\begin{example}(R. Gondim) Consider the cubic $f\in\mathfrak{F}$
\begin{equation}
f =X_1U_1U_2+X_2U_2U_3+X_3U_3U_4+X_4U_4U_1.
\end{equation}
The associated algebra $A = R/I$, of Hilbert function $H(A)=(1,8,8,1)$ with $I =\Ann f$ does not have the WLP: the map $\ell:A_1\to A_2$ is not injective for any $\ell\in A_1$. The algebra $ A$ is presented by 28 quadrics:
\begin{align*}
I = &(\mathfrak{m}_x^2,u_1^2,u_2^2,u_3^2,u_4^2,u_1u_3,u_2u_4,x_1u_3,x_1u_4,x_2u_4,x_2u_1,x_3u_1,x_3u_2,x_4u_2,x_4u_3\\&
x_1u_1-x_2u_3,x_2u_2-x_3u_4,x_3u_3- x_4u_1,x_4u_4-x_1u_2
) .
\end{align*}
and has Jordan type $(4, 2^6, 1, 1)$ so strictly less than $H(A)^\vee = (4, 2^7)$. A random element (non-homogeneous) gives the same Jordan type (calculation in \textsc{Macaulay2}).
\end{example}
\begin{example} (R. Gondim) 
Let $F=XU^{[3]}+YUV^{[2]}+ZU^{[2]}V\in\mathfrak{F}=\mathsf{k}_{DP}[U,V,X,Y,Z]$. Consider $R=\mathsf{k}[u,v,x,y,z]$ and the algebra $A = R/I, I=\Ann F$, where
\[
I=\langle x^2, y^2, z^2, u^4, v^3, xy, xz, yz, xv, zv^2, yu^2, u^2v^2, u^3v, xu-zv, zu-yv\rangle.
\]
Since $A$ is a bigraded idealization it is easy to see that $H(A)=(1,5,6,5,1)$. Since the partial derivatives $x\circ F=U^{[3]}, y\circ F=UV^{[2]}$ and $z\circ F=U^{[2]}V$ are algebraically dependent, by the Gordan-Noether criterion (\cite{GorN,Go,MW}) the Hessian Hess$_F=0$. By the Maeno-Watanabe criterion (\cite{MW}, \cite[Theorem 3.76]{H-W}) this implies that $A$ fails the strong Lefschetz Property. On the other hand it is easy to see that $u+v$ is a WL element for $A$. \par
Since $A$ is not strong Lefschetz, the Jordan decomposition $P_\ell$ for $\ell=u+v+x+y+z$ (a generic-enough linear form) is by Theorem \ref{thm:JtIneq} less in the dominance order than the conjugate $H(A)^\vee=(5,3^4,1)$; since $A$ has WLP, $P_\ell $ has the same number of parts as $H(A)^\vee$, namely the Sperner number $H(A)_\max=6$. Since $\ell^4\not=0$ in $A$, the string $S_1=(1,\ell,\ell^2,\ell^3,\ell^4)$ so $P_\ell$ has a part 5; since $P_\ell<(5,3^4,1)$ and has 6 parts the only possibility is $P_\ell =(5,3,3,3,2,2)$, with Jordan degree type $\mathcal{P}_{\ell}=(\mathsf{5}_0,\mathsf{3}^3_1,\mathsf{2}_1,\mathsf{2}_2).$
By Proposition \ref{SLJTprop} since $A$ is standard graded, has unimodal Hilbert function, and is not strong Lefschetz, $A$ cannot have an element -- even non-homogeneous -- that has strong Lefschetz Jordan type.
\end{example}
R. Gondim gives many further examples, using special bihomogeneous forms. R. Gondim and G. Zappal\`{a} have determined further graded Gorenstein algebras that are non-unimodal, sometimes completely non-unimodal (with decreasing Hilbert function from $h_1$ to $h_{j/2}$, then increasing to degree $j-1$): they accomplish this by using properties of complexes to choose a suitable bihomogeneous dual generator $f\in \mathfrak{F}$ \cite{GoZ}. In a sequel work A. Cerminara, R. Gondim, G. Ilardi, F. Maddaloni study ``higher order'' $(d_1,d_2)$ AG Nagata idealizations determined by Macaulay dual forms $\sum\mu_ig_j$ where $\mu_i$ runs through the $d_1$ powers of one set $X_1,\ldots, X_r$ of variables, while $g_i\in \mathsf{k}[U_1,\ldots, U_s]_{d_2}$ have degree $d_2$: they show in specific cases that the idealization is not SL, but when $d_1\ge d_2$ they show it is WL \cite[Proposition~2.7]{CGIM}; their main results are related to geometric properties and ``simplicial'' Nagata polynomials (ibid. Theorem 3.5). A. Capasso, P. De Poi, and G. Ilardi generalize this work in \cite{CDI}.  J. Mccullough and A. Seceleanu use idealization and a subtle choice of base level algebra to construct a new infinite sequence of quadratic Gorenstein rings with, in general, non-unimodal Hilbert functions, that are non-Koszul, with non-subadditive minimal resolutions (\cite[Theorem 4.3]{McSe}):  they have not been studied for their Jordan types.  
\subsection{Tensor products and Jordan type.}\label{tensorprodsec}
It is well known that a graded Artinian algebra $A$ with symmetric Hilbert function is SL with Lefschetz element $\ell\in A_1$ if and only if $A$ carries an $\mathfrak{sl}_2$-representation where the raising operator $E$ in the standard basis of $\mathfrak{sl}_2$ is multiplication by $\ell$, and element of $A_1$, and the weight space decomposition agrees with the grading of $A$, see \cite[Theorem 3.32]{H-W}. Equivalently, such a pair $(A,\ell)$ is strong Lefschetz if  $P_\ell=H^\vee$ (Proposition \ref{prop:LefHilb}).
The well-known Clebsch-Gordan formula decomposes a tensor product of $\mathfrak{sl}_2$ representations into irreducibles; it is equivalent to the strong Lefschetz property of $\mathsf{k}[x,y]/(x^m, y^n).$ 
\begin{lemma}[Clebsch-Gordan]\label{Clebschlem} Assume $m\ge n$ are positive integers, and $\cha \mathsf{k}=0$ or $\cha \mathsf{k}\ge m+n-1$. Then the Jordan type of $\ell=x+y$ in $ \mathsf{k}[x,y]/(x^m, y^n)$ satisfies
\begin{equation}\label{Clebscheq}
P_\ell=(m+n-1,m+n-3,m+n-5,\ldots, m-n+1 ).
\end{equation}
\end{lemma}
To prove this it is sufficient to know the strong Lefschetz property of standard graded quotients of $\mathsf{k}[x,y]$ in characteristic zero or characteristic larger than the socle degree (Lemma~\ref{heighttwolem}). See also \cite[Theorem 3.29 and Lemma 3.70]{H-W}. As a consequence we have
\begin{proposition}
\label{tensorprop}
\cite[Prop. 3.66]{H-W}.\footnote{Although \cite{H-W} restricts to $\cha \mathsf{k}=0$, there is no change in showing it for $\cha \mathsf{k}=p$ large enough.} Let $z\in A$, $w\in B$ be two non-unit elements of Artinian local algebras $ A, B$.
Set $P_z=(d_1,d_2,\ldots, d_t)$ and $P_w=(f_1,f_2,\ldots f_s)$. Denote by $\ell=z\otimes 1+1\otimes w\in A\otimes_{\sf k} B.$ Assume $\cha \mathsf{k}=0$, or $\cha \mathsf{k}\ge \max\{d_i+f_j-1\}$. Then
\begin{equation}\label{tensorpropeq}
P_\ell=\oplus_{i,j}\oplus_{k=1}^{\min\{ d_i,f_j\}}(d_i+f_j+1-2k).
\end{equation}
Also, $\dim_\mathsf{k} \ker (\times \ell)=\sum_{i,j}\min\{d_i,f_j\}.$
\end{proposition}
Recall that we denote by $j_A$ the socle degree of $A$. The following corollary of Proposition~\ref{tensorprop} is not hard to show, and we leave the proof to the reader. With the additional assumption that $H(A)$ and $H(B)$ are unimodal (so $A,B$ are SL in the narrow sense of Remark \ref{rem:Narrow}) this Corollary is shown in
\cite[Theorem~3.34]{H-W}.
\begin{corollary}\label{SLcor}
Assume that $A,B$ are graded Artinian algebras with symmetric Hilbert functions and that $\cha \mathsf{k}=0$ or $\cha \mathsf{k}>j_A+j_B $. Then the element $\ell=z\otimes 1+1\otimes w\in A\otimes_{\sf k} B$ is SL if and only if $z$ and $w$ are both SL, respectively, in $A$ and in $B$.
\end{corollary}
For a different proof of Corollary \ref{SLcor}, resting on the connection between the strong Lefschetz property of $C$ and the weak Lefschetz properties of $C\otimes_{\sf k} \mathsf{k}[t]/(t^i)$ see T. Harima and J.~Watanabe's \cite[Theorem~3.10]{HW}. It is open whether $A\otimes_{\sf k} B$ is SL implies both $A$ and $B$ are SL, without a prior assumption on the Hilbert functions of $A,B$.\par
We may use Proposition \ref{tensorprop} to determine the Jordan types of other, special elements $\ell\in\mathfrak{m}_A$ for certain Artinian algebras $A$.
\begin{example}\cite[Cor. 0.4]{DIKS} Consider $\ell=x^2+y^2\in A= \mathsf{k}[x,y]/(x^3,y^3)$ and suppose $\cha \mathsf{k}\not= 2,3.$ Then $P_\ell=(3,2,2,1,1)$. Indeed here $P_{x^2}=(2,1)$ on $\mathsf{k}[x]/(x^3)$; likewise $P_{y^2}=(2,1)$ on $\mathsf{k}[y]/(y^3)$, and hence by Proposition \ref{tensorprop}
\[
P_\ell=(2+2+1-2,2+2+1-4)\oplus(2+1+1-2)\oplus(1+2+1-2)\oplus(1+1+1-2)=(3,2,2,1,1).
\]
We found that for $\cha \mathsf{k}\not= 2,3$ we could achieve all Jordan types of $\ell\in \mathfrak{m}_A$ from elements of the form $\ell=x^a+y^b$ or $\ell =x^a$ using Proposition \ref{tensorprop}, except for
$P_{x^2y}=(2,2,1^5)$ (\cite[Cor. 0.4]{DIKS}).
\end{example} 
The following is a special case of a family of examples due to J. Migliore, U. Nagel, and H.~Schenck \cite{MNS}: they show that without the assumption that the component Hilbert functions are symmetric, the tensor product in general will not preserve any Lefschetz properties.
\begin{example}
\label{ex:MNS}
With the standard gradings, set 
\[
A=B=\frac{\F[x,y]}{(x^2,y^2,xy)} \text { and } \ C=A\otimes_\mathsf{k}B\cong\frac{\F[x,y,z,w]}{(x^2,y^2,z^2,w^2,xy,zw)}.
\]
A $\F$-basis for the tensor product $C$ is $\{1,x,y,z,w,xz,yz,xw,yw\}$, of Hilbert function $H(C)=(1,4,4)$. For a general linear form $\ell=ax+by+cz+dw\in C_1$ the matrix for $\times\ell\colon C_1\rightarrow C_2$ with respect to that basis is given by
\[
M_\ell=
\begin{pmatrix}
c & 0 & a & 0\\ 
d & 0 & 0 & a\\ 
0 & c & b & 0\\ 
0 & d & 0 & b\\
\end{pmatrix}
\]
which has $\det(M)=0$ for every $a,b,c,d\in\F$. Therefore $C$ is not even WL, let alone SL. Also, since $C$ is standard graded, Proposition \ref{SLJTprop} implies that $C$ cannot have SLJT. On the other hand, $A$ and $B$ certainly have all of these properties.
\end{example}

\subsubsection{Jordan degree type and Clebsch-Gordan.}
Following the Definition \ref{degreeJT-def} Equation \eqref{string2eq} of Jordan degree type, we denote by $\mathsf{m}_s$ a string of length $m$ beginning in degree $s$, and by $\mathsf{n}_t$ a string of length $n$ beginning in degree $t$. We next specify the Jordan degree type of their tensor product $\mathsf{m}_s\otimes_{\sf k} \mathsf{n}_t$: our result refines the Clebsch-Gordan Lemma \ref{Clebschlem}.
\begin{proposition}\label{tpdJTprop} Under the same assumptions on characteristic as in Corollary \ref{SLcor}, we have for Jordan degree type, 
\begin{equation}
\begin{aligned}
\mathsf{m}_s\otimes_{\sf k} \mathsf{n}_{t}&=\bigl((\mathsf{n}+\mathsf{m}-1)_{s+t}, (\mathsf{n}+\mathsf{m}-3)_{s+t+1},\ldots, (\mathsf{n}-\mathsf{m}+1)_{s+t+m-1}\bigr)\\
&=\oplus_{k=1}^m (\mathsf{n}+\mathsf{m}+1-2k)_{s+t+k-1}\label{dpdJTeq}
\end{aligned}
\end{equation}
\end{proposition}
\begin{proof} This follows from the Definition \ref{degreeJT-def} of contiguous Jordan degree type of a Hilbert function, and noting that when $s=t=0$ the Jordan degree type of the algebra $A=\mathsf{k}[x,y]/(x^m,y^n)$ is just the contiguous Jordan degree type $P_c(H)$ of the Hilbert function $H=H(A)$, which is the right side of Equation \eqref{dpdJTeq}. For $s,t\not=0$ one just shifts the degrees by $s+t$.
\end{proof}\par
We leave to the reader the statement of the obvious analogue of Proposition \ref{tensorprop} where we replace the $d_i$ and $f_j$ by $(\sf {d_i})_{\alpha_i}$ and $(\sf {f_j})_{\beta_j}$, respectively.  Using this, we may define the tensor product of contiguous partitions (which have the initial degree information for strings) $\P_{c,\deg}\bigl(H(A)\bigr)\otimes \P_{c,\deg}\bigl(H(B)\bigr)$.  The following Corollary is shown using the layering of $\P_{c,\deg}(H)$ when $H$ is unimodal: then the Jordan degree type $\P_{c,\deg}(H)$ is determined by the conjugate partition $H^\vee$. 
\begin{corollary}\label{JDtensorcor} Assume that $A,B$ are standard graded with unimodal Hilbert functions $H(A),H(B)$.\par
i.  Then we have for the contiguous partitions 
\begin{equation}\label{contigprodeq}
\P_{c,\deg}(H(A\otimes_{\sf k} B))=\P_{c,\deg}(H(A))\otimes \P_{c,\deg}(H(B)).
\end{equation}\par
ii. Also, if $\alpha\in A_1$ has Jordan degree-type $\P_{c,\deg}(H(A))$ and $\beta\in B_1$ has that of the contiguous partition $\P_{c,\deg}(H(B))$, then $\ell=\alpha\otimes 1_B+1_A\otimes \beta$ has the Jordan degree type $\P_{c,\deg}(H(A\otimes_{\sf k} B))$.
\end{corollary}
The following example shows that the unimodal condition is necessary for Equation \eqref{contigprodeq}.
\begin{example}
 Letting $H(A)=(1, 13, 12, 13,1)$ and $ H(B)=(1,1)$, we have $P_c\bigl(H(A)\bigr)=$ $(5_0,3_{1}^{11}, 1_1,1_3)$, and $P_c\bigl(H(B)\bigr)=(2_0)$ with product $P_c\bigl(H(A)\bigr)\otimes P_c\bigl(H(B)\bigr)=(6_0,4_1^{12},2_2^{11},2_1,2_3)$. But $H(A)\otimes H(B)=(1,14,25,25,14,1)$, and $P_c\bigl(H(A\otimes_{\sf k} B)\bigr)=(6_0, 4_1^{13},2_2^{11})$.
 
 \end{example}

The next example shows that even with symmetric Hilbert functions, we should not expect tensor products to preserve the relative Lefschetz property of Definition \ref{def:relativeLef} (involving Jordan degree-type and the contiguous Hilbert function) when the grading of $A$ or $B$ is not standard.\vskip 0.2cm
\begin{example}
\label{ex:MLtensor}
Let $\alpha\in \mathsf{k}$ and  $A=\F[a,b]/(a^3-\alpha ab,b^3)$ with weight function $\w(a,b)=(1,2)$, and $B=\F[t]/(t^2)$ with the standard grading.  It is easy to see that both $A$ and $B$ have the relative Lefschetz property for their Hilbert functions, which are, respectively, $H(A)=(1,1,2,1,2,1,1)$ and $H(B)=(1,1)$. On the other hand, the Hilbert function of their tensor product $C=A\otimes_\mathsf{k}B$ is $H(C)=(1,2,3,3,3,3,2,1)$ which is unimodal. If $\cha \mathsf{k}=0$ or $\cha \mathsf{k}>7$ then $C$ cannot have the relative Lefschetz property: if it did, then by Proposition \ref{SLJTprop} $C$ would also be SL which would imply by Corollary \ref{SLcor} (for which we need the assumptions on characteristic of $\sf k$) that both $A$ and $B$ are SL, which is impossible since $H(A)$ is not unimodal. 
\end{example}
\begin{example}[SLJT in a tensor product]\label{2.44ex}
Let $S=\mathsf{k}[e]$ with weight $\mathsf{w}(e)=2$, $S^\vee=\mathfrak{E}=\mathsf{k}[E]$ and $A=S/I_F$, $F=E^2$, so $A=S/(e^3)$ and $H(A)=(1,0,1,0,1)$. Now let $R=\mathsf{k}[x,y]$ with standard grading and take $B=\mathsf{k}[x,y]/(x^2,y^2)$; here $(x^2,y^2)=\Ann(XY)$, $XY\in R^\vee=\mathsf{k}_{DP}[X,Y]$, and $H(B)=(1,2,1)$. Now consider $A\otimes_\mathsf{k} B=\mathsf{k}[x,y,e]/(x^2,y^2,e^3)$
of Hilbert function $H(A)\otimes H(B)=(1,2^5,1)$. Note that $\ell=x+y+e$ has Jordan type $P_\ell=(7,5)$; $\ell $ is not homogeneous but this shows that $A\otimes_{\sf k} B$ has SLJT. 
Since $A\otimes_{\sf k} B$ is not standard-graded, even though $H(A\otimes_{\sf k} B)$ is unimodal and symmetric, having an SLJT element does not imply that $A\otimes_{\sf k} B$ is SL. It is not SL as the only linear elements $x,y,x+y$ (up to scalar) of $A\otimes_{\sf k} B$ are not SL: for example $P_{x+y}=(3^3,1^3)$. \par
If we regrade so that $e$ has degree one, then we have a standard graded CI of generator degrees $(2,2,3)$, Hilbert function $H_{(2,2,3)}=(1,3,4,3,1)$ and Jordan type $P_\ell=(5,3,3,1)$, and it is SL. 
\end{example}
This example suggests
\begin{conjecture}\label{SLJTconj} If $A$ has SLJT, and $B$ is standard graded SL, and if $H(A\otimes_{\sf k} B)=H(A)\otimes H(B)$ (the graded product) then $A\otimes_{\sf k} B$ has SLJT. 
\end{conjecture} 
\subsubsection{Clebsch-Gordan in the modular case.}
\begin{remark}\label{modularTPrem} 
There is substantial work determining Clebsch-Gordan formulas in the modular case $\cha p\le j$. S.P. Glasby, C.E. Praeger, and B. Xia in \cite{GPX1} summarize previous algorithmic results of Kei-ichiro Iima and Ryo Iwamatsu \cite{IiI} using Schur functions, and of J.-C. Renaud \cite{Re}; they obtain formulas that in principle allow one to compute the generic Jordan type of $R(m,n)$ in arbitrary characteristic $p$ -- they term this the Jordan type $\lambda(m,n,p)$ of $R(m,n,p)$, which always has $m$ parts -- so $R(m,n,p)$ is always weak Lefschetz (a result they ascribe to T. Ralley \cite{Ra}). In \cite[Theorem 2]{GPX2} they show that $R(m,n,p)$ has SLP (in their language, ``is standard'') if $n\not\equiv\pm 1, \pm 2, \dots \pm m \mod p$. 
They define a deviation vector $\epsilon (m,n,p)=\lambda(m,n,p)-(n,n,\ldots ,n)$, and show 
\begin{lemma}\cite[Theorems 4,6,7]{GPX2}\label{dualitythm} Let $m\le\min\{ p^k,n,n'\}$.\par
\begin{enumerate}[(i).]
\item (periodicity) If $n\equiv n'\mod p^k$ then $\epsilon (r,n,p)=\epsilon (r,n',p)$.
\item (duality) If $n'=-n\mod p^k$ then $\epsilon (m,n',p)=(-\epsilon_r,\ldots ,-\epsilon_1)$, the ``negative reverse'' of $\epsilon(m,n,p)$.
\item (bound) There are at most $2^{m-1}$ different deviation vectors $\epsilon (m,n,p)$ for all $n\ge m$ and characteristics $p$.
\item (computation) For fixed $m$, a finite computation suffices to compute the values of $\epsilon(m,n,p)$ for all $n$ with $n\ge m$, and all primes $p$.
\end{enumerate}
\end{lemma}
The authors warn that, in contrast, determining $\lambda(m,n,p)$ is not a finite computation as it involves considering $n\mod p$ for infinitely many $n$.\par
Such tables of $\lambda (m,n,p)$ for $m=3$ and $m=4$ have been calculated by Jung-Pil Park and Yong-Su Shin in \cite{PS}. Recent articles by L. Nicklasson and S. Lundqvist with L. Nicklassen further clarify the strong Lefschetz property for monomial complete intersections (see \cite{Ni,LuNi} and references cited there).\par
We can similarly define for sequences $M=(m_1,\ldots,m_r)$, $m_1\le m_2\le \cdots \le m_r$ deviation vectors $\epsilon(M,p)=\epsilon(m_1,\ldots,m_r;p)$ for the Jordan types of CI algebras $R(M)=\mathsf{k}[x_1,\ldots,x_r]/(x_1^{m_1},\ldots,x_r^{m_r})$ when $\cha \mathsf{k}=p$. 
\vskip 0.2cm\noindent
\begin{question} What does the Lemma \ref{dualitythm} tell us about determining modular Jordan types $\lambda(m_1,m_2,m_3,p)$ (three variables) or in more variables? This is a problem that has been studied and appears quite complex. For example in three variables work on it has involved tilings by lozenges \cite{CoNa1,CoNa2}, see also \cite{BK,Co1,Co2,KuVr}. Further studies of the weak Lefschetz properties of monomial ideals and the relation to algebraic-geometric Togliatti systems have been made by E. Mezzetti, G. Ottaviani, R. M.
Mir\'{o}-Roig, and others \cite{MezMO,MezM,MicM}\par 
\end{question}
G. Benkart and J. Osborn studied representations of Lie algebras in characteristic $p$ \cite{BO}. Subsequently, A. Premet \cite[Lemma 3.4,3.5]{Pr}, and J. Carlson, E. Friedlander and J. Pevtsova in \cite[\S 10]{CFP} Appendix ``Decomposition of tensor products of ${\sf}k[t]/(t^p)$ modules'' gave formulas for such products that apply to $R(m,n,p)$ when $m,n\le p$. J. Carlson et al regard $T=\mathsf{k}[t]/(t^p)$ as a self-dual Hopf algebra, with coproduct $t\to 1\otimes t+t\otimes 1$ and determine the tensor product of irreducible modules over $T$.\par
There has been substantial work on rank varieties and connections between commutative algebra and the representation theory of $p$-groups, beyond our scope: see, for example \cite{AvIy,BIKP} and works cited there.
\end{remark}
\subsection{Free extensions.}\label{twoexsec}
We have remarked that the Jordan type of the multiplication map does not depend on the grading of the algebra $\A$ - a remark also of S. Kleiman and J.O. Kleppe in \cite{KlKl}. Nevertheless, the Hilbert function of $\A$ will depend on its grading, and, in particular, whether it is regarded as a graded algebra $A$ or an $\mathfrak{m}_\A$-adic filtered local algebra $\A$: thus the strong Lefschetz property is grading-dependent.\par
We will give some examples where $\A$ is the base of a free extension $C$ with fiber $B$. We first recall the definition of free extensions, introduced by T. Harima and J. Watanabe in \cite{HW0}. Then we give Examples \ref{Gorprodex} and \ref{dualgenex} of algebras $C$ where $C$ is $A$-free, and we compare $A$ with the related local algebra $\A$. In Theorem \ref{thm:CI} we give a method of producing/verifying free extensions related to complete intersections. \par

\begin{definition}
\label{def:freeext} (\cite{HW0,IMM}). 
Given graded Artinian algebras $A$, $B$, and $C$, we say that $C$ is a free extension of $A$ with fiber $B$ if there exist algebra maps $\iota\colon A\rightarrow C$ and $\pi\colon C\rightarrow B$ 
\begin{enumerate}[(i).]
\item $\iota$ makes $C$ into a free $A$-module, 
\item $\pi$ is surjective and $\ker(\pi)=\tau(\mathfrak{m}_A)\cdot C$.
\end{enumerate}
We showed in \cite[Theorem 2.1]{IMM} that a free extension $C$ is a flat deformation of a finite  algebra $B$ over a local Artinian algebra $(\mathcal{A},\mathfrak{m},\mathsf{k})$. Thus, a free extension  $C$ of $A$ with fibre $B$ is an $A$-algebra such that the associated map $\gamma: \Spec(C)\to \Spec(\mathcal{A})$ is finite and flat, and with a closed embedding $\pi: \Spec(B) \to \Spec(C)$ inducing a cartesian diagram:
\[
\xymatrix{{\Spec(B)}\ar[d]\ar@{^{(}->}^\pi[r]&{\Spec(C)}\ar[d]^\gamma\\
{\Spec (\mathsf{k})}\ar@{^{(}->}[r]&{\Spec(\mathcal{A})}.}
\]
\end{definition}
Related to the notion of free extension is that of a \emph{coexact sequence}.\footnote{This notion was introduced by J.C. Moore, and occurs in the topology literature.} Given graded Artinian algebras $A$, $B$, and $C$ we say that a sequence of algebra maps $\xymatrix{A\ar[r]^-\iota & C\ar[r]^-\pi & B}
$
is \emph{coexact at $C$} if $\ker(\pi)=\iota(\mathfrak{m}_A)\cdot C$, and a \emph{coexact sequence} is a sequence of algebra maps 
\begin{equation}
\label{eq:coexact}
\xymatrix{\F\ar[r] & A\ar[r]^-\iota & C\ar[r]^-\pi & B\ar[r] & \F}
\end{equation}
that is coexact in every slot (here $\F$ is the graded Artinian algebra concentrated in degree zero with $\mathfrak{m}_\F=0$). The following result is not difficult, but we record it here for future reference.
\begin{lemma}\label{coexactfreelem}
Let $A$, $B$, $C$ be graded Artinian algebras with maps $\iota\colon A\rightarrow C$ and $\pi\colon C\rightarrow B$ and suppose that $\pi$ is surjective. Then the following are equivalent.
\begin{enumerate}[(i).]
\item For every $\F$-linear section $\mathsf{s}\colon B\rightarrow C$ of $\pi$, the map $A$-module map $\Phi_\mathsf{s}=\iota\otimes\mathsf{s}:_A(A\otimes_\mathsf{k}B)\rightarrow _AC$ is an isomorphism, i.e.\ $C$ is an $A$-module tensor product.
\item The sequence \eqref{eq:coexact}
is coexact and $\iota\colon A\rightarrow C$ is a free extension.
\item $\iota\colon A\rightarrow C$ is a free extension and $\ker(\pi)=\bigl(\iota(\mathfrak{m}_A)\bigr)\cdot C$.
\item $\ker(\pi)=\bigl(\iota(\mathfrak{m}_A)\bigr)\cdot C$ and $\dim_\mathsf{k}C=\dim_\mathsf{k}A\cdot \dim_\mathsf{k}B$.
\end{enumerate}
\end{lemma}
\begin{proof}
(i) $\Rightarrow$ (ii). Assume (i) holds and fix an $\F$-linear $\pi$-section $\mathsf{s}\colon B\rightarrow C$ so that $\Phi_\mathsf{s}\colon A\otimes_\mathsf{k}B\rightarrow C$ is an $A$-module isomorphism. Let $\pi_0\colon A\otimes_\mathsf{k}B\rightarrow \F\otimes_\mathsf{k}B\cong B$ be the natural projection onto the $B$ factor of the tensor product, with $\ker(\pi_0)=\mathfrak{m}_A\otimes_\mathsf{k}B$. Then clearly we have $\pi_0=\pi\circ\Phi_\mathsf{s}$, and hence $\ker(\pi)=\Phi_\mathsf{s}^{-1}(\mathfrak{m}_A\otimes_\mathsf{k}B)=\iota(\mathfrak{m}_A)\cdot C$. 

(ii) $\Rightarrow$ (iii) follows from the definitions.

(iii) $\Rightarrow$ (iv). Assume (iii). Then we have $B\cong C/\ker(\pi)=C/\iota(\mathfrak{m}_A)\cdot C$. By Nakayama's Lemma any $\F$-linear basis for $B\cong C/\iota(\mathfrak{m}_A)\cdot C$ lifts to an $A$-linear basis for $C$, hence we have $\dim_\mathsf{k}C=\dim_{{\F}}A\cdot\dim_{{\F}}B$.

(iv) $\Rightarrow$ (i). Assume (iv) holds. Then $B\cong C/\iota(\mathfrak{m}_A)\cdot C$, and hence Nakayama's lemma implies any $\F$-linear basis for $B$ lifts to an $A$-generating set for $C$. Put another way, for any $\F$-linear $\pi$-section $\mathsf{s}\colon B\rightarrow C$, Nakayama implies that we have an $A$-linear epimorphism $\Phi_\mathsf{s}\colon A\otimes_\mathsf{k}B\rightarrow C$. But since the dimensions of these two vector spaces are equal, $\Phi_\mathsf{s}$ must in fact be an isomorphism.
\end{proof}

Note that the tensor product algebra $C=A\otimes_\mathsf{k}B$ is a free extension over $A$
with fiber $B$ (it is also a free extension over $B$ with fiber $A$). In general a free extension $C$ is isomorphic to the tensor product $A\otimes_\mathsf{k}B$, but only as $A$-modules. A free extension is a deformation of the tensor product algebra \cite[Theorem 2.1]{IMM}. Thus we have
\begin{proposition}[Theorem 2.4, \cite{IMM}]
\label{prop:TPDef}
Given graded Artinian algebras $A$, $B$, and $C$ such that $C$ is a free extension over $A$ with fiber $B$, the generic linear Jordan type of $C$ is at least as large as the generic linear Jordan type of the tensor product algebra $A\otimes_\mathsf{k}B$ with respect to the dominance order, i.e.
\[
P_C\geq P_{A\otimes_\mathsf{k}B}.
\]
\end{proposition}
This can be used to prove the following result of T. Harima and J. Watanabe.
\begin{proposition}\label{SL2prop}\cite[Theorem 6.1]{HW}. Let $C$ be a free extension of $A$ with fiber $B$. Assume that $\cha \mathsf{k}=0$ or $\cha \mathsf{k}\ge j_A+j_B$, that the Hilbert functions of both $A$ and $B$ are symmetric, and that both $A$ and $ B$ are strong Lefschetz. Then $C$ is also strong Lefschetz.
\end{proposition}
We prove the following general result which gives a useful construction and criterion for obtaining $A$-free extensions. If $I=(f_1,\ldots,f_s)$ is an ideal in a ring $S$ and $\tau\colon S\rightarrow R$ is a ring homomorphism, we denote by $\bigl(\tau(I)\bigr)$ the ideal in $R$ generated by $\{\tau(f_1),\ldots,\tau(f_s)\}$. Note that in general, the image $\tau(I)$ is not itself an ideal of $R$ (Remark \ref{conditionrem}).
\begin{theorem}
\label{thm:CI} 
Let $S=\mathsf{k}[e_1,..,e_d]$, $R=\mathsf{k}[x_1,\ldots, x_r]$ be (not necessarily standard) graded polynomial rings, and let $\tau\colon S\rightarrow R$ be a map of $\F$-algebras which makes $R$ into a finite $S$-module. For any ideal $I\subseteq S$ of finite colength, set $A=S/I$ and $C=R/\bigl(\tau(I)\bigr)$, and let $\iota=\bar{\tau}\colon A\rightarrow C$ be the induced map between them. Set $B=R/(\tau(\mathfrak{m}_S))$ where $\mathfrak{m}_S=(e_1,\ldots,e_d)\subset S$ is the homogeneous maximal ideal of $S$, and let $\pi\colon C\rightarrow B$ be the natural projection map.
\begin{enumerate}[(i)]
\item Then $A,B,C$ are all graded Artinian algebras, and we have 
\begin{equation}
\label{eq:Kernel}
\ker(\pi)=\bigl(\iota(\mathfrak{m}_A)\bigr). 
\end{equation}
In particular, we have a coexact sequence of graded Artinian $\F$-algebras 
\[
\xymatrix{\F\ar[r] & A\ar[r]^-{\iota} & C\ar[r]^-{\pi} & B\ar[r] & \F.}
\]

\item Furthermore, if $d=r$, and if $A$ is a complete intersection, then so are $B$ and $C$; also $C$ is a free extension with base $A$ and fiber $B$.
\end{enumerate}

\end{theorem} \begin{proof} 
(i). That $A$, $B$, and $C$ are Artinian follows from our finiteness assumptions on $I$ and the algebra extension $\tau\colon S\rightarrow R$. To see that Equation \eqref{eq:Kernel} holds, note that we have the string of equalities: 
\[
\ker(\pi)=\bigl(\tau(\mathfrak{m}_S)\bigr)/\bigl(\tau(I)\bigr)=\bigl(\iota(\mathfrak{m}_S/I)\bigr)=\bigl(\iota(\mathfrak{m}_A)\bigr). 
\]
(ii). If $d=r$ and $A$ is a complete intersection, then $A=S/I=\F[e_1,\ldots,e_d]/(f_1,\ldots,f_d)$ for some $S$-regular sequence $f_1,\ldots,f_d$. Then since $B=R/\bigl(\tau(\mathfrak{m}_S)\bigr)=\F[x_1,\ldots,x_r]/\bigl(\tau(e_1),\ldots,\tau(e_d)\bigr)$ and $C=R/\bigl(\tau(I)\bigr)=\F[x_1,\ldots,x_r]/\bigl(\tau(f_1),\ldots,\tau(f_d)\bigr)$ are both Artinian (so Krull dimension zero) and since $d=r$, the sequences $\tau(e_1),\ldots,\tau(e_d)$ and $\tau(f_1),\ldots,\tau(f_d)$ must be $R$-regular, hence $B$ and $C$ are complete intersections too. Since $\tau(e_1),\ldots,\tau(e_d)$ is an $R$-regular sequence, they are algebraically independent and $R$ is a free module over $S$. In particular, we have $C=R/\bigl(\tau(I)\bigr)\cong R\otimes_SS/I$ which shows that $C$ is free as an $A=S/I$-module, hence $C$ is a free extension over $A$ with fiber $B$.
\end{proof}\par\noindent
{\bf Note.} The statement (ii) seems to be related to a result of L. Avramov on flat extensions \cite{Av}. See also Proposition 23.8.4 in https://stacks.math.columbia.edu/tag/09PY.
\begin{remark}\label{conditionrem}
It is tempting to think that the hypothesis that $R$ and $S$ have the same Krull dimension in Theorem \ref{thm:CI} (ii) could be replaced by the requirement that $B$ is a complete intersection. But we have the following counter example:
Define $\tau\colon \mathsf{k}[e_1,e_2,e_3]\to \mathsf{k}[x,y]$, $\tau(e_1)=x^2$, $\tau(e_2)=y^2$, $\tau(e_3)=x^2+y^2$ and $A=\mathsf{k}[e_1,e_2,e_3]/(e_1^2,e_2^2,e_3^2-e_1e_2)$. Then $B=\mathsf{k}[x,y]/(x^2,\,y^2,\,x^2+y^2)=\mathsf{k}[x,y]/(x^2,y^2)$ is a complete intersection, but $C=\mathsf{k}[x,y]/\bigl(x^4,\,y^4,\,(x^2+y^2)^2-x^2y^2\bigr)=\mathsf{k}[x,y]/(x^4,y^4,x^2y^2)$ is not. Moreover $C$ has $A$ torsion, e.g.\ $0=\tau(e_3-e_2-e_1)$, hence it cannot be an $A$-free extension. Note that, as is usual, $\tau(I)$ is not an ideal of $R$: here, for example $x(x^2)$ is not in $\tau(I)$; hence our notation $\bigl(\tau(I)\bigr)$ throughout for the ideal in $R$ generated by $\tau(I)$.\vskip 0.2cm\noindent
\end{remark}
A level Artinian algebra $A=R/\Ann F$ is one having its socle $0:\mathfrak{m}$ in a single degree. The level algebra is a \emph{connected sum} over a field $\sf k$ if its dual generator is a sum $F=F_1+F_2$ where $F_1,F_2$ are in two distinct set of variables (see \cite{BBKT} and references cited there). For some discussion of Jordan type in more general connected sums over a Gorenstein algebra $T$ in place of $\sf k$ see \cite{IMS} and a related paper of E. Babson and E. Nevo \cite{BN}. In the next examples $A$ is a connected sum over $\sf k$.
\begin{example}[$C$ is an $A$-free CI]\label{Gorprodex}
Take $S=\mathsf{k}[e_1,e]$ with weights $\mathsf{w}(e_1,e)=(1,4)$, take $ R=\mathsf{k}[x,y]$ with $\mathsf{w}(x,y)=(1,1)$, define $\tau: S\to R$ by $\tau(e_1)=x+y$, $\tau(e)=x^2y^2$. We consider $A=S/I_F$ where $F=E^{[3]}+E_1^{[12]}$, a connected sum,
we let $B=\mathsf{k}[x,y]/(x+y,x^2y^2)$, and take $C=R/\bigl(\tau(I_F)\bigr)$. Here the ideal $I_F$ satisfies
\begin{equation}
I_F=(e_1e,e^3-e_1^{12})\text{ and }S\circ F=A^\vee=\langle F, \{E_1^{[i]}, 0\le i\le 11\},F,E^{[2]}\rangle,
\end{equation}
and the ideal $\bigl(\tau(I_F)\bigr)=\bigl((x+y)x^2y^2,\,(x^2y^2)^3-(x+y)^{12}\bigr) \subset R$.
We have,\footnote{Given sequences of non-negative integers $H(A)=(a_0,\ldots,a_r)$ and $H(B)=(b_0,\ldots,b_s)$ we use the shorthand notation $H(A)\otimes H(B)=(c_0,\ldots,c_{r+s})$ for the convolution sequence $c_j=\sum_{i=0}^ja_ib_{j-i}$.}
\begin{align*}
H(A)&=(1,1,1,1,2,1,1,1,2,1,1,1,1),\, H(B)=(1,1,1,1), \text { and }\\
H(C)&= H(A)\otimes H(B)=(1,2, 3,4,5,5,5,5,5,5,5,5,4,3,2,1).
\end{align*}
Here $H(C)$ is the Hilbert function of the complete intersection $C=\mathsf{k}[x,y]/\bigl(\tau(I_F)\bigr)$ of generator degrees $(5, 12)$: being of codimension two $C$ is SL by Lemma \ref{heighttwolem}; also $B\cong \mathsf{k}[x]/(x^4) $ is SL. That $H(C)=H(A)\otimes H(B)$ implies by Lemma \ref{coexactfreelem} (iv) that $C$ is a free extension of $A$ with fiber $B$. 
We note that $A$ is not strong-Lefschetz, as there is no linear SLJT element: the only linear element is $e_1$ (up to constant multiple), and the partition $P_{e_1,A}=(13,1,1)$, not $H(A)^\vee=(13,2)$. But $\ell=e_1+e$, a non-homogeneous element of $A$, satisfies $P_{\ell,A}=(13,2)$ so $\ell$ has SLJT. The corresponding element $\ell$ in the local ring $\mathcal{A}=\mathsf{k}\{e_1,e\}/(e_1e,e^3-e^{12})$ (a quotient of standard-graded $\mathsf{k}\{e_1,e\}$) also has SLJT (the Jordan type does not change with grading); thus, the local ring $\mathcal{A}$ is strong Lefschetz, of Hilbert function $H(\A)=(1,2,2,1^{10})$.
This example generalizes to $F$ being any quasihomogeneous polynomial in $E_1, E$.
\end{example}
In the following example, the Macaulay dual generators of $A,B$ and $C$ are simply related.
\begin{example}[$C$ is $A$-free]\label{dualgenex}
Let $R=\mathsf{k}[x,y,z]$ with the standard grading, $S=\mathsf{k}[e_1,e_2,e_3]$ with grading $\mathsf{w}(e_1,e_2,e_3)=(2,2,2)$, and $\tau: S\to R$ be the morphism defined by $\tau(e_1)=x^2$, $\tau(e_2)=y^2$, $\tau(e_3)=z^2$. Let $F=E_1^{[3]}+E_2^{[3]}+E_3^{[3]} \in \mathfrak{F}=\Hom_{\sf k}(S,\mathsf{k})$; the algebra $A=S/I_F$ is a connected sum 
\[
A=S/I_F,\quad I_F=\Ann F=(e_1e_2,\,e_1e_3,\,e_2e_3,\,e_1^{\,3}-e_2^{\,3},\,e_1^{\,3}-e_3^{\,3}),
\]
an Artinian Gorenstein non-CI ring of Hilbert function $H(A)=(1,0,3,0,3,0,1)$; it is not SL as $A_1=0$, but it is straightforward to see that for $e=e_1+e_2+e_3$ we have $P_{e,A}=(4,2,2)$ so $e$ has SLJT. Considering $\A$ as the local ring $A$ regraded to standard grading, we have that $\A$ is SL of Hilbert function $(1,3,3,1)$. \par
Recall $\tau: S\to R$: $\tau(e_1)=x^2$, $\tau(e_2)=y^2$, $\tau(e_3)=z^2$, and set
$B=R/(x^2,y^2,z^2)$, a complete intersection of Hilbert function $H(B)=(1,3,3,1)$ and Macaulay dual generator $XYZ\in \mathfrak{D}=\Hom_{\sf k}(R,\mathsf{k})$. Then $P_{x+y+z,B}=(4,2,2)$ so $x+y+z$ is SL. 
\[
C=R/\bigl(\tau(I_F)\bigr), \text { where } \bigl(\tau(I_F)\bigr)=(x^2y^2,\,x^2z^2,\,y^2z^2,\,x^6-y^6,\,x^6-z^6),\]
of Hilbert function $H(C)=(1,3,6,10,12,12,10,6,3,1)=H(A)\otimes H(B)$, so $C$ is $A$-free and is a free extension of $A$ with fiber $B$.\footnote{Here $H(C)$ is the Hilbert function of a CI of generator degrees $(4,4,4)$ but $I_F$ has five generators: is there some importance to this: can we make a similar example where $I_F$ is a CI of this Hilbert function?} We have that $C$ is Gorenstein, as $I_C^\perp=R\circ G_C$ with dual generator $G_C=\left(XYZ\cdot (X^{[6]}+Y^{[6]}+Z^{[6]})\right)$; note that $G_C$ is the product of the dual generator $XYZ$ for $B$ and $ \tau'(F)$ where $F$ is the dual generator for $A$, and the homomorphism $\tau':\mathfrak{F}\to \mathfrak{D}$ corresponds to $\tau: S\to R$.\par
Although $C$ is strong Lefschetz (calculated using \textsc{Macaulay2}), we note that the tensor product $A\otimes_\mathsf{k}B$, unlike $C$, is not standard graded: the only degree-one elements of $A\otimes_\mathsf{k}B$ are of the form $\ell=1\otimes_{\sf k} \ell'$, $\ell'\in \langle x,y,z\rangle$ and since $\ell^3=0$ we have that $P_{\ell,A\otimes_{\sf k} B}\le (4,2,2)^8=(4^8,2^8,2^8)$ rather than the conjugate $H(A\otimes_{\sf k} B)^\vee=H(C)^\vee.$ So $A\otimes_\mathsf{k}B$ is not SL, but its deformation $C$ is SL. \footnote{In \cite[Theorem 2.1]{IMM} the deformation is explicitly given as a one-parameter family $\mathfrak{C}=\{C(t), t\in\mathbb A^1=\sf k\}$. The embedding dimension in a constant length family $\mathfrak{C}$ is semicontinuous, has the value 4 for $A\otimes_\mathsf{k} B$ above (namely $\{x,y,z,e\}$) but three ($\{x,y,z\})$ for $C$; so this must be a jump deformation: for $t\in U$, an open dense of $\mathbb A^1$ not containing $t=0$, the algebra $C(t)$ is standard-graded and also is SL (Corollary~\ref{semicontcor}).} \vskip 0.2cm
A similar example is obtained, replacing $F$ by $F'=E_1^{[3]}+E_2^{[3]}+E_3^{[3]}+E_1E_2E_3$, defining the algebra $A'=S/I'$ where $I'=I_{F'}=(e_1e_2-e_3^{\,2},\,e_1e_3-e_2^{\,2},\,e_2e_3-e_1^{\,3},\, e_1^{\,6}-e_2^{\,6},\,e_1^{\,6}-e_2^{\,6})$, and defining $\tau$ as before. Then $C'=R/\bigl(\tau(I')\bigr)$ is Gorenstein, again an $A'$-free extension with fiber $B$, and with dual generator $G_{C'}=$ $XYZ\cdot \left(X^{[6]}+Y^{[6]}+Z^{[6]}+X^{[2]}Y^{[2]}Z^{[2]}\right)$, again the product of the dual generator for $B$ and $\tau'(F")$.
\end{example}
For further discussion of free extensions or the strong Lefschetz property for complete intersection extensions see \cite{HW0,HW0',HW} and \cite[\S 4.2-4.4]{H-W}, \cite{HW,Pop}, For examples of free extensions related to invariant theory, with some similar behavior to the examples above see \cite{IMM,MCIM}.

\subsection{Commuting Jordan types.}\label{commutingJTsec}
Work of the last ten years has shown that there are strong restrictions on the pairs $P_{\ell,A}, P_{\ell',A} $ that can coexist for an Artinian algebra $A$ \cite{Ob1,Ob2,KOb,IKh,Kh1,Kh2,Pan}. We state several such results. For a more complete discussion, including open questions, see \cite{IKVZ,Ob2}.\par
We say that a partition $P=(p_1,p_2,\ldots,p_s)$ where $ p_1\ge p_2\ge\cdots\ge p_s$ of $n$ is \emph{stable} if its parts differ pairwise by at least two: 
\begin{equation}\label{stableeq}
P \text { is stable if }p_i-p_{i+1}\ge 2, \text { for } 1\le i\le s-1.
\end{equation}
Let $B$ be a nilpotent $n\times n$ matrix over an infinite field $\sf k$ having Jordan type $P=P_B$. Denote by $\mathcal{C}_B$ the commutator of $B$ in $\Mat_n(\sf k)$, and by $\mathcal{N}_B$ the nilpotent elements of $C_B$. It is well known that $\mathcal{N}_B$ is an irreducible variety, hence there is a generic Jordan type $Q(B)$ of matrices in $\mathcal{N}_B$. 
\begin{theorem}[P. Oblak and T.Ko\v{s}ir \cite{KOb}]\label{KOthm}\footnote{This is shown in \cite{KOb} over an algebraically closed field of $\cha k=0$, but their proof carries through for any infinite field of $\cha k=0$ or $\cha k=p>n$. See \cite[Remark 2.7]{IKVZ}.}
Assume $\cha k=0$ or $\cha k=p>n$, let $B$ be an $n\times n$ nilpotent matrix of Jordan type $P$. Then $Q(B)$ is stable and depends only on the Jordan type $P$.
\end{theorem}
Their proof relied on showing that for a general enough matrix $A\in \mathcal{N}_B$, the local ring $\mathcal{A}=\mathsf{k}\{A,B \}$ is Gorenstein (note that, in general, it is non-homogeneous). A result of F.H.S.~Macaulay shows that a Gorenstein (so complete intersection) quotient of $\mathsf{k}\{x,y\}$ has a Hilbert function $H(\mathcal{A})$ whose conjugate $H(\mathcal{A})^\vee$ is stable; and by Lemma \ref{heighttwolem} the generic Jordan type $P_{\mathcal{A}}=H(\mathcal{A})^\vee$. See also \cite[Theorem~2.27]{BaI} for a discussion of these steps and \cite[Section~2.4]{BIK} for a discussion
of the P.~Oblak-T.~Ko\v{s}ir result that $\mathsf{k}[A,B]$ is Gorenstein for $A$ general enough in $\mathcal{N}_B$.
\begin{corollary} There can be at most one stable partition among the partitions $P_{\ell}, \ell \in \mathfrak{m}_{\mathcal{A}}$ for a local Artinian algebra $\mathcal{A}$.
\end{corollary}
For example no two of $\{8, (7,1), (6,2),(5,3)\}$ can occur for partitions $P_\ell$ for the same (commutative) local algebra $\mathcal{A}$. \par
P. Oblak has made a conjecture giving a recursive way to determine $Q(P)$ from $P$:
it is still open in general. The largest and smallest part of $Q(P)$ is known \cite{Ob1,Kh2}, and ``half'' of the conjecture is shown in \cite{IKh}. Showing the other half is equivalent to proving a combinatorial result about a certain poset associated to $\mathcal{N}_B$ \cite{Kh2,IKh}, that would be independent of characteristic.\par
Given an Artinian graded algebra $A$, or a local algebra $\mathcal{A}$ there is a generic Jordan type $P_A$ or $P_{\mathcal{A}}$ (Definition \ref{genJTdef}) by Lemma \ref{genericJTlem}, simply because $\mathfrak{m}_A$ is an affine space, so is irreducible. However,
$P_A$ or $P_{\mathcal{A}}$ need not be stable in the sense of Equation \eqref{stableeq}.
\begin{example} For $A=\mathsf{k}[x,y]/(x^2,xy,y^2)$, or, more generally, for $A_{r,k}=\mathsf{k}[x_1,\ldots ,x_r]/\mathfrak{m}^k $ for $ r,k\ge 2$ the algebra $A$ has generic Jordan type
$P=H(A)^\vee$, which is non-stable. For example $H(A_{2,k}) $ for $k>1$ satisfies $H=(1,2\ldots ,k)$, whose conjugate $H^\vee$ is $(1,2,\ldots ,k)$.\vskip 0.2cm These are examples of algebras having \emph{constant Jordan type} (CJT): the Jordan type $P_{\ell,A}$ is the same for each linear element $\ell\in A$.
Modules of CJT have been extensively studied, and connected to vector bundles over projective space \cite{CFP}.
\end{example}
When the partition $P_\ell$ occurs for a pair $(\ell,M)$ where $M$ is a finite $A$-module and $A$ is Artinian, and $\ell\in A$ is nilpotent, then $P_{\ell^k}$ for a power $\ell^k$ can be simply described in terms of $P_\ell$, and of course must also occur for $M$ or for $A$ (and likewise for $\mathcal{A}$ local and $\ell\in \mathfrak{m}_{\mathcal{A}}$). We briefly describe this. 
\begin{definition}[Almost rectangular partition]\cite{KOb} A partition $P$ of $n$ is \emph{almost rectangular} if its parts
differ pairwise by at most $1$. We denote by $[n]^k$ the unique almost rectangular partition of $n$ having $k$ parts. If $n=qk+r, 0\le r<k$ then 
\begin{equation}
[n]^k=\left((q+1)^r,q^{k-r}\right).
\end{equation}
Given a partiton $P=(p_1,p_2,\ldots,p_s)$, $p_1\ge p_2\ge \cdots \ge p_s$ we denote by 
$[P]^k$ the partition $\left([p_1]^k, [p_2]^k, \ldots, [p_s]^k\right)$ having $ks$ parts.
\end{definition} 
For example $[7]^2=(4,3)$, $[7]^3=(3,2,2)$, $[7]^4=(2,2,2,1)$, $[7]^5=(2,2,1,1,1)$, and
if $P=(7,5)$ then $[P]^2=(4,3,3,2)$.
\par It is easy to see,
\begin{lemma}\label{powerMlem}
Suppose a nilpotent $n\times n$ matrix $M$ is \emph{regular}: has Jordan type $[n]$. Then $M^k$ has Jordan type $[n]^k$. Suppose that the Jordan type of $M$ is $P_M$, then $P_{M^k}=[P_M]^k$.
\end{lemma}
The term ``almost rectangular partition'' was introduced by T. Ko\v{s}ir and P. Oblak; the notion is key in studying $Q(P)$, and occurs earlier in R. Basili's \cite{Ba}, who showed
that the number of parts of $Q(P)$ is $r_P$, the smallest number of almost-rectangular subpartitions needed to cover $P$.\vskip 0.2cm
The question of determining which Jordan types may commute (may coexist in the same Artinian local algebra $A$) is quite open in general. See J. Britnell and M. Wildon's \cite[\S 4]{BrWi} and P. Oblak's \cite{Ob2} for some discussion. The former shows that the problem of when two nilpotent matrices commute is in general characteristic dependent: when $\sf k$ is infinite the orbits $(d,d)$ and $(d+1,d-1)$ are commuting, but when $\sf k$ is a finite field, whether those orbits commute depend on the residue classes of $d$ mod powers of $p$ (\cite[Prop. 4.12, Rem. 4.15]{BrWi}). A result of G.~McNinch \cite[Lemma 22]{McN} shows that the Jordan type of a generic element in the pencil of matrices $M+tN$ can depend on the characteristic. See also \cite[Remark 3.16]{IKh}.\par
There appears to be substantial structure to the set of partitions $P$ having $Q(P)=Q$ where $Q$ is a given stable partition -- see \cite{IKVZ} which shows this structure for stable partitions $Q$ with two parts, and poses a ``box conjecture'' for general stable $Q$.
\subsection{Problems.}\label{prodsec} We end with some further problems concerning Jordan type.

\subsubsection{Compatibility of the partition $P_\ell$ and its refinements with the Hilbert function $H$.}

\begin{question} For which Hilbert functions $H$ can we find graded Artinian algebras $A$
with $H(A)=H$, such that for a generic $\ell \in A_1$ we have, in increasing level of refinement,
\begin{equation}
P_{\ell,A}=P(H), \text { or } P_{c,\ell}=P_c(H), \text { or } P_{c,\deg,\ell}=P_{c,\deg}(H)?
\end{equation}
Note that a graded algebra $A=\mathsf{k}[x,y,z]/I$ of Hilbert function $H(A)=(1,3,3,4)$ cannot be even weak Lefschetz as the minimal growth from degree 2 to degree 3 implies that $I_2=a_1(x,y,z)$ for some $a_1\in A_1$, so multiplication by an $\ell\in A_1$ cannot be injective from $A_1$ to $A_2$.
\end{question}
\begin{remark}In codimension two, the Hilbert function of a graded ideal (actually, also of non-graded ideal) is determined by the partition $P(H)$; here, the family $\Gr (H)$ parametrizing all graded quotients of $R=\mathsf{k}[x,y]$ having Hilbert function $H$ is (smooth and) irreducible. However, in codimension three, even for graded algebras the family $\Gr(H)$ of quotients of $R=\mathsf{k}[x,y,z]$ having Hilbert function $H$ may be reducible. One example is given in \cite[Theorem 2.3]{BoI1}, where $H=(1,3,4,4)$. Here the behavior of one component of $\Gr(H)$ with respect to Jordan type may be different than that of another. Considering the family $\Gor(H)$ of non-graded Gorenstein height three algebra quotients of
$R=\mathsf{k}\{x,y,z\}$ (the regular local ring) of Hilbert function $H$, in \cite{IM}, the first and second authors use the Jordan type and the semicontinuous property of the symmetric decomposition, to show that $\Gor(H)$ has several irreducible components, for suitable $H$, in particular for $H=(1,3,3,2,2,1)$. 
\end{remark}
\vskip 0.2cm
There has been some study of a different question, which Hilbert functions force one of the Lefschetz properties \cite{MiZa1,ZaZy}. See also \cite{MiNa}.

\subsubsection{Additivity of Jordan type.}\label{addsec}
\begin{question}
Let ${0\to L\to M\to N\to 0}$ be an exact sequence of finite $A$-modules 
(or $\mathcal{A}$-modules). How can we compare the generic Jordan types 
$P_L$, $P_M$, and $P_N$? Under what conditions could we have additivity 
${P_M=P_L + P_N}$, in a suitable sense for the Jordan types, or have
${\mathcal{P}_M=\mathcal{P}_L + \mathcal{P}_N}$ for the Jordan degree types (Definition \ref{degreeJT-def})? \par
We can ask the analogous question for non-generic Jordan types (see, for example, \cite{AIK}).
\end{question}
\begin{example}[Additivity]
Let $R=\mathsf{k}[x,y]$, and consider the ideals $I=(x^4,xy,y^3)$ and $J=(x^6,xy,y^6)$.
Let $L=I/J$, $M=R/J$, and $N=R/I$. Then $\ell=x+y$ 
computes the generic Jordan type in each of the modules $L$, $M$, and $N$.
We have for the Jordan degree types,
\begin{align*}
L=I/J&=\langle y^3,y^4,x^4,y^5,x^5\rangle, \text { and } \mathcal{P}_{\ell,L}=\left(3_3,2_4\right);\\
N=R/I&=\langle 1,x,y,x^2,y^2,x^3\rangle, \text { and }\mathcal{P}_{\ell,N}=\left(4_0,2_1\right)\\
M=R/J&=\langle 1,x,y,x^2,y^2,x^3,y^3,x^4,y^4,x^5,y^5\rangle, \text { and } \mathcal{P}_{\ell,M}=\left(6_0,5_1\right).
\end{align*}
Thus, $ \mathcal{P}_M=\left( 6_0=4_0+_c 2_4, 5_1=2_1+_c3_3\right)$ showing that the Jordan degree type $\mathcal{P}_{\ell,M}=\overline{\mathcal{P}_{\ell,N}+P_{\ell,L}}$ in a natural sense. On the other hand we have the Jordan type $P_{\ell,M}=(6,5)\not= P_{\ell,L}+P_{\ell,N}=(4,3,2,2)$ nor is $P_{\ell,M}$ the dominance sum $P_{\ell,L}+_bP_{\ell,N}=(4+3,2+2)=(7,4)$. Each Jordan type $P_{\ell,N},P_{\ell, L}, $ and $P_{\ell,M}$ is SL.
\end{example}
\par\noindent
\begin{question}What partitions $P_{\ell,M}$ and degree-partitions $\mathcal{P}_{\ell,M}$ can we obtain for $M$, fixing those invariants for $L$ and $N$?
\end{question}
\subsubsection{Loci in $\mathbb P(\maxA)$ defined by Jordan type.}
Recall from Definition \ref{posetPMdef} that the set of Jordan types of elements of $A$ acting on $M$ is a poset $\mathfrak{P}_M$ under the ``dominance'' partial order; this poset $\mathfrak{P}_M$ is an invariant of the module $M$. Given a partition $P$ of $m=\dim_\mathsf{k}M$, the locus $\mathfrak{Z}_{P,M}\subset \mathbb P({\mathfrak{m}}_A)$, the projective space of the maximal ideal, parametrizes those elements $\ell\in {\mathfrak{m}}_A$ such that the action of $m_\ell$ on $M$ has Jordan type $P_\ell=P$. The closures $\overline{\mathfrak{Z}_{P,M}}$ form a poset under inclusion. Of course, the actual loci $\mathfrak{Z}_{P,M}$ in either the $A$ graded or $\mathcal{A}$ local case give more information than just the poset. 
\begin{example}\label{2posetex} Let $M=\mathsf{k}[x,y]/(x^2,y^3)$: then $P_x=(2,2,2)$, $P_y=(3,3)$ and $P_\ell=(4,2)$ for $ \ell=x+by$, $b\not=0$. Here $\mathfrak{P}_M=\{(4,2)\ge (3,3)\ge (2,2,2)\}$. However $\mathfrak{Z}_{P,M}=\{ \overline{\mathfrak{Z}_{4,2}}\supset \{\overline{\mathfrak{Z}_{3,3}}\cup \overline{\mathfrak{Z}_{2,2,2}}\}$ as $\mathfrak{Z}_{3,3}$ is a single point. So the the two posets $P_M$ and the poset of closures of loci $\mathfrak{Z}_{P,M}$ may be different, the latter being necessarily a subposet of the former, by the semicontinuity of Jordan type.
\end{example}
There has been some study of these Jordan type loci by commutative algebraists, for example M. Boij, J. Migliore, R.~Mir\'{o}-Roig, and U. Nagel, on the non-weak Lefschetz locus \cite{BMM}, and the notes \cite{DIKS}. On the other hand, there has been recent study of the Jordan type loci $\mathfrak{Z}_P$ in the nilpotent commutator $\mathcal{N}_B$ of an $n\times n$ matrix $B$ \cite{Ob1,Ob2}; when $B$ is a Jordan matrix of stable Jordan type $Q$, then it is conjectured that the set $\mathcal{B}(Q)$ of loci in the nilpotent commutator $\mathcal{N}_B$ can be arranged in a rectangular $r$-box, whose dimensions are determined by the $r$ parts of $Q$ (see \cite[Conjecture 4.11]{IKVZ}), and that the equations for these loci are complete intersections \cite[Remark 4.13]{IKVZ}. The first conjecture is shown for stable $Q$ having $r=2$ parts \cite[Theorem~1.1]{IKVZ}. 
\vskip 0.2cm\noindent
\begin{question}
Let $x\in \maxA$ have Jordan type $Q$. Denote the matrix of $m_x$ by $B$ and recall that $\mathcal{N}_B$ is the nilpotent commutator of $B$ (\S \ref{commutingJTsec}). Is there a morphism $\tau_{x,\mathcal{A}}\colon \maxA\to \mathcal{N}_B$, such that the Jordan type of $y\in \maxA$ satisfies $P_y=P_{\tau_{x,\mathcal{A}}(y)}$? 
\end{question}
Recall that for $x\in \mathcal{A}$ we denote by $\nu(x)$ its order: the maximum $i$ such that $x\in{\mathfrak{m}_\A}^i$.
\begin{question}
Given an Artinan algebra $\A$ with socle degree $j_\A$, by a simple linear algebra argument, we can always find a $\F$-basis $B$ for $\A$ such that for every $i\in\{0,\ldots,j_\A\}$
\begin{equation}
\label{orderHF}
\#\{x\in B\mid \nu(x)=i\}=H(\A)_i.
\end{equation}
Given $\ell\in\m_\A$, is it possible to find a Jordan basis for $m_\ell$ also satisfying \eqref{orderHF}?   
\end{question}

This is clearly possible when $A$ is graded, and $\ell$ is homogeneous.
\subsubsection{Tensor products of local Artinian algebras and Hilbert function.}
\begin{question} Assume that for local Artinian $\sf k$ algebras $\mathcal{A}$ and $\mathcal{B}$ their tensor product $\mathcal{A}\otimes_\mathsf{k}\mathcal{B}$ is also a local Artinian $\sf k$ algebra. By a criterion of M. Sweedler \cite[Theorem, b.iii.]{Sw} this is equivalent to $(\A/\mathfrak{m}_\A)\otimes_{\sf k} (\B/\mathfrak{m}_\B)$ is local. How is the Hilbert function for $\A\otimes_\mathsf{k}\B$ related to the Hilbert functions of $\A$ and of $\B$?  Here we assume the residue fields are equal  $(\A/\mathfrak{m}_\A)= (\B/\mathfrak{m}_\B)\cong \sf k$: then what is the relation between the associated graded algebra  $(\A\otimes_\mathsf{k}\B)^\ast$ and $\A^\ast\otimes_\mathsf{k}\B^{\ast}$, the tensor product of the associated graded algebras of $\A$, $\B$?\par
Jordan degree type may be defined for local Artinian algebras $\A$ by replacing degree of elements $a$ by their order (the highest power of $\mathfrak{m}_\A$ containing $a$) in Definition \ref{degreeJT-def}. In Proposition \ref{tpdJTprop}ff. we study the behavior of Jordan degree type for tensor products of standard graded Artinian algebras. What is the behavior of Jordan degree type for tensor products of local Artinian algebras?
\end{question}

\begin{ack}
We appreciate conversations with Shujian Chen and Ivan Martino. We thank Oana Veliche for conversations and help with \textsc{Macaulay2}. The first author appreciates conversations with participants of the informal work group on Jordan Type at Institute Mittag Leffler in July, 2017 \cite{DIKS}. A subsequent remark of Yong-Su Shin that there was no published reference available in the Lefschetz-related commutative algebra literature about Jordan types was a spark for this work. Discussion during his visit to Northeastern University in January 2018 led to \S\ref{tensorprodsec} about the literature on modular tensor products. We thank Rodrigo Gondim for permission to use several examples. We appreciate comments and answers to our questions by Larry Smith, Junzo Watanabe, Alexandra Seceleanu and Steven L. Kleiman. We thank Lorenzo Robbiano, Gordana Todorov, Shijie Zhu and Jerzy Weyman for their comments. We thank the referee for comments, in particular the referee suggestions led to Proposition~\ref{Wiebelem}.\par
The second author was partially supported by CIMA -- Centro de
Investiga\c{c}\~{a}o em Matem\'{a}tica e Aplica\c{c}\~{o}es,
Universidade de \'{E}vora, project UID/MAT/04674/2019
(Funda\c{c}\~{a}o para a Ci\^{e}ncia e Tecnologia). Part of this work was done while the second author was visiting KU~Leuven, he wishes to thank Wim Veys and the Mathematics Department for their hospitality.
\end{ack}

\end{document}